\documentclass[a4paper]{article}

\usepackage{amsmath}
\usepackage{amsfonts, amssymb, multirow}
\usepackage{amsthm}
\usepackage{enumitem}
\usepackage[T1]{fontenc}
\usepackage{Baskervaldx}
\usepackage{tocloft}
\usepackage{color}
\usepackage[dvipsnames]{xcolor}	
	\definecolor{Blue}{HTML}{3d25b9}
	\definecolor{Red}{HTML}{c00054} 
	\definecolor{Green}{HTML}{358028}
\usepackage[UKenglish]{babel}
\usepackage{csquotes}
\usepackage{booktabs}
\usepackage{tabularx}
\usepackage{microtype}
\usepackage{titling} 
\usepackage{mathrsfs}

\usepackage[margin=25mm]{geometry}
\usepackage{sectsty}
	\sectionfont{\large}
	\subsectionfont{\normalsize}
	\subsubsectionfont{\normalsize}
\usepackage{titlesec}
	\titlespacing{\section}{0pt}{12pt}{0pt}
	\titlespacing{\subsection}{0pt}{6pt}{0pt}
\usepackage[hang,flushmargin,bottom]{footmisc}

\setlist{topsep=0pt,itemsep=0pt}

\newcommand{\dd}{\mathrm{d}}

\newcommand{\Aut}[1]{\mathrm{Aut}(#1)}

\newcommand{\EE}{\mathcal{E}}
\newcommand{\Res}{\mathop{\mathrm{Res}}}

\newcommand{\HH}{DH}
\newcommand{\HHall}{\mathbb{H}}

\usepackage[colorlinks=true,linkcolor=Blue,citecolor=Blue]{hyperref}
	\hypersetup{breaklinks=true}
\usepackage[capitalise,noabbrev]{cleveref}
	\crefname{equation}{equation}{equations}

\newcommand\blfootnote[1]{
	\begingroup
	\renewcommand\thefootnote{}\footnote{#1}
	\addtocounter{footnote}{-1}
	\endgroup
}

\theoremstyle{plain}
	\newtheorem{theorem}{Theorem}
	\newtheorem{proposition}[theorem]{Proposition}
	\newtheorem{corollary}[theorem]{Corollary}
	\newtheorem{lemma}[theorem]{Lemma}
	
	\numberwithin{theorem}{section}
\theoremstyle{definition}
	\newtheorem{definition}[theorem]{Definition}
	
\theoremstyle{remark}
	\newtheorem {remark }{Remark}
	\newtheorem{remark}[theorem]{Remark}

\begin{document}

{\large \bfseries Double Hurwitz numbers: polynomiality, topological recursion and intersection theory}

{\bfseries Ga\"{e}tan Borot\footnote{\label{A1}Max Planck Institut f\"ur Mathematik, Vivatsgasse 7, 53111 Bonn, Germany.}, Norman Do\footnote{\label{A2}School of Mathematics, Monash University, VIC 3800, Australia.}, Maksim Karev\footnote{St. Petersburg Department of the Steklov Mathematical Institute, Fontanka 27, St Petersburg 191023, Russia.}, Danilo Lewa\'{n}ski\textsuperscript{\normalfont\ref{A1},}\footnote{Universit\'e Paris-Saclay, CNRS, CEA, Institut de physique th\'eorique (IPhT), 91191 Gif-sur-Yvette, France.}\textsuperscript{,}\footnote{Institut des Hautes \'Etudes Scientifiques, le Bois-Marie, 35 route de Chartres, 91440 Bures-sur-Yvette, France.}, Ellena Moskovsky\textsuperscript{\normalfont\ref{A2}}}

\emph{Abstract.} Double Hurwitz numbers enumerate branched covers of $\mathbb{CP}^1$ with prescribed ramification over two points and simple ramification elsewhere. In contrast to the single case, their underlying geometry is not well understood. In previous work by the second- and third-named authors, the double Hurwitz numbers were conjectured to satisfy a polynomiality structure and to be governed by the topological recursion, analogous to existing results concerning single Hurwitz numbers. In this paper, we resolve these conjectures by a careful analysis of the semi-infinite wedge representation for double Hurwitz numbers. We prove an ELSV-like formula for double Hurwitz numbers, by deforming the Johnson--Pandharipande--Tseng formula for orbifold Hurwitz numbers and using properties of the topological recursion under variation of spectral curves. In the course of this analysis, we unveil certain vanishing properties of the Chiodo classes.
 
\blfootnote{\par\vspace{-10pt} 
} 


\hrule

~
 
\tableofcontents

~

\hrule

~

\section{Introduction} \label{sec:intro}

\subsection{Review of Hurwitz numbers}
\setcounter{footnote}{0}
Hurwitz theory is concerned with enumerations of branched covers of complex curves with specified ramification. Despite the fact that such problems have been studied since the late nineteenth century~\cite{hur1891}, the field has seen an explosion of activity over the past two decades. Much of this has been motivated by the rich tapestry of mathematical structures found in the study of single Hurwitz numbers, which can be treated as a paradigm for other problems in the enumerative geometry of curves.

 The {\em single Hurwitz number} $H_{g,n}(\mu_1, \ldots, \mu_n)$ is the weighted enumeration of connected genus $g$ branched covers of smooth complex curves $f: (\Sigma; p_1, \ldots, p_n) \to (\mathbb{CP}^1; \infty)$ such that\footnote{The notation here indicates that $p_1, \ldots, p_n$ are points on $\Sigma$ that map to $\infty \in \mathbb{CP}^1$ and $[p]$ is the divisor supported on $p$.}
\begin{itemize}
\item $f^{-1}([\infty]) = \mu_1[p_1] + \cdots + \mu_n [p_n]$; and
\item all other branch points are simple and occur over prescribed points of $\mathbb{CP}^1$.
\end{itemize}
The weight of such a branched cover is
\[
\frac{1}{m!} \cdot \frac{1}{|\Aut{f}|},
\]
where $m = 2g - 2 + n + \sum_{i = 1}^n \mu_i$ is the number of simple branch points, as determined by the Riemann--Hurwitz formula, and $\Aut{f}$ is the group of automorphisms $\phi: (\Sigma; p_1, \ldots, p_n) \to (\Sigma; p_1, \ldots, p_n)$ such that $f \circ \phi = f$. In this convention, the preimages of $\infty$ are labelled while the simple branch points are unlabelled. It is worth remarking that the normalisation adopted here may not agree with others in the literature, but lends itself well to the generalisation to double Hurwitz numbers that we later consider.

It was observed by Goulden, Jackson and Vainshtein that the single Hurwitz numbers satisfy the following structural property~\cite{gou-jac-vai00}.
\begin{enumerate}
\item[(I)] {\em Polynomiality.} For $2g-2+n > 0$, there exists a symmetric polynomial $P_{g,n}$ of degree $3g-3+n$ such that
\[
H_{g,n}(\mu_1, \ldots, \mu_n) = \prod_{i=1}^n \frac{\mu_i^{\mu_i}}{\mu_i!} P_{g,n}(\mu_1, \ldots, \mu_n).
\]
\end{enumerate}

The first proof of this result required the following seminal result of Ekedahl, Lando, Shapiro and Vainshtein, relating single Hurwitz numbers to intersection theory on the moduli spaces of curves~\cite{eke-lan-sha-vai01}.
\begin{enumerate}
\item[(II)] {\em ELSV formula.} For $2g-2+n > 0$, the single Hurwitz numbers satisfy
\begin{equation} \label{eq:ELSV}
H_{g,n}(\mu_1, \ldots, \mu_n) = \prod_{i=1}^n \frac{\mu_i^{\mu_i}}{\mu_i!} \int_{\overline{\mathcal M}_{g,n}} \frac{c(\Lambda^{\vee})}{\prod_{i=1}^n (1 - \mu_i \psi_i)},
\end{equation}
where $\Lambda^{\vee} \to \overline{\mathcal M}_{g,n}$ is the dual Hodge bundle.
\end{enumerate}

Subsequently, Bouchard and Mari\~{n}o conjectured~\cite{bou-mar08} that single Hurwitz numbers are governed by the topological recursion of Chekhov, Eynard and Orantin~\cite{che-eyn06,EORev}; the first complete proof of this conjecture was given by Eynard, Mulase and Safnuk~\cite{eyn-mul-saf11}. The conjecture arose via the analysis of the large framing limit in the $\mathbb{C}^3$ case of the remodelling conjecture of~\cite{bou-kle-mar-pas09} proved in~\cite{eyn-ora15,fan-liu-zon19}, which relates Gromov--Witten invariants of toric Calabi--Yau threefolds to the topological recursion.

\begin{enumerate}
\item[(III)] {\em Topological recursion.} When applied to the spectral curve $(\mathbb{C}^*,x,y,\omega_{0,2})$ with
\[ 
x(z) = \ln z - z,\qquad y(z) = z,\qquad \omega_{0,2}(z_1, z_2) = \frac{\mathrm{d}z_1\otimes \mathrm{d}z_2}{(z_1-z_2)^2},
\]
the topological recursion produces correlation differentials $\omega_{g,n}$ indexed by $g \geq 0$ and $n \geq 1$, which have the following all-order series expansion when $z_i \rightarrow 0$, which we indicate by the symbol $\sim$ as follows.
\[
\omega_{g,n}(z_1, \ldots, z_n) \sim \sum_{\mu_1, \ldots, \mu_n \geq 1} H_{g,n}(\mu_1, \ldots, \mu_n) \bigotimes_{i=1}^n \mathrm{d}(e^{\mu_ix(z_i)}).
\]
\end{enumerate}

The generalisation and analogue of the results (I)-(II)-(III) above to other Hurwitz problems has been a major avenue of research over the last decade. As of the time of writing, this program has been carried out in the following cases.
\begin{itemize}
\item[$\bullet$] \textit{Weakly monotone}: polynomiality in~\cite{Novakpoly0,Novakpoly}, topological recursion in~\cite{do-dye-mat17, kra-pop-sha19}, ELSV-like formula in~\cite{ale-lew-sha16}.
\item[$\bullet$] \textit{Strictly monotone}: polynomiality in~\cite{nor13,dum-mul-saf-sor13,kra-lew-sha16}, ELSV-like formula still unknown, topological recursion in~\cite{E1MM,nor13,dum-mul-saf-sor13,Ebook}.
\item[$\bullet$] \textit{Orbifold}: polynomiality and ELSV-like formula in~\cite{joh-pan-tse11, dun-lew-pop-sha15}, topological recursion in~\cite{do-lei-nor16,bou-her-liu-mul14, dun-lew-pop-sha15}.
\item[$\bullet$] \textit{Orbifold weakly monotone}: polynomiality in~\cite{kra-lew-sha16}, ELSV-like formula still unknown, topological recursion conjectured in~\cite{do-kar-orbmon} and proved in~\cite{kra-pop-sha19}.
\item[$\bullet$] \textit{Orbifold strictly monotone}: polynomiality in~\cite{kra-lew-sha16}, ELSV-like formula partially deduced in~\cite{BGFsimple}, topological recursion proved in~\cite{dun-ora-pop-sha14}.
\item[$\bullet$] \textit{Orbifold spin}: polynomiality in~\cite{kra-lew-pop-sha19}, topological recursion and ELSV-like formula conjectured in~\cite{Zvo06,SSZ15,kra-lew-pop-sha19}, proved to be equivalent in~\cite{SSZ15,kra-lew-pop-sha19}, and finally proved in~\cite{bor-kra-lew-pop-sha20,dun-kra-pop-sha19-2}.

\end{itemize}
The present work gives analogues of the results (I)-(II)-(III) for double Hurwitz numbers, which resolves conjectures put forward in the previous work of the second- and third-named authors~\cite{do-kar18}.

\subsection{Double Hurwitz numbers}
\label{sec:dhdef}
Double Hurwitz numbers enumerate branched covers of $\mathbb{CP}^1$ with specified ramification profiles over two points --- fixed to be $0$ and $\infty$ --- and simple branching elsewhere. The prescribed data then comprises an integer corresponding to the genus and two partitions corresponding to the ramification profiles over $0$ and $\infty$. This enumeration problem has been previously studied, and a wealth of structures ubiquitous in enumerative geometry have been unveiled: Toda integrable hierarchy~\cite{oko00}, tropical geometry description~\cite{TropH}, relation to lattice point enumeration, piecewise polynomiality and wall-crossing~\cite{gou-jac-vak05,Shadrin2,CJM11}. Connections to intersection theory on certain moduli spaces and the geometry of the double ramification cycle have also been proposed, but are much less understood at present~\cite{gou-jac-vak05,Cavalierireview}.

For our purposes, it is advantageous to collect double Hurwitz numbers according to their genus and ramification profile over $\infty$ only, while recording the ramification profile over $0$ via a set of weights. Furthermore, we only consider branched covers whose ramification orders over $0 \in \mathbb{CP}^1$ are bounded by a fixed positive integer $d$. By this assumption, we avoid issues of convergence, particularly since we require analytic data when dealing with the topological recursion. Thus, let us fix here and throughout the paper a positive integer $d$, a parameter $s \in \mathbb{C}$ that keeps track of the number of simple branch points, and a set of weights $q_1, \ldots, q_d \in \mathbb{C}$. We store these weights as coefficients of the polynomial
\[
Q(z) := \sum_{j = 1}^{d} q_{j} z^{j}.
\]
We assume that $q_ds \neq 0$ throughout.

\begin{definition}
\label{doubledef} The {\em double Hurwitz number} $\HH_{g,n}(\mu_1, \ldots, \mu_n)$ is the weighted enumeration of connected genus $g$ branched covers $f: (\Sigma; p_1, \ldots, p_n) \to (\mathbb{CP}^1; \infty)$ such that
\begin{itemize}
\item $f^{-1}([\infty]) = \mu_1[p_1] + \cdots + \mu_n [p_n]$;
\item the order of any ramification point above $0 \in \mathbb{CP}^1$ is at most $d$; and
\item all other branch points are simple and occur at prescribed points of $\mathbb{CP}^1$.
\end{itemize}
The weight of such a branched cover is
\[
 \frac{s^m}{m!} \cdot \frac{\vec{q}_{\lambda}}{|\Aut{f}|},
\]
where the ramification profile over $0 \in \mathbb{CP}^1$ is given by the partition $\lambda = (\lambda_1, \ldots, \lambda_\ell)$ of size $\sum_{i = 1}^n \mu_i$, $m = 2g - 2 + n + \ell$ is the number of simple branch points as determined by the Riemann--Hurwitz formula, and
\[
\vec{q}_{\lambda} = q_{\lambda_1}\cdots q_{\lambda_{\ell(\lambda)}}.
\]
\end{definition}

\begin{remark}
We recover the so-called $d$-orbifold Hurwitz numbers by setting $q_1 = \cdots = q_{d - 1} = 0$ and $q_d = s = 1$, while $d = 1$ corresponds to the case of single Hurwitz numbers. Observe that the parameter $s$ is redundant in the sense that the dependence of $DH_{g,n}(\mu_1, \ldots, \mu_n)$ on $s$ can be recovered from the case $s = 1$.
\end{remark}

Alternatively, Hurwitz numbers can be approached via representation-theoretic methods. Every representation $\rho: \pi_1(\mathbb{CP}^1 \setminus B) \to \mathfrak{S}_N$ gives rise to a unique possibly disconnected degree $N$ branched cover of $\mathbb{CP}^1$ with branching only over the set $B$. This correspondence between representations and covers is surjective due to the Riemann existence theorem, and two monodromy representations correspond to the same cover if and only if they differ by a conjugation. Moreover, the subgroup of $\mathfrak{S}_N$ that fixes a monodromy representation is isomorphic to the group of automorphisms of the corresponding cover. Monodromy representations in turn correspond to factorisations in $\mathfrak{S}_N$, whose enumeration may be expressed in terms of symmetric group characters. This yields formulas for double Hurwitz numbers as inner products in the semi-infinite wedge space~\cite{oko00}, which we review in~\Cref{sec:DHviainfinitewedge} and constitute the starting point of our study.

\subsection{Main results on double Hurwitz numbers}

Our main results are the following analogues for double Hurwitz numbers of the polynomiality structure, the topological recursion and the ELSV formula reviewed above for single Hurwitz numbers, which we obtain in this logical order.

\begin{theorem}[Polynomiality, Conjecture 24 in~\cite{do-kar18}] \label{thm:poly}
For $2g-2+n > 0$, there exist
\[
C_{g,n} \big(\begin{smallmatrix} j_1 & \cdots & j_n \\ m_1 & \cdots & m_n \end{smallmatrix}\big) \in \mathbb{C}(q_1, \ldots, q_d,s),
\]
which vanish for all but finitely many values of $(j_i,m_i)_{i = 1}^n$, such that
\[
\HH_{g,n}(\mu_1, \ldots, \mu_n) = \sum_{\substack{1 \leq j_1, \ldots, j_n \leq d \\ m_1, \ldots, m_n \geq 0}} C_{g,n} \big(\begin{smallmatrix} j_1 & \cdots & j_n \\ m_1 & \cdots & m_n \end{smallmatrix}\big)\prod_{i=1}^n A_{\mu_i,j_i}\,\mu_i^{m_i},
\]
where
\[
A_{\mu,j} := j \sum_{\lambda \vdash \mu-j} \frac{(\mu s)^{\ell(\lambda)}}{|\Aut{\lambda}|} \, \vec{q}_{\lambda} = j\,[z^{\mu - j}] \,e^{\mu s Q(z)}.
\]
Here, the summation is over partitions $\lambda$ of size $\mu - j$. We use $\ell(\lambda)$ to denote the number of parts of $\lambda$ and $\Aut{\lambda}$ to denote the set of permutations leaving the tuple $(\lambda_1, \lambda_2, \ldots, \lambda_{\ell(\lambda)})$ invariant.
\end{theorem}

This property does not seem to be implied by the piecewise polynomiality of double Hurwitz numbers known by~\cite{gou-jac-vak05,CJM11}: the latter states that for fixed $g,n,\ell$, the coefficient of $q_{\lambda_1}\cdots q_{\lambda_\ell}$ in $DH_{g,n}(\mu_1,\ldots,\mu_n)$ is a piecewise polynomial in $\lambda_1,\ldots,\lambda_\ell,\mu_1,\ldots,\mu_n$, while for us $\ell = \ell(\lambda)$ is not fixed and can be arbitrarily large.

\begin{theorem}[Topological recursion] \label{thm:TR}
When applied to the spectral curve $(\mathbb{C}^*,x,y,\omega_{0,2})$ with
\begin{equation}
\label{spdouble} x(z) = \ln z - sQ(z),\qquad y(z) = Q(z),\qquad \omega_{0,2}(z_1, z_2) = \frac{\mathrm{d}z_1\otimes\mathrm{d}z_2}{(z_1-z_2)^2},
\end{equation}
the topological recursion produces correlation differentials $\omega_{g,n}$ indexed by $g \geq 0$ and $n \geq 1$, which have the following all-order series expansion when $z_i \rightarrow 0$:
\begin{equation}
\label{omefgngun}\omega_{g,n}(z_1, \ldots, z_n) - \delta_{g,0}\delta_{n,2}\,\frac{{\rm d}x(z_1)\otimes {\rm d}x(z_2)}{(x(z_1) - x(z_2))^2} \sim \sum_{\mu_1, \ldots, \mu_n \geq 1} \HH_{g,n}(\mu_1, \ldots, \mu_n) \bigotimes_{i=1}^n {\rm d}(e^{\mu_i x(z_i)}).
\end{equation}
\end{theorem}

The definition of the topological recursion will be recalled in~\Cref{subsec:TR}. \Cref{thm:TR} is a generalisation of the $d$-orbifold case~\cite{do-lei-nor16,bou-her-liu-mul14}, which is recovered by setting $Q(z) = z^d$. The $(g,n) = (0,1)$ and $(0,2)$ cases were previously known --- see~\cite[Propositions 8 and 10]{do-kar18}. We give in~\Cref{app02} a new proof of the $(0,2)$ case, relying only on the semi-infinite wedge expression for double Hurwitz numbers.

We take Theorem~\ref{thm:TR} together with the $d$-orbifold ELSV formula of \cite{joh-pan-tse11,lew-pop-sha-zvo17} as the starting point to obtain an ELSV-like formula for double Hurwitz numbers. Both of these formulas involve the Chiodo classes $\Omega^{[d]}_{g;a_1,\ldots,a_k} \in H^*(\overline{\mathcal{M}}_{g,n}; \mathbb{Q})$ indexed by $a_1,\ldots,a_k \in \{0, 1, \ldots,d - 1\}$, coming from the moduli space of $d$-spin curves, whose definition and properties are briefly reviewed in~\Cref{Chiodorev}.\footnote{We could have instead treated the arguments of $\Omega^{[d]}_{g;a_1,\ldots,a_k}$ as elements of $\mathbb Z/d\mathbb Z$. However, one would then lose the direct connection with Chiodo classes, for which periodicity modulo $d$ does not hold in general.}

\begin{theorem}[ELSV-like formula] \label{thm:ELSV}
For $d \geq 2$, $2g - 2 + n > 0$ and $\mu_1, \ldots, \mu_n > 0$, we have
\begin{equation}
\label{DHDHDH0}\begin{split} 
DH_{g,n}(\mu_1,\ldots,\mu_n) & = \sum_{\lambda \vdash |\mu|} (ds)^{2g - 2 + n + \ell(\lambda)}\,\vec{q}_{\lambda} \prod_{i = 1}^n \frac{(\mu_i/d)^{\lfloor \mu_i/d \rfloor}}{\lfloor \mu_i/d \rfloor !} \\ 
& \quad \times \Bigg(\sum_{k = 0}^{\ell(\lambda')} \frac{(-1)^{\ell(\lambda') - k}}{k!}
 \sum_{
 \substack{\boldsymbol{\rho} \in (\mathscr{\tilde{P}}_{d - 1})^{k} \\ \sqcup_{\kappa} \rho^{(\kappa)} = \check{\lambda'}}
 }
\prod_{\kappa=1}^k \frac{\big[\frac{d - |\rho^{(\kappa)}|}{d}\big]_{\ell(\rho^{(\kappa)}) - 1}}{|{\rm Aut}(\rho^{(\kappa)})|} 
\int_{\overline{\mathcal{M}}_{g,n + k}} \frac{\Omega^{[d]}_{g;-\overline{\mu},d - |\boldsymbol{\rho}|}}{\prod_{i = 1}^n \big(1 - \frac{\mu_i}{d}\psi_i\big)} \Bigg), 
\end{split}  
\end{equation} 
where we adopt the following notation.
\begin{itemize}
\item The outermost sum ranges over partitions $\lambda$ of size $|\mu| = \sum_{i = 1}^n \mu_i$, whose parts are at most $d$. This latter condition is in agreement with $q_{j} = 0$ for $j > d$. 
\item For a partition $\lambda$ with parts $\lambda_1, \ldots, \lambda_\ell$, define the (possibly empty) partition $\lambda'$ by removing from $\lambda$ all parts equal to $d$, and define the partition $\check{\lambda'}$ with parts $d-\lambda_1, \ldots, d-\lambda_\ell$, where we remove any parts equal to 0.
\item The set $\tilde{\mathscr{P}}_{d - 1}$ consists of all partitions $\rho$ whose size $|\rho|$ is at most $d - 1$. 
\item We use the notation $[x]_{\ell} = x(x + 1)\cdots (x + \ell - 1)$.
\item The innermost sum ranges over $k$-tuples $\boldsymbol{\rho} = (\rho^{(1)},\ldots,\rho^{(k)})$ with $\rho^{(\kappa)} \in \tilde{\mathscr{P}}_{d - 1}$ and whose concatenation $\sqcup_{\kappa} \rho^{(\kappa)}$ is $\check{\lambda'}$.
\item For $\boldsymbol{\rho} \in (\mathscr{\tilde{P}}_{d - 1})^{k}$, denote by $d - |\boldsymbol{\rho}|$ the $k$-tuple $(d - |\rho^{(1)}|, \ldots, d - |\rho^{(k)}|)$.
\item We use the shorthand $-\overline{\mu} = (-\overline{\mu}_1,\ldots,-\overline{\mu}_n)$, where $-\overline{\mu}_i$ is the unique integer in $\{0, 1, \ldots,d - 1\}$ such that $-\mu_i = -\overline{\mu}_i \pmod{d}$.
\item By convention, if $\lambda' = \emptyset$ the sum over $k$ reduces to a single term corresponding to $k = 0$, which is simply equal to the integral over $\overline{\mathcal M}_{g,n}$; on the other hand, for $\lambda' \neq \emptyset$ this sum ranges over $k \in \{1, 2, \ldots,\ell(\lambda')\}$.
\end{itemize}
\end{theorem}

The term $k = \ell(\lambda')$ should be considered as the first term in the sum, and is equal to
\begin{equation} \label{ksumsimp}
\frac{1}{|\Aut{\lambda'}|} \int_{\overline{\mathcal{M}}_{g,n + \ell(\lambda')}} \frac{\Omega^{[d]}_{g;-\overline{\mu}, \lambda'}}{\prod_{i = 1}^n \big(1 - \frac{\mu_i}{d}\psi_i\big)}.
\end{equation}

Johnson, Pandharipande and Tseng have computed $[s^{2g - 2 + n + \ell(\lambda)}\vec{q}_{\lambda}]\,DH_{g,n}(\mu_1, \ldots, \mu_n)$ in \cite{joh-pan-tse11} under the ``boundedness condition''
\begin{equation} \label{boundeddef}
\max_{i \neq j} {(\lambda'_i + \lambda'_j)} \leq d,
\end{equation}
and their formula agrees with \eqref{ksumsimp} after pushing forward to $\overline{\mathcal{M}}_{g,n + \ell(\lambda')}$ using \cite[Section 5]{lew-pop-sha-zvo17}. Indeed, for us boundedness imposes that $\ell(\rho^{(\kappa)}) = 1$ since $|\rho^{(\kappa)}| \leq d - 1$, so only \eqref{ksumsimp} contributes to the sum over $k$. When boundedness is not satisfied, we find a linear combination of Chiodo integrals instead of a single Chiodo integral. Theorem~\ref{thm:ELSV} can therefore be seen as an unconditional generalisation of the results of \cite{joh-pan-tse11}.

Observe that for $b \in \{1,\ldots,d - 1\}^{k}$, the coefficient of the integral of $\Omega^{[d]}_{g;-\overline{\mu},b}$ in the parenthesised expression of \eqref{DHDHDH0} is
\[
\frac{(-1)^{\ell(\lambda') - k}}{k!} 
\sum_{
\substack{
\boldsymbol{\rho} \in (\mathscr{\tilde{P}}_{d - 1})^{k} \\ \sqcup_{\kappa} \rho^{(\kappa)} = \check{\lambda'}
\\
 |\boldsymbol{\rho}| = d - b
 }
} 
\prod_{\kappa=1}^k \frac{\big[\frac{b_\kappa}{d}\big]_{\ell(\rho^{(\kappa)}) - 1}}{|\Aut{\rho^{(\kappa)}}|}. 
\]
It would be interesting to compute it in closed form.

\subsection{Consequences for Chiodo classes}

In the course of our analysis, we in fact find terms on the right-hand side of \eqref{DHDHDH0}, which have negative powers of $q_d$. As double Hurwitz numbers are polynomials in $q_d$, this leads to the following vanishing properties of Chiodo classes.

\begin{theorem}[Vanishing] \label{thm:vanishC}
Let $g \geq 0$ and $n,\ell \geq 1$. For any integers $d \geq 2$, any partitions $\mu = (\mu_1,\ldots,\mu_n)$ and $\eta = (\eta_1,\ldots,\eta_{\ell})$ such that $\eta_1 \leq d - 1$, we have
\[
|\mu| + |\eta| < d\ell \quad \Longrightarrow \quad \sum_{k = 1}^{\ell} \frac{(-1)^{\ell - k}}{k!} 
\!\!\!
\sum_{
\substack{
\boldsymbol{\rho} \in (\mathscr{\tilde{P}}_{d - 1})^{k} \\ \sqcup_{\kappa} \rho^{(\kappa)} = \eta
}
}
\prod_{\kappa = 1}^{k} \frac{\big[\frac{d - |\rho^{(\kappa)}|}{d}\big]_{\ell(\rho^{(\kappa)}) - 1}}{|\Aut{\rho^{(\kappa)}}|} \int_{\overline{\mathcal{M}}_{g,n + k}} \frac{\Omega^{[d]}_{g;-\overline{\mu},d - |\boldsymbol{\rho}|}}{\prod_{i = 1}^n \big(1 - \frac{\mu_i}{d}\psi_i\big)} = 0.
\]
Here, the sum ranges over tuples $\boldsymbol{\rho} = (\rho^{(1)},\ldots,\rho^{(k)})$ of partitions whose concatenation is $\eta$ and such that $|\rho^{(\kappa)}| \leq d - 1$ for any $\kappa \in \{1,\ldots,k\}$.
\end{theorem}

It is remarkable that the vanishing relation is independent of the genus $g$. Examples of this formula are described in Appendix~\ref{appC}.

The assumption on $\lambda$ in this result is equivalent to the negativity condition in the work of Johnson, Pandharipande and Tseng~\cite{joh-pan-tse11}. When $\lambda$ also satisfies the boundedness assumption, our result specialises to the vanishing found in \cite{joh-pan-tse11}. To our knowledge, all other cases are new. It is possible that they may be deduced via the algebro-geometric methods of \cite{joh-pan-tse11}, but that would require a more systematic study of the combinatorics of degenerations of orbifold covers of $\mathbb{P}^1$. Our proofs of Theorem~\ref{thm:vanishC} and the closely related Theorem~\ref{thm:ELSV} ignore such geometric considerations and the intricate combinatorics in fact results from Taylor expansions of composed functions, starting from the known ELSV formula for orbifold Hurwitz numbers.

This vanishing can be interpreted by saying that for given $\mu$, a partition $\eta$ is {\em stable} if $|\mu| + |\eta| = dK$ for $K \geq \ell(\eta)$, and {\em unstable} if $K \leq \ell(\eta) - 1$. Using Theorem~\ref{thm:vanishC} iteratively, one can compute any unstable Chiodo integral
\begin{equation}
\label{vanishD}\int_{\overline{\mathcal{M}}_{g,n + \ell(\eta)}} \frac{\Omega^{[d]}_{g;-\overline{\mu},d - \eta}}{\prod_{i = 1}^n \big(1 - \frac{\mu_i}{d}\psi_i\big)} 
\end{equation}
as a linear combination of stable Chiodo integrals. The partitions $\tilde{\eta}$ involved in these integrals are the ones obtained from $\eta$ by merging enough rows so that they satisfy the stability condition $K \geq \ell(\tilde{\eta})$ while keeping $\tilde{\eta}_j \leq d - 1$ for all $j$. More precisely, if $\eta$ is unstable, a single application of Theorem~\ref{thm:vanishC} expresses \eqref{vanishD} as sum of over $k \in \{1,\ldots,\ell(\eta) - 1\}$ of Chiodo integrals over $\overline{\mathcal{M}}_{g,n + k}$. The terms $k > K$ are unstable so one can repeatedly use Theorem~\ref{thm:vanishC} to equate \eqref{vanishD} with a sum
\begin{equation}
\label{equate} \sum_{k = 1}^{K} \frac{1}{k!} \sum_{\substack{\boldsymbol{\rho} \in (\tilde{\mathscr{P}}_{d - 1})^{k} \\ \sqcup_{\kappa} \rho^{(\kappa)} = \eta}} B_{\boldsymbol{\rho}} \int_{\overline{\mathcal{M}}_{g,n + k}} \frac{\Omega^{[d]}_{g;-\overline{\mu},d - |\boldsymbol{\rho}|}}{\prod_{i = 1}^n \big(1 - \frac{\mu_i}{d}\psi_i\big)}
\end{equation}
for some rational coefficients $B_{\boldsymbol{\rho}}$ independent of $g$ --- see \eqref{eq:59} for an example.

For a given $\lambda,\mu \in \mathscr{P}$ such that $|\lambda| = |\mu|$, ~\cref{thm:ELSV} allows us to express the double Hurwitz number enumerating genus $g$ branched covers of $\mathbb{CP}^1$ with ramification profile $\mu$ over $\infty$ and $\lambda$ over $0$ via Chiodo integrals, whenever $d > \lambda_1$. As these double Hurwitz numbers do not depend on the choice of $d$, we deduce identities between $\Omega^{[d']}$ and $\Omega^{[d]}$ for $d' \geq d \geq 2$.

\begin{corollary}[Stability in $d$] \label{thm:identityd}
Let $g \geq 0$ and $n,\ell \geq 1$. For any partitions $\mu = (\mu_1,\ldots,\mu_n)$ and $\lambda = (\lambda_1,\ldots,\lambda_{\ell})$ such that $|\lambda| = |\mu|$, the following quantity is independent of the choice of an integer $d > \lambda_1$.
\[
d^{2g - 2 + n + \ell} \prod_{i = 1}^n \frac{(\mu_i/d)^{\lfloor \mu_i/d \rfloor}}{\lfloor \mu_i/d \rfloor !}\Bigg(\sum_{k = 1}^{\ell} \frac{(-1)^{\ell - k}}{k!} \sum_{
\substack{
\boldsymbol{\rho} \in (\mathscr{\tilde{P}}_{d - 1})^{k} \\ \sqcup_{\kappa} \rho^{(\kappa)} = \check{\lambda}
}
}
 \prod_{\kappa = 1}^{k} \frac{\big[\frac{d - |\rho^{(\kappa)}|}{d}\big]_{\ell(\rho^{(\kappa)}) - 1}}{|{\rm Aut}(\rho^{(\kappa)})|} \int_{\overline{\mathcal{M}}_{g,n + k}} \frac{\Omega^{[d]}_{g;-\overline{\mu},d - |\boldsymbol{\rho}|}}{\prod_{i = 1}^n \big(1 - \frac{\mu_i}{d}\psi_i\big)} \Bigg)
\]
\end{corollary}

When $\lambda$ satisfies the boundedness condition of \Cref{boundeddef}, the parenthesised expression above reduces to the single term \eqref{ksumsimp} and this corollary could have been deduced from \cite{joh-pan-tse11}.

Swapping the role of $(\lambda,\mu)$ merely exchanges the roles of $0$ and $\infty$ in the branched cover, so it does not change the corresponding double Hurwitz numbers. Invoking this symmetry, we obtain another identity, which now requires $d > \max(\mu_1, \lambda_1)$.

\begin{corollary}[Duality]
Let $g \geq 0$ and $n,\ell \geq 1$. For any partitions $\mu = (\mu_1,\ldots,\mu_n)$ and $\lambda = (\lambda_1,\ldots,\lambda_{\ell})$ such that $|\lambda| = |\mu|$, and for any integer $d > \max(\mu_1,\lambda_1)$, we have the identity
\begin{equation*}
\begin{split}
& \quad d^{2g - 2 + n + \ell} \Bigg(\sum_{k = 1}^{\ell} \frac{(-1)^{\ell - k}}{k!} \sum_{\substack{\boldsymbol{\rho} \in (\mathscr{\tilde{P}}_{d - 1})^{k} \\ \sqcup_{\kappa} \rho^{(\kappa)} = \check{\lambda}}} \prod_{\kappa = 1}^{k} \frac{\big[\frac{d - |\rho^{(\kappa)}|}{d}\big]_{\ell(\rho^{(\kappa)}) - 1}}{|{\rm Aut}(\rho^{(\kappa)})|} \int_{\overline{\mathcal{M}}_{g,n + k}} \frac{\Omega^{[d]}_{g;d - \mu,d - |\boldsymbol{\rho}|}}{\prod_{i = 1}^{n} \big(1 - \frac{\mu_i}{d}\psi_i\big)}\Bigg) \\
& = d^{2g - 2 + n + \ell} \Bigg(\sum_{k = 1}^{n} \frac{(-1)^{n - k}}{k!} \sum_{\substack{\boldsymbol{\rho} \in (\mathscr{\tilde{P}}_{d - 1})^{k} \\ \sqcup_{\kappa} \rho^{(\kappa)} = \check{\mu}}} \prod_{\kappa = 1}^{k} \frac{\big[\frac{d - |\rho^{(\kappa)}|}{d}\big]_{\ell(\rho^{(\kappa)}) - 1}}{|{\rm Aut}(\rho^{(\kappa)})|}  \int_{\overline{\mathcal{M}}_{g,\ell + k}} \frac{\Omega^{[d]}_{g;d-\lambda,d - |\boldsymbol{\rho}|}}{\prod_{j = 1}^{\ell} \big(1 - \frac{\lambda_j}{d}\psi_i\big)}\Bigg).
\end{split} 
\end{equation*}
Moreover, these expressions are independent of the choice of integer $d > \max(\mu_1,\lambda_1)$.
\end{corollary}

When $\lambda$ and $\mu$ satisfy the boundedness condition, the parenthesised expressions again reduce to single terms and the identity could have been deduced from \cite{joh-pan-tse11}.

\subsection{Remarks}

One of the main difficulties in the present work lies in the proof of~\Cref{thm:poly}, in which we deduce the polynomiality structure for double Hurwitz numbers. The method we employ is based on a careful analysis of the semi-infinite wedge expression for double Hurwitz numbers. These methods are developed in~\cite{dun-lew-pop-sha15, kra-lew-sha16} and have been successfully applied to many Hurwitz problems in the past. However, for double Hurwitz numbers, the analysis is much more intricate and we demonstrate how to carry it through.

Once~\Cref{thm:poly} is established, it is relatively straightforward to show (\cref{LemmaV}) that, for $2g - 2 + n > 0$, the generating series
\[
\sum_{\mu_1,\ldots,\mu_n \geq 1} DH_{g,n}(\mu_1,\ldots,\mu_n) \prod_{i = 1}^n e^{\mu_i x(z_i)}, 
\]
where $e^{x(z_i)} = z_ie^{-sQ(z_i)}$, has an analytic continuation to rational functions of $z_1,\ldots,z_n$ that satisfy the so-called ``linear loop equations''. Working with~\Cref{thm:poly} as a starting point, the proof of~\Cref{thm:TR} was completed in the previous work of the second- and third-named authors~\cite{do-kar18}. This idea of this proof consists in symmetrising the cut-and-join equations of~\cite{zhu12} with respect to the local Galois involution near the zeroes of ${\rm d}x$. This strategy has already been applied successfully to, for example, single Hurwitz numbers in~\cite{eyn-mul-saf11}, and $d$-orbifold Hurwitz numbers in~\cite{do-lei-nor16,bou-her-liu-mul14}. We add in~\Cref{Scontinuity} a continuity argument to waive the technical assumption --- initially needed in this approach --- that the zeroes of ${\rm d}x$ are simple.

It is interesting to note that Alexandrov, Chapuy, Eynard and Harnad have considered in~\cite{ACEH2,ACEH3} more general, ``weighted'' double Hurwitz problems, where the ramifications over $0$ and $\infty$ are encoded as in~\Cref{sec:dhdef}, but various types of branching are allowed elsewhere with a weight encoded by a formal series $G(z)$. The double Hurwitz numbers in~\Cref{doubledef} correspond to $G(z) = e^{sz}$ in their notation. They prove in~\cite{ACEH3} the topological recursion in the case $G$ is polynomial, thus leaving our~\Cref{thm:TR} out of reach. This shortcoming seems difficult to avoid in their method, which is based on the theory of reconstruction of WKB expansions by the topological recursion originating in~\cite{BEdet,BBEpq} and applied to the ODEs derived in~\cite{ACEH2}, which is finite order when $G$ is polynomial.

One remarkable feature of the topological recursion is that its correlation differentials can always be expressed via intersection theory on moduli spaces of curves~\cite{eyn14}. However, when applied to the spectral curve \eqref{spdouble} governing double Hurwitz numbers, the expression taken directly from~\cite{eyn14} looks rather complicated. We take an alternative route to obtain~\Cref{thm:ELSV}. For $d$-orbifold Hurwitz numbers, the computation has already been carried out in~\cite{lew-pop-sha-zvo17} and we make use of an ELSV-like formula involving Chiodo classes and equivalent to the Johnson--Pandharipande--Tseng formula of~\cite{joh-pan-tse11}. Using the properties of the topological recursion under variations of the spectral curve~\cite{eyn-ora07}, we then flow to non-zero values of $q_1,\ldots,q_{d - 1}$ to obtain~\Cref{thm:ELSV}. As this should be carried out keeping $x(z)$ fixed along the flow, this introduces further complications, responsible for the sums over $k$. This method is general and could be applied to obtain ELSV-like formulas for other types of double Hurwitz problems.

At present, we are inclined to think that~\Cref{thm:ELSV} is not the most ``elegant'' ELSV formula that should exist. First, it would be desirable to convert it to an explicit formula involving only intersection theory on $\overline{\mathcal M}_{g,n}$; we cannot currently carry this out as the behaviour of the Chiodo class for arbitrary indices under pushforward along the morphism forgetting marked points does not seem to be known. Second, an ``elegant'' ELSV formula should directly imply the polynomiality structure, and it is currently unclear to us whether~\Cref{thm:ELSV} logically implies \Cref{thm:poly}. Goulden, Jackson and Vakil have conjectured another ELSV-like formula, relating $\HH_{g,1}$ to intersection theory on the universal Picard stack $\overline{{\rm Pic}}_{g,n}$ of certain classes that remain to be determined~\cite{gou-jac-vak05}. This conjecture is discussed in~\cite{do-lew19}. An even more vague extension of this conjecture to all double Hurwitz numbers has been proposed in~\cite{Cavalierireview}. We can only imagine that the common pushforwards of all such formulas to $\overline{\mathcal{M}}_{g,n}$ would agree.

\subsection{Organisation of the paper}

In~\Cref{sec:background}, we briefly review the topological recursion of Chekhov, Eynard and Orantin. For the spectral curve \eqref{spdouble}, we study in detail a vector space of rational $1$-forms in the variable $z$ to which the correlation differentials belong, and give its isomorphic description in terms of sequences of coefficients for their expansion near $z = 0$. This exhibits the type of polynomial structure which we seek to establish for double Hurwitz numbers, and allows us to prove the topological recursion,~\Cref{thm:TR}, conditionally on~\Cref{thm:poly} by invoking~\cite{do-kar18}. 

\Cref{sec:DHviainfinitewedge} starts from a well-known expression of double Hurwitz numbers as vacuum expectation in the semi-infinite wedge (\Cref{thm:DHinIW1}) and rewrites it in a fashion convenient for the forthcoming analysis (\Cref{thm:DHinIW2}). \Cref{sec:poly} is devoted to proving the polynomiality structure,~\Cref{thm:poly}, within this framework.

In~\Cref{sec:ELSV}, we introduce the Chiodo classes and review the $d$-orbifold ELSV formula. We then show how it can be deformed to prove~\Cref{thm:ELSV}, and obtain the vanishing~\Cref{thm:vanishC} by expressing the consistency of our computations.

The paper concludes with appendices that provide a new proof of the previously known formula for the generating series for $\HH_{0,2}(\mu_1,\mu_2)$, and examples of our intersection-theoretic results tested numerically with the help of the SageMath package \textsc{admcycles}~\cite{admcycles}.

\subsection{Notations}

The symmetric group on the set $\{1, 2, \ldots, N\}$ is denoted $\mathfrak{S}_{N}$. The set of all integer partitions $\lambda_1 \geq \cdots \geq \lambda_{\ell} > 0$ is denoted $\mathscr{P}$ and it includes by convention the empty partition. For a partition $\lambda$, we denote by $\ell(\lambda)$ the number of parts, $|\lambda| = \sum_{i = 1}^{\ell(\lambda)} \lambda_i$ the size, and ${\rm Aut}(\lambda)$ the set of permutations leaving the tuple $(\lambda_1,\ldots,\lambda_{\ell})$ invariant. Another notation for a partition $\lambda$ is $(1^{p_1}2^{p_2}\cdots)$, which indicates that it contains $p_i$ parts of length $i$, for each positive integer $i$. In that case, we have
\[
\ell(\lambda) = \sum_{i \geq 1} p_i,\qquad |\lambda| = \sum_{i \geq 1} ip_i,\qquad |{\rm Aut}(\lambda)| = \prod_{i \geq 1} p_i!\,.
\]
If $d \geq 2$, we let $\mathscr{P}_{d - 1}$ be the set of partitions such that $\lambda_1 \leq d - 1$. If $\lambda = (1^{p_1}2^{p_2}\cdots (d - 1)^{p_{d - 1}}) \in \mathscr{P}_{d - 1}$, its complement is $\check{\lambda} = (1^{p_{d - 1}}2^{p_{d - 2}}\dots (d - 1)^{p_{1}})$. The complement operation is an involution on $\mathscr{P}_{d - 1}$, such that
\[
\ell(\check{\lambda}) = \ell(\lambda),\qquad |\check{\lambda}| = d\ell(\lambda) - |\lambda|,\qquad |{\rm Aut}(\check{\lambda})| = |{\rm Aut}(\lambda)|.
\]
We let $\tilde{\mathscr{P}}_{d - 1}$ be the set of partitions of size at most $d - 1$.

We write $\lambda \vdash N$ to say that $\lambda$ is a partition of size $N$. If $I$ is a finite set, we also write $M \vdash I$ to say that $M$ is a set partition of $I$; that is, a set $\{M_{1},\ldots,M_{m}\}$ of pairwise disjoint non-empty subsets $M_i \subseteq I$ whose union is $I$.

Let $\mathscr{R}_{N} = \mathbb{Q}[x_1,\ldots,x_N]^{\mathfrak{S}_{N}}$ be the ring of symmetric functions in $N$ variables, graded by homogeneous degree. We will work with its projective limit $\mathscr{R} = \varprojlim \mathscr{R}_N$ with respect to the restriction morphisms $\mathscr{R}_{N + 1} \rightarrow \mathscr{R}_{N}$ that set $x_{N + 1}$ to zero. By construction, there exists a specialisation map which is a graded ring homeomorphism ${\rm ev}_{N} \colon \mathscr{R} \rightarrow \mathscr{R}_N$. As a graded ring, $\mathscr{R}$ is generated by homogeneous elements $\mathfrak{p}_{k}$ of degree $k$ and indexed by $k \geq 1$, whose specialisation are the power sums
\[
{\rm ev}_{N}(\mathfrak{p}_{k})(x_1,\ldots,x_N) = \sum_{i = 1}^N x_i^{k}.
\]
We may drop ${\rm ev}_N$ from the notation when the context is clear. If $\lambda$ is a partition, then we set $\mathfrak{p}_{\lambda} = \prod_{i = 1}^{\ell(\lambda)} \mathfrak{p}_{\lambda_i}$.

If $x \in \mathbb{Q}$, $\lceil x \rceil$ is the unique integer such that $\lceil x \rceil - 1 < x \leq \lceil x \rceil$, while $\lfloor x \rfloor$ is the unique integer such that $\lfloor x \rfloor \leq x < \lfloor x \rfloor + 1$.

If $f$ is a formal Laurent series in a variable $z$, the coefficient of $z^{j}$ in $f$ is denoted $[z^{j}]\,f$.

If $l$ is an integer and $x \in \mathbb{Q}$, we write $[x]_{l} = x(x + 1)\cdots (x + l - 1)$.

\subsection{Acknowledgements.} G.B. benefits from the support of the Max-Planck-Gesellschaft. N.D. was supported by the Australian Research Council grant DP180103891. D.L. has been supported by the Max-Planck-Gesellschaft, by the ERC Synergy grant ``ReNewQuantum'' at IPhT Paris and at IH\'{E}S Paris, France, and by the Robert Bartnik Visiting Fellowship at Monash University, Melbourne, Australia. E.M. was supported by an Australian Government Research Training Program (RTP) Scholarship. M.K is grateful to F. Petrov and P. Zatitskii for fruitful conversations. We thank Johannes Schmitt for his help with the SageMath package \textsc{admcycles} and for his efforts adapting it to include the possibility of intersecting Chiodo classes of arbitrary parameters. This allowed us to test and correct previous versions of Theorems~\ref{thm:ELSV} and \ref{thm:vanishC}. We thank Alessandro Giacchetto for independent numerical checks of integrals of Chiodo classes via topological recursion.

\section{Topological recursion and preliminaries}

\label{sec:background}

\subsection{Definition and basic properties} \label{subsec:TR}

The topological recursion of Chekhov, Eynard and Orantin~\cite{che-eyn06,EORev} emerged from the theory of matrix models and has since then been shown to govern a widespread collection of problems in enumerative geometry in a universal way. Here we spell out its definition in a restricted setting which is sufficient for the statement of~\Cref{thm:TR}; for a more detailed review, see~\cite{EynardICM}.

\textbf{Input.} A \emph{spectral curve} is the data of two meromorphic functions $x$ and $y$ on a smooth complex curve $\mathscr{C}$, together with a bidifferential $\omega_{0,2}$ on $\mathscr{C}^2$ whose only singularity is a double pole on the diagonal with biresidue $1$. We shall assume that ${\rm d}x$ has finitely many zeroes, which are simple and disjoint from the zeroes and the poles of ${\rm d}y$. We denote by $\mathfrak{a}$ the set of zeroes of ${\rm d}x$.

Topological recursion takes as input a spectral curve and outputs a family of multidifferentials $\omega_{g,n}$ on $\mathscr{C}^n$, indexed by integers $g \geq 0$ and $n \geq 1$ and which are symmetric under permutation of the factors of $\mathscr{C}^n$. We call them {\em correlation differentials}, and more precisely, for $2g - 2 + n > 0$,
\[
\omega_{g,n} \in H^0(\mathscr{C}^n,K_{\mathscr{C}}(*\mathfrak{a})^{\boxtimes n}\big)^{\mathfrak{S}_{n}},
\]
where $K_{\mathscr{C}}(*\mathfrak{a})$ is the sheaf of meromorphic $1$-forms with poles of arbitrary order at $\mathfrak{a}$.

\textbf{Base cases.} We set $\omega_{0,1}(z_1) = y(z_1) \dd x(z_1)$, while $\omega_{0,2}(z_1, z_2)$ is specified by the initial data.

\textbf{Recursion kernel.} For $\alpha \in \mathfrak{a}$, we define $\sigma_{\alpha}$ to be the non-trivial holomorphic involution defined locally near $\alpha$ such that $x \circ \sigma_{\alpha} = x$. The recursion kernel is defined to be
\[
K_{\alpha}(z_0,z) := \frac{1}{2}\,\frac{\int_{\sigma_{\alpha}(z)}^{z} \omega_{0,2}(z_0,\cdot)}{[y(z) - y(\sigma_{\alpha}(z))] \, {\rm d}x(z)}.
\]
It is a $1$-form in the variable $z_0 \in \mathscr{C}$ and the inverse of a $1$-form in the variable $z$ near $\alpha$.

\textbf{Recursion formula.} The correlation differentials are defined by induction on $2g - 2 + n > 0$ by the formula
\begin{equation}
 \label{eq:TRrecursion}
\begin{split}
\omega_{g,n}(z_1, \vec{z}_I) = \sum_{\alpha \in \mathfrak{a}} \Res_{z=\alpha} K_{\alpha}(z_1, z) \bigg( \omega_{g-1, n+1}(z, \sigma_{\alpha}(z), \vec{z}_I) + \sum_{\substack{h + h' = g \\ J \sqcup J' = I}}^{\circ} \omega_{h, |J|+1}(z, \vec{z}_J)\otimes \omega_{h', |J'|+1}(\sigma_{\alpha}(z), \vec{z}_{J'}) \bigg).
	\end{split}
\end{equation}
Here, we take $I = \{2,\ldots,n\}$ and use the notation $\vec{z}_{S} := \{z_i\}_{i \in S}$. The symbol ${\circ}$ over the inner summation means that we exclude all terms involving $\omega_{0,1}$ from the sum. Although it is not obvious from \eqref{eq:TRrecursion}, the $\omega_{g,n}$ defined in this way are indeed invariant under permutations of $z_1,\ldots,z_n$~\cite{eyn-ora07}. 

\textbf{Linear loop equation.} The $\omega_{g,n}(z_1,\ldots,z_n)$ defined by \eqref{eq:TRrecursion} provide the unique solution of a set of linear and quadratic loop equations~\cite{bor-eyn-ora15,BSblob}. In the present work, we need only deal with the linear loop equations. They state that for $2g - 2 + n > 0$ and each $\alpha \in \mathfrak{a}$, 
\[
\omega_{g,n}(z,\vec{z}_I) + \omega_{g,n}(\sigma_{\alpha}(z),\vec{z}_I)
\]
is regular at $z = \alpha$.

\subsection{Higher order zeroes for ${\rm d}x$ and continuity}
\label{Scontinuity}

In~\cite{BEthink}, Bouchard and Eynard give the appropriate (more complicated) definition of the correlation differentials when the zeroes of ${\rm d}x$ are not assumed to be simple. In the case of simple zeroes, it reduces to the definition \eqref{eq:TRrecursion}. They also prove a continuity result, which we only need in a restricted setting and can state in the following way.
\begin{theorem}~\cite[Section 3.5]{BEthink} \label{thm:cont}
Let $\mathcal{T} \subset \mathbb{C}$ be an open subset and $(\mathscr{S}_{t})_{t \in \mathcal{T}}$ be a continuous family of spectral curves with $\mathscr{C}_{t} = \mathbb{C}^*$ and such that $y_t$ (resp. ${\rm d}x_t$) extends to a rational function (resp. $1$-form). Denoting $\omega_{g,n}^{t}$ the correlation differentials for $\mathscr{S}_{t}$, for $2g - 2 + n > 0$ the function
\[
(t,z_1,\ldots,z_n) \mapsto \frac{\omega_{g,n}^{\mathscr{S}_{t}}(z_1,\ldots,z_n)}{{\rm d}x_t(z_1) \otimes \cdots \otimes {\rm d}x_t(z_n)}
\]
is continuous in the domain $\mathcal{T} \times (\mathbb{C}^* \setminus \mathfrak{a})^{n}$.
\end{theorem}

\begin{corollary}
\label{cor:higher} If~\Cref{thm:TR} holds under the stronger condition that $1 - szQ'(z)$ has simple roots, then it holds in full generality provided we use the Bouchard--Eynard topological recursion~\cite{BEthink} to define the correlation differentials.
\end{corollary}
\begin{proof} The definition of $DH_{g,n}(\mu_1,\ldots,\mu_n)$ immediately shows that, for each $g,n,\mu_1,\ldots,\mu_n$, it is a polynomial (in particular, a continuous function) of the parameters $s,q_1,\ldots,q_d$, whose value can be reached by taking a limit of parameters for which $1 - szQ'(z)$ has simple roots. ~\Cref{thm:TR} for the latter case is equivalent to
\[
DH_{g,n}(\mu_1,\ldots,\mu_n) = \Res_{z_1 = 0} \cdots \Res_{z_n = 0} \omega_{g,n}(z_1,\ldots,z_n) \prod_{i = 1}^n \frac{e^{-\mu_i x(z_i)}}{\mu_i},
\]
and~\Cref{thm:cont} implies that the right-hand side is compatible with taking limits in the $\omega_{g,n}$ defined by the Bouchard--Eynard topological recursion.
\end{proof}

\subsection{A vector space of rational functions}
\label{defnugfngu}
In this section, we focus on the spectral curve given by \eqref{spdouble}. As the linear loop equations are satisfied by the correlation differentials $\omega_{g,n}$, it is useful to study in more detail their space of solution. This leads us to introduce the following vector space over $\mathbb{K} := \mathbb{C}(s,q_1,\ldots,q_d)$.

\begin{definition} \label{defn:V(z)}
Let $\hat{W}(z)$ be the $\mathbb{K}$-vector space consisting of rational functions $\Phi(z)$ such that
\begin{itemize}[topsep=0pt]
\item $\Phi(z)$ has poles only when $z \in \mathfrak{a}$; and
\item for each $\alpha \in \mathfrak{a}$, $\Phi(z) + \Phi(\sigma_{\alpha}(z))$ is regular at $z = \alpha$.
\end{itemize}
\end{definition}

The basic properties of the correlation differentials (linear loop equations and symmetry in their $n$ variables) imply that, for $2g - 2 + n > 0$,
\begin{equation}
\label{primHurw} \int_{0}^{z_1}\cdots \int_{0}^{z_n} \omega_{g,n} \in \bigotimes_{i = 1}^n \hat{W}(z_i).
\end{equation}

We shall describe a convenient basis for $\hat{W}(z)$. For this purpose we define a family of rational functions $\hat{\phi}^j_{m}(z)$ indexed by $j \in \{1,\ldots,d\}$ and an integer $m \geq -1$, by induction on $m$. We set $\hat{\phi}^j_{-1}(z) = z^j$ and use the notation $\partial_z = \frac{\partial}{\partial z}$ to define, for all $m \geq -1$,
\[
\hat{\phi}^j_{m+1}(z) = \frac{z}{1 - szQ'(z)} \partial_{z}\hat{\phi}_m^j(z) = \partial_{x(z)} \hat{\phi}_m^j(z).
\]

\begin{lemma} \label{lem:basisV} 
The vector space $\hat{W}(z)$ admits the linear basis $\hat{\phi}^j_m(z)$ indexed by $j \in \{1,\ldots,d\}$ and $m \geq 0$. It has the structure of a $\mathbb{K}[\partial_{x(z)}]$-module and as such it is generated by $(\hat{\phi}^j_0(z))_{j = 1}^{d}$.
\end{lemma}
\begin{corollary} 
\label{cor:decomw} For each $(g,n)$ such that $2g - 2 + n > 0$, there exists $C_{g,n}\big(\begin{smallmatrix} j_1 & \cdots & j_n \\ m_1 & \cdots & m_n\end{smallmatrix}\big) \in \mathbb{K}$ indexed by $j_i \in \{1,\ldots,d\}$ and non-negative integers $m_i$, which vanish for all but finitely many $(j_i,m_i)_{i = 1}^n$, and satisfying
\[
\omega_{g,n}(z_1,\ldots,z_n) = \sum_{\substack{1 \leq j_1,\ldots,j_n \leq d \\ m_1,\ldots,m_n \geq 0}} C_{g,n}\big(\begin{smallmatrix} j_1 & \cdots & j_n \\ m_1 & \cdots & m_n \end{smallmatrix}\big) \bigotimes_{i = 1}^n {\rm d}\hat{\phi}_{m_i}^{j_i}(z_i).
\]
\end{corollary} 
\begin{proof}[Proof of~\Cref{lem:basisV,cor:decomw}]
For each $\alpha \in \mathfrak{a}$, we define a local coordinate $w_{\alpha}$ near $\alpha$ such that $x = w_{\alpha}^2 + x(\alpha)$. The local involution $\sigma_{\alpha}$ is therefore realised by $w_{\alpha} \mapsto -w_{\alpha}$. We observe that for $\hat{\phi}^j_{-1}(z)  = z^j $, its even part $\hat{\phi}^j_{-1}(z) + \hat{\phi}^j_{-1}(\sigma_{\alpha}(z))$ is regular at $z=\alpha$, and that $\hat{\phi}^j_{0}(z)$ remains finite when $z \rightarrow \infty$. Given that $\hat{\phi}_{m}^j(z) = \partial_{x(z)}^{m}\hat{\phi}_0^j(z)$, showing that $\hat{\phi}_{m}^j(z) \in \hat{W}(z)$ for $j \in \{1,\ldots,d\}$ and $m \geq 0$ reduces to showing that $\partial_{x(z)}$ preserves $\hat{W}(z)$.

The Laurent expansion of a rational function $\Phi(z)$ at $z = \alpha$ in the local coordinate $w_{\alpha}$ gives an element $L_{\alpha}(\Phi) \in \mathbb{K}((w_{\alpha}))$. Elements of $\hat{W}(z)$ are characterised by the property that they only have poles at $\mathfrak{a}$ and for each $\alpha \in \mathfrak{a}$, the even part of $L_{\alpha}(\Phi)$ has only non-negative powers. The action of $\partial_{x}$ is realised in this space of Laurent series as $(2w_{\alpha})^{-1}\partial_{w_{\alpha}}$, and it preserves the parity of any monomial in $w_{\alpha}$. Since applying $\partial_{x(z)}$ does not introduce poles outside of $\mathfrak{a}$, we deduce that $\partial_{x(z)}\hat{W}(z) \subseteq \hat{W}(z)$. Thus, $\hat{\phi}^j_m(z) \in \hat{W}(z)$ for $j \in \{1,\ldots,d\}$ and $m \geq 0$.

Let us now consider the linear map
\[
L_-\,\colon\,\hat{W}(z) \longrightarrow \bigoplus_{\alpha \in \mathfrak{a}} w_{\alpha}^{-1}\mathbb{K}[w_{\alpha}^{-2}],
\] 
that associates $\Phi \in \hat{W}(z)$ to the principal part of $L_{\alpha}(\Phi)$; that is, keeping only the monomials of negative degree. It is easy to check that $L_-(\hat{\phi}^j_{m})$ indexed by $j \in \{1,\ldots,d\}$ and $m \geq 0$ gives a basis of $\bigoplus_{\alpha} w_{\alpha}^{-1}\mathbb{K}[w_{\alpha}^{-2}]$. Liouville's theorem implies that $L_-$ is injective and the formula (where $z$ is outside the contour of integration)
\[
\Phi(z) = \sum_{\alpha \in \mathfrak{a}} \Res_{\tilde{z} = \alpha} \frac{{\rm d}\tilde{z}}{z - \tilde{z}}\,L_{-}(\Phi)(w_{\alpha}(\tilde{z}))
\]
provides an inverse to $L_-$, showing that $L_-$ is an isomorphism. Therefore, $\hat{\phi}^j_{m}(z)$ with the given range of indices forms a basis of $\hat{W}(z)$.~\Cref{cor:decomw} then comes from the application of $\otimes_{i = 1}^n {\rm d}_{z_i}$ to the decomposition of the left-hand side of \eqref{primHurw} in this basis. 
\end{proof} 

It is convenient to introduce the following filtration on the vector space $\hat W(z)$.
\begin{definition}
\label{def:filtration}
The subspace $\hat W_f(z) \subset \hat{W}(z)$ for $f \in \mathbb Z_{\ge 0}$ is spanned by the functions $\hat \phi^j_m$ for $j \in \{1,\ldots,d\}$, and $m \in \{0,\ldots, f\}$.
\end{definition}
Clearly $\partial_x \hat W_f(z) \subset \hat W_{f+1}(z)$ for any $f\in \mathbb Z_{\ge 0}$.

\subsection{Isomorphic description as a vector space of sequences}

Double Hurwitz numbers are stored in the power series expansion of the correlation differentials near $z_i \rightarrow 0$ in the variable $X(z_i) := e^{x(z_i)}$. It is therefore useful to study such power series expansions for elements of $\hat{W}(z)$. Up to the application of $\partial_{x(z)} = X(z)\, \partial_{X(z)}$, it is sufficient to do it for $\hat{\phi}_{-1}^j(z) = z^j$.

\begin{lemma}
\label{lem:expmoins1} We have the following all-order series expansion when $z \rightarrow 0$:
\[
\hat{\phi}_{-1}^j(z) \sim \sum_{\mu \geq 1} \phi_{-1}^{j}(\mu)\,e^{\mu x(z)},\qquad \phi_{-1}^{j}(\mu) := \frac{j}{\mu} \sum_{\lambda \vdash \mu - j} \frac{(\mu s)^{\ell(\lambda)}}{|{\rm Aut}(\lambda)|}\,\vec{q}_{\lambda}.
\]
\end{lemma}
\begin{proof} The existence of the coefficients $\phi_{-1}^{j}(\mu)$ for which we have an expansion of this form is clear since $X(z) = z + O(z^2)$. We compute by Lagrange inversion
\begin{equation*}
\begin{split} 
\phi_{-1}^{j}(\mu) & = \Res_{z = 0} \frac{{\rm d}X(z)}{X(z)^{\mu + 1}}\,z^j = \frac{j}{\mu}\,\Res_{z = 0} X(z)^{-\mu}\,z^{j - 1}\,{\rm d}z \\
& = \frac{j}{\mu}\,\Res_{z = 0} z^{j - 1 - \mu} e^{\mu s Q(z)}\,{\rm d}z = \frac{j}{\mu}\,[z^{\mu - j}] \bigg(\sum_{p_1,\ldots,p_d \geq 0} \prod_{j = 1}^{d} \frac{(\mu q_j s)^{p_j}\,z^{jp_j}}{p_j!}\bigg) \\
& = \frac{j}{\mu}\,[z^{\mu - j}] \sum_{\substack{\lambda \in \mathscr{P} \\ \lambda_1 \leq d}} \frac{(\mu s)^{\ell(\lambda)}}{|{\rm Aut}(\lambda)|}\,\vec{q}_{\lambda}\,z^{\lambda_1 + \cdots + \lambda_{\ell(\lambda)}}.
\end{split}
\end{equation*}
Here, the second equality is obtained via integration by parts. In the last line, we have mapped the tuple $(p_1,\ldots,p_d)$ to the partition $\lambda = (1^{p_1}2^{p_2}\cdots d^{p_d})$, but must divide by $|{\rm Aut}(\lambda)|$, which is the number of tuples realising a given partition. The restriction that $\lambda_1 \leq d$ can be waived with the convention that $q_j = 0$ for $j > d$ and we obtain the claim.
\end{proof}

This leads us to introduce the following $\mathbb{K}$-vector space of sequences which are obtained by series expansion of elements of $\hat{W}(z)$.

\begin{definition} \label{defn:V(mu)}
Define the vector space $W(\mu) \subset \mathbb{K}^{\mathbb{Z}_{> 0}}$ to be the $\mathbb{K}$-span, over $j \in \{1,\ldots,d\}$ and $m \geq 0$, of the sequences $(\phi^j_m(\mu))_{\mu > 0}$ defined by
\[
\phi^j_{m}(\mu) = j \sum_{\lambda \vdash \mu - j} \frac{(\mu s)^{\ell(\lambda)}\,\mu^{m}}{|{\rm Aut}(\lambda)|}\,\vec{q}_{\lambda}.
\]
\end{definition}

By successive applications of $\partial_{x} = X\partial_{X}$ to~\Cref{lem:expmoins1}, we find for any $j \in \{1,\ldots,d\}$ and $m \geq 0$, the all-order series expansion when $z \rightarrow 0$ is
\begin{equation}
\label{theexpan}
\hat{\phi}^j_{m}(z) \sim \sum_{\mu \geq 1} \phi_m^j(\mu) {\,e^{\mu x(z)}}.
\end{equation}

\begin{lemma}\label{LemmaV}
The linear map sending $\Phi \in \hat{W}(z)$ to the sequence of coefficients $(\mu \mapsto [X^{\mu}]\,\Phi(z)) \in W(\mu)$ is an isomorphism.
\end{lemma}
\begin{proof} Equation~\eqref{theexpan} shows that the image of this map is indeed $W(\mu)$. Since elements of $\hat{W}(z)$ are analytic functions of $X(z)$ near $z = 0$, this map is injective.
\end{proof}

The vector space $W(\mu)$ inherits the filtration from $\hat W(z)$. Define $W_f(\mu)$ for $f\in \mathbb Z_{\ge 0}$ to be the span of the sequences $\phi_m^j$ for $j \in \{1,\ldots,d\}$ and $m \in \{0,\ldots,f\}$. 

\begin{corollary}
\label{cor:polyeq} The polynomiality structure,~\Cref{thm:poly}, is equivalent to the statement that for $2g - 2 + n > 0$, the formal series
\[
\sum_{\mu_1,\ldots,\mu_n \geq 1} DH_{g,n}(\mu_1,\ldots,\mu_n) \prod_{i = 1}^n e^{\mu_i x(z_i)}
\]
is the all-order series expansion at $z_i \rightarrow 0$ of a rational function belonging to $\otimes_{i = 1}^n \hat{W}(z_i)$. \hfill $\blacksquare$
\end{corollary}

\begin{corollary}
\label{cor:TRimplypoly}~\Cref{thm:TR} (Topological recursion) implies~\Cref{thm:poly} (Polynomiality).
\end{corollary}
\begin{proof}
Since for $2g - 2 + n > 0$ the correlation differentials $\omega_{g,n}(z_1,\ldots,z_n)$ satisfy the linear loop equations with respect to each variable (by symmetry) and have only poles at $z_i \in \mathfrak{a}$, we have $\omega_{g,n} \in \otimes_{i = 1}^n {\rm d}\hat{W}(z_i)$. We conclude by comparing the result of~\Cref{cor:decomw} and the statement of~\Cref{thm:TR} with the statement of~\Cref{thm:poly}.
\end{proof}

It will also be useful to consider the action of differentiating with respect to the parameters $q_j$.
\begin{lemma}
\label{lem:dmodq} The vector space $W(\mu) \cong \hat{W}(z)$ has the structure of a $\mathbb{C}[q_1\partial_{q_1},\ldots,q_d\partial_{q_d}]$-module. For any $j \in \{1,\ldots,d\}$ and $f \geq 0$, we have $q_j \partial_{q_j} W_f(\mu) \subset W_{f+1}(\mu)$.
\end{lemma}
\begin{proof}
For $j \in \{1,\ldots,d\}$, let us consider the operator $\vartheta_{j}$ on $\mathbb{K}(z)$ consisting in applying the derivation $q_j\partial_{q_j}$ while $x(z)$ remains fixed. Note that $\partial_{x}$ commutes with $\vartheta_{1},\ldots,\vartheta_{d}$. For $f \in \mathbb{K}(z)$, we have
\[
\vartheta_{j} f(z) = q_j\frac{\partial}{\partial q_j}\bigg|_{z}\,f(z) + \frac{sq_jz^{j + 1}}{1 - szQ'(z)}\,\partial_{z}f(z).
\]
In particular, for $k \in \{1,\ldots,d\}$, we compute
\[
\vartheta_{j} \hat{\phi}^k_{-1}(z) = \vartheta_{j} z^{k} = \frac{ksq_jz^{j + k}}{1 - szQ'(z)}.
\]
Recalling that, for $l \in \{1,\ldots,d\}$,
\[
\hat{\phi}^l_{0}(z) = \frac{lz^{l}}{1 - szQ'(z)},
\]
we can perform the Euclidean division $z^{j + k } = A_{j + k}(z)(1 - szQ'(z)) + B_{j + k}(z)$ in $\mathbb{K}[z]$, and decompose
\[
A_{j + k}(z) = \sum_{l = 1}^{j + k - d} A_{j + k,l} z^{l },\qquad B_{j + k}(z) = \sum_{l = 2}^{d} l\,B_{j + k,l} z^{l }.
\]
We obtain
\[
\vartheta_{j} \hat{\phi}^{k}_{-1} = ksq_j \bigg[ \sum_{l = 1}^{j + k - d} A_{j + k,l} \hat{\phi}^{l}_{-1} + \sum_{l = 2}^{d} B_{j + k,l}\,\hat{\phi}^l_{0}\bigg].
\]
Since $\vartheta_{j}$ commutes with $\partial_{x}$, this relation implies that for any $m \geq 0$, $\vartheta_j \hat{\phi}^{k}_{m}$ is a linear combination of $\hat{\phi}^{l}_{m}$ and $\hat{\phi}^{l}_{m + 1}$ with $l \in \{1,\ldots,d\}$. Hence $\vartheta_{j}$ leaves $\hat{W}(z)$ stable and is compatible with the filtration in the announced way.
\end{proof}

\subsection{Equivalence between polynomiality and topological recursion} \label{subsec:equivTRpoly}

In this section, we rely on the previous work of the second- and third-named authors~\cite{do-kar18} to prove the following lemma.
\begin{lemma}
\Cref{thm:poly} (Polynomiality) implies~\Cref{thm:TR} (Topological recursion). 
\end{lemma}

Assuming the statement of~\Cref{thm:poly}, for $2g - 2 + n > 0$,~\Cref{cor:polyeq} implies the existence of rational multidifferentials $\tilde{\omega}_{g,n}(z_1,\ldots,z_n)$ with the following properties:
\begin{itemize}
\item[$\bullet$] their only poles are located at $z_i \in \mathfrak{a}$; 
\item[$\bullet$] they satisfy the linear loop equations; that is, for each $\alpha \in \mathfrak{a}$ and setting $I = \{2,\ldots,n\}$, $\tilde{\omega}_{g,n}(z,\vec{z}_I) + \tilde{\omega}_{g,n}(\sigma_{\alpha}(z),\vec{z}_I)$ is regular at $z = \alpha$; and
\item[$\bullet$] they admit the following all-order series expansion when $z_i \rightarrow 0$:
\begin{equation}
\label{seriesomtilde} \tilde{\omega}_{g,n}(z_1,\ldots,z_n) \sim \sum_{\mu_1,\ldots,\mu_n \geq 1} DH_{g,n}(\mu_1,\ldots,\mu_n)\,\bigotimes_{i = 1}^n {\rm d}(e^{\mu_i x(z_i)}).
\end{equation}
\end{itemize}
 
The work of~\cite{do-kar18} shows that this information is sufficient to derive the topological recursion theroem,~\Cref{thm:TR}. Let us summarise their strategy.
The story starts with the cut-and-join equation for double Hurwitz numbers which originally appears in~\cite{zhu12} and follows from the semi-infinite wedge representation of these numbers (reviewed in~\Cref{Svacc}). The cut-and-join allows a direct computation of the generating series of double Hurwitz numbers for $(g,n) = (0,1)$ and $(0,2)$, the outcome being that \eqref{seriesomtilde} remains valid for $(g,n) = (0,1)$ and $(0,2)$ with
\[
\tilde{\omega}_{0,1}(z) := y(z) \, {\rm d}x(z) = \omega_{0,1}(z),\qquad \tilde{\omega}_{0,2}(z_1,z_2) := \omega_{0,2}(z_1,z_2) - \frac{{\rm d}x(z_1) \otimes {\rm d}x(z_2)}{(x(z_1) - x(z_2))^2},
\]
where we recall that $\omega_{0,2}(z_1,z_2) = {\rm d}z_1 \otimes {\rm d}z_2/(z_1 - z_2)^2$. For $2g - 2 + n > 0$, the cut-and-join equations imply quadratic functional relations for the multidifferentials $\tilde{\omega}_{g,n}(z_1,\ldots,z_n)$, which cannot be directly solved and where all variables $(z_i)_{i = 1}^n$ play a symmetric role. When the zeroes $\alpha$ of ${\rm d}x$ are simple, one can project this functional equation to the even part for the involution $\sigma_{\alpha}$ with respect to $z_1$. The repeated use of linear loop equations then shows by induction on $2g - 2 + n > 0$ that $\tilde{\omega}_{g,n}$ satisfies the quadratic loop equations (in which $z_1$ plays a special role). The results of~\cite{bor-eyn-ora15,BSblob} in turn imply that $\tilde{\omega}_{g,n}$ is computed by the topological recursion for the spectral curve announced in \eqref{spdouble}; that is, it coincides with $\omega_{g,n}$ and~\Cref{thm:TR} holds.~\Cref{cor:higher} then waives the restriction on the simplicity of the zeroes of ${\rm d}x$ and implies~\Cref{thm:TR} in full generality.
 \subsection{From one to many variables}\label{sec:sym}

At the end of the proof of polynomiality in~\Cref{sec:poly}, we need the following algebraic lemma.

\begin{lemma} \label{lem:symmetry}
Let $F(\mu_1, \mu_2, \ldots, \mu_n)$ be a symmetric function of positive integers $\mu_1, \mu_2, \ldots, \mu_n$. If there exists $f\in \mathbb Z_{\ge 0}$ such that for any $\mu_2,\ldots,\mu_n\in \mathbb Z_{>0}$, we have $F(\mu_1, \mu_2, \ldots, \mu_n) \in W_f(\mu_1)$, then
\[
F(\mu_1, \mu_2, \ldots, \mu_n) \in \bigotimes_{i = 1}^n W(\mu_i).
\]
\end{lemma}

Before presenting the proof of~\Cref{lem:symmetry}, we describe another useful basis for the vector space $W(\mu)$.
\begin{lemma}
\label{newblem}
The vector space $\hat{W}(z)$ admits a basis $(\hat{\psi}_k)_{k \geq 1}$ such that for any $k \geq 1$, we have $\hat{\psi}_k(z) = z^{k}+ O(z^{k + 1})$ when $z \rightarrow 0$.
Equivalently, the vector space $W(\mu)$ admits a basis $(\psi_k)_{k \geq 1}$ such that $\psi_{k}(\mu) = \delta_{k,\mu}$ for any $1 \leq \mu \leq k$. Moreover, for $f \in \mathbb{Z}_{\geq 0}$ the subspace $\hat W_f(z)$ is spanned by the functions $\hat \psi_k$ for $k \in \{1,\ldots,fd\}$. 
\end{lemma}
\begin{proof} The proof is constructive. Define for $k \geq 1$ the rational functions
\begin{equation} \label{eq:defhatpsi}
\hat{\psi}_k(z) = \partial_{x(z)}^{\lceil k/d \rceil} \bigg(\frac{z^{k}}{k^{\lceil k/d \rceil}}\bigg).
\end{equation}
For $k \geq 1$, the function $z \mapsto z^k$ has a pole of order $k$ at $z = \infty$, and each application of $\partial_{x(z)}$ reduces this order by $d$. First note that the functions $\hat{\psi}_k(z)$ belong to $\hat{W}(z)$; in fact, we have applied the operator $\partial_{x(z)}$ just enough times for $\hat{\psi}_{k}$ to be regular at $z = \infty$, and the oddness (with respect to the local involution of the spectral curve) of the principal part near $\alpha \in \mathfrak{a}$ follows from the one for $z^{k}$. Now, observe that the operator $\partial_{x(z)} = z(1 - szQ'(z))^{-1}\partial_{z}$ preserves the order of rational functions at $z = 0$. This can immediately be seen from the expansion at $z \rightarrow 0$ of
\[
\frac{z}{1 - szQ'(z)} = z + O(z^2).
\]
In particular, we have $\hat{\psi}_k(z) = z^{k} + O(z^{k + 1})$ when $z \rightarrow 0$, which clearly implies that the functions $(\hat{\psi}_k)_{k \geq 1}$ are linearly independent. Let us now prove that $(\hat{\psi}_k)_{k \geq 1}$ span $\hat{W}(z)$, and hence that they form a basis. We do so by an inductive argument on the application of the operator $\partial_{x(z)}$.
Consider the filtration on $\hat{W}(z)$ introduced in~\Cref{def:filtration}
\[
\hat W_0(z) \subset \hat W_1(z) \subset \cdots \subset \hat W(z),
\]
where the subspace $\hat W_f(z)$ has dimension $fd$.
In fact we are going to show that elements $\hat{\psi}_k(z)$ for $k \in \{1,\ldots, fd\}$ span $\hat W_f(z)$ inductively on $f$. Let us start with $f=0$. For $k \in \{1,\ldots,d\}$, the $\hat{\phi}_0^{k}(z)$ are expressed in terms of the $\hat{\psi}_k(z)$ by
\[
 \hat{\phi}_0^{k}(z) = \partial_{x(z)} \hat{\phi}_{-1}^{k}(z) = \partial_{x(z)} z^k = k \, \hat{\psi}_k(z), 
 \]
 and therefore $\hat W_0(z)$ is spanned by the $\hat{\psi}_k(z)$.
Let us now deal with $f=1$ by considering the first application of $\partial_{x(z)}$ to $z^k$. Compute
\[
\partial_{x(z)}(z^{k}) = \frac{z \partial_z}{1 - s z Q'(z)} z^k = \frac{ k z^k}{1 - s z Q'(z)}.
\] 
For $k > d$ the degree of the polynomial at the numerator is bigger than the degree of the polynomial at the denominator (which is equal to $d$). Therefore, applying Euclidean division between polynomials, we obtain that
\[
\partial_{x(z)}(z^{k}) = \frac{R(z)}{1 - s z Q'(z)} + P_{k - d}(z).
\] 
for a polynomial $R(z)$ of degree up to $d$ and a polynomial $P_{k - d}$ of degree $k - d$ (in fact, the operator $\partial_{x(z)}$ decreases the degree by $d$). Now, as $\hat{\phi}_0^{l}(z) = l \frac{z^l}{1 - s z Q'(z)}$ for $l \in \{1, \dots, d\}$, we can express $R(z)(1-szQ'(z))^{-1}$ as a linear combination of $\hat{\phi}_0^{l}(z)$, and we obtain
\[
\partial_{x(z)}(z^{k}) = \sum_{l = 1}^{d} c_{l}\,\hat{\phi}_0^{l}(z) + P_{k - d}(z),
\] 
for some scalars $c_1,\ldots,c_d$, which cannot all vanish. If we apply $\partial_{x(z)}$ to the equation above for $d+1 \leq k \leq 2d$, the left-hand side gives a multiple of $\hat \psi_k$, whereas the right-hand side gives a linear combination of elements $\hat \phi_{1}^l$ plus the $x$ derivative of a polynomial of degree up to $d$, which as observed before is a linear combination of elements $\hat \phi_{0}^l$, so $\hat \psi_{d+1}(z),\ldots, \hat \psi_{2d}(z) \in \hat W_1(z)$. The dimension count shows that $\hat \psi_k(z)$ for $k \in \{1,\ldots, 2d\}$ in fact generate $\hat W_1(z)$. 
By successive applications of $\partial_{x(z)}$, this proves inductively that $(\hat{\psi}_k)_{k \geq 1}$ spans $\hat{W}(z)$, hence is a basis. To see that the vector space $W(\mu)$ admits a basis $(\psi_k)_{k \geq 1}$ such that $\psi_{k}(\mu) = \delta_{k,\mu}$ for any $\mu \in \{1,\ldots,k\}$, we simply apply the linear isomorphism of~\Cref{LemmaV} and note that $X(z) = z + O(z^2)$ implies that $\hat{\psi}_{k}(z) = O(X^{k})$.
\end{proof}

\begin{proof}[Proof of~\Cref{lem:symmetry}]
We can write $F(\mu_1, \mu_2, \ldots, \mu_n)$ in the basis provided by~\Cref{newblem}:
\[
	F(\mu_1, \mu_2, \ldots, \mu_n) = \sum_{k=1}^{fd} b_k(\mu_2, \mu_3, \ldots, \mu_n)\,\psi_k(\mu_1),
\]
where $b_k \,\colon (\mathbb Z_{>0})^{n-1} \to \mathbb K$ is a function mapping the tuple $(\mu_2,\ldots,\mu_{n})$ to the corresponding coefficient. To show that $F(\mu_1, \mu_2, \ldots, \mu_n) \in W(\mu_1 ) \otimes W(\mu_2)$ for any $\mu_3,\ldots,\mu_n \in \mathbb{Z}_{> 0}$, it is necessary to show that $b_k(\mu_2, \mu_3, \ldots, \mu_n) \in W(\mu_2)$ for all $\mu_3,\ldots,\mu_n \in \mathbb{Z}_{> 0}$ and $k \in \{1,\ldots,fd\}$. 

Evaluating $F$ at $\mu_1 \in \{1,\ldots,fd\}$ results in the following system of equations.
\begin{equation*}
\begin{split}
    F(1, \mu_2, \ldots, \mu_n) &= \sum_{k=1}^{fd} \psi_k(1)\,b_k(\mu_2, \ldots, \mu_n) \\
    F(2, \mu_2, \ldots, \mu_n) &= \sum_{k=1}^{fd} \psi_k(2)\, b_k(\mu_2,\ldots, \mu_n) \\
    &\hspace{5pt} \vdots \\
    F(fd, \mu_2, \ldots, \mu_n)  &= \sum_{k=1}^{fd} \psi_k(fd)\,b_k(\mu_2,\ldots, \mu_n).
\end{split}
\end{equation*}
We use the symmetry of $F(\mu_1, \mu_2, \ldots, \mu_n)$ in $\mu_1, \mu_2, \ldots, \mu_n$ to write $F(\mu_1, \mu_2, \ldots, \mu_n) = F(\mu_2, \ldots, \mu_n, \mu_1)$, and find the following equivalent system of linear equations.
\begin{equation*}
\begin{split}
	F(\mu_2, \ldots, \mu_n, 1) &= \sum_{k=1}^{fd} \psi_k(\mu_2)\, b_k(\mu_3, \ldots, \mu_n, 1) \\
    F(\mu_2, \ldots, \mu_n, 2) &= \sum_{k=1}^{fd} \psi_k(\mu_2)\,b_k(\mu_3, \ldots, \mu_n, 2) \\
    &\hspace{5pt} \vdots \\
	F(\mu_2, \ldots, \mu_n, fd) &= \sum_{k=1}^{fd} \psi_k(\mu_2)\, b_k(\mu_3, \ldots, \mu_n, fd)
\end{split}
\end{equation*}
Equating the right-hand sides of the two systems above yields
\begin{equation*}
\begin{split}
	\sum_{k=1}^{fd} \psi_k(1)\,b_k(\mu_2, \ldots, \mu_n)  &= \sum_{k=1}^{fd} \psi_k(\mu_2)\, b_k(\mu_3, \ldots, \mu_n, 1)\\
	\sum_{k=1}^{fd} \psi_k(2)\,b_k(\mu_2, \ldots, \mu_n)  &= \sum_{k=1}^{fd} \psi_k(\mu_2)\,b_k(\mu_3, \ldots, \mu_n, 2)\\
	&\hspace{5pt} \vdots \\
	\sum_{k=1}^{fd} \psi_k(fd)\,b_k(\mu_2, \ldots, \mu_n)  &= \sum_{k=1}^{fd} \psi_k(\mu_2)\,b_k(\mu_3, \ldots, \mu_n, fd).
\end{split}
\end{equation*}
We consider it as a linear system of equations for the unknowns $b_k(\mu_2,\ldots,\mu_n)$ with coefficients in $\mathbb{K}$ on the left-hand side, where the right-hand sides are elements of $W(\mu_2)$ for any fixed $\mu_3,\ldots,\mu_n$. By construction of the basis in~\Cref{newblem}, the system is upper-triangular, and therefore for any $\mu_3,\ldots,\mu_n$ it has a unique solution, which must also be in $W(\mu_2)$. We repeat this argument with $\mu_3,\ldots,\mu_n$ successively to conclude that $b_k(\mu_2,\ldots,\mu_n) \in \otimes_{i = 2}^n W(\mu_i)$ and thus that $F(\mu_1,\ldots,\mu_n) \in \otimes_{i = 1}^n W(\mu_i)$.
\end{proof}

\section{Double Hurwitz numbers via the semi-infinite wedge} \label{sec:DHviainfinitewedge}

\subsection{Semi-infinite wedge formalism} \label{subsec:infinitewedge}

We now present the basic tools of the semi-infinite wedge formalism. As this is now a common feature in Hurwitz theory, one can find many good introductions to it in the literature. Given this proviso, our exposition will be brief and we refer the reader to~\cite{oko-pan06} for more details.

Let $V$ be the $\mathbb{C}$-vector space with basis $\{\underline{s} \mid s \in \mathbb{Z} + \frac{1}{2} \}$. The semi-infinite wedge space, denoted $\mathcal{V} = \Lambda ^ \frac{\infty}{2} V$, has a basis defined by
\[ 
v_S := \{\underline{s_1} \wedge \underline{s_2} \wedge \underline{s_3} \wedge \cdots \mid s_1 > s_2 > s_3 > \cdots \}, 
\]
where $S = \{s_1 > s_2 > \cdots \} \subset \mathbb{Z} + \frac{1}{2}$ is such that the sets
\[
S_+ = S \setminus \left(\mathbb{Z}_{\le 0} - \tfrac{1}{2} \right) \quad \text{and} \quad S_- = \left( \mathbb{Z}_{\le 0} - \tfrac{1}{2}\right) \setminus S
\]
are finite. We equip $\mathcal{V}$ with the inner product $\left( \cdot, \cdot\right)$ for which the above basis elements are orthonormal.

There exists a unique $c \in \mathbb{Z}$ such that $s_k + k - 1/2 = c$ for $k$ sufficiently large; this constant $c$ is called the \emph{charge}. The charge-zero subspace, denoted $\mathcal{V}_0 \subset \mathcal{V}$, is spanned by semi-infinite wedge products of the form 
\[
\underline{\lambda_1 - \tfrac{1}{2}} \wedge \underline{\lambda_2 - \tfrac{3}{2}} \wedge \underline{\lambda_3 - \tfrac{5}{2}} \wedge \cdots,
\]
indexed by partitions $\lambda \in \mathscr{P}$. The basis element in $\mathcal{V}_0$ corresponding to the empty partition,
\[
v_{\emptyset} = \underline{-\tfrac{1}{2}} \wedge \underline{-\tfrac{3}{2}} \wedge \underline{-\tfrac{5}{2}} \wedge \cdots,
\]
is called the \emph{vacuum vector} and plays a special role. Similarly, the dual of the vacuum vector with respect to the inner product $\left( \cdot, \cdot \right)$ is called the \emph{covacuum vector}. 

We also define the following operators that will be used in the rest of the paper.
\begin{definition}
For $k \in \mathbb{Z} + \tfrac{1}{2}$, the \emph{fermionic operator} $\psi_k$ is defined by
\[
\psi_k v_S = \underline{k} \wedge v_{S}.
\]
The operator $\psi_k^*$ is defined to be the adjoint of $\psi_k$ with respect to the inner product. The normally ordered product is defined by
\[
: \psi_i \psi_j^*: \ = \begin{cases} 
	\psi_i \psi_j^*, & \text{if } j > 0, \\
	-\psi_j^* \psi_i, & \text{if } j < 0. 
	\end{cases}
\]
\end{definition}

\begin{definition} 
For a non-negative integer $n$, define the operator
\[
\mathcal{F}_n := \sum_{k\in \mathbb{Z} + \frac{1}{2}} \frac{k^n}{n!} :\psi_k \psi_k^* :. 
\]
The operator $\mathcal{F}_1$ is called the \emph{energy operator}. We say that an operator $\mathcal{O}$ acting on $\mathcal{V}_0$ has energy $c \in \mathbb{Z}$ if 
\[
[\mathcal{O}, \mathcal{F}_1] = c \, \mathcal{O}. 
\]
The operators $:\psi_i \psi_j^* :$ have energy $j - i$, while the operators $\mathcal{F}_n$ have energy zero for all $n$.
\end{definition}

We introduce the functions
\begin{equation*}
\begin{split}
\varsigma(z) & := e^{z/2} - e^{-z/2} = 2 \sinh(z/2) = z + \frac{z^3}{24} + \frac{z^5}{1920} + O(z^7), \\
\mathcal{S}(z) & := \frac{\varsigma(z)}{z} = \frac{\sinh(z/2)}{z/2} = 1 + \frac{z^2}{24} + \frac{z^4}{1920} + O(z^6),
\end{split}
\end{equation*}
and record the first terms of the following series expansion, which will subsequently be useful.
\[
\frac{1}{\varsigma(z)} = \frac{1}{z\mathcal{S}(z)} = \frac{1}{z} - \frac{z}{24} + \frac{7z^3}{5760} + O(z^5)
\]

\begin{definition}
For $n \in \mathbb{Z}$ and a formal variable $z$, we define the operator
\[
\mathcal{E}_n(z) := \sum_{k \in \mathbb{Z} + \frac{1}{2}} e^{z(k - \frac{n}{2})} :\psi_{k-n} \psi_k^*: + \frac{\delta_{n,0}}{\varsigma(z)}.
\]
A specialisation of this operator for $n \neq 0$ and $z=0$ defines the \emph{bosonic operator} 
\[
\alpha_n := \mathcal{E}_n(0) = \sum_{k \in \mathbb{Z} + \frac{1}{2}} \, :\psi_{k-n} \psi_k^*:.
\]
These operators satisfy commutation relations
\begin{equation}
\label{Ecomut} [\mathcal{E}_a(z), \mathcal{E}_b(w)] = \varsigma(aw - bz) \, \mathcal{E}_{a+b}(z+w),
\end{equation}
and 
\[
	[\alpha_m, \alpha_n ] = m \delta_{m + n,0}.
\]
\end{definition}

The operators $\mathcal{E}_n(z)$ and $\alpha_n$ have energy $n$. Operators with positive energy annihilate the vacuum, while operators with negative energy annihilate the covacuum.

\subsection{Connected and disconnected correlators}\label{sec:inclexcl}

\begin{definition}
Let $\mathcal{O}_1,\ldots,\mathcal{O}_{n}$ be operators acting on $\mathcal{V}_0$. The \emph{vacuum expectation} or \emph{disconnected correlator} is defined to be
\[
\langle \mathcal{O}_1 \cdots \mathcal{O}_{n} \rangle^{\bullet} := \left( v_{\emptyset}, \mathcal{O}_1\cdots \mathcal{O}_{n} v_{\emptyset} \right)\!.
\]
The \emph{connected correlator} is defined by means of an inclusion-exclusion formula,
\begin{equation}
\label{inclexclO} \langle \mathcal{O}_1\cdots\mathcal{O}_{n} \rangle^{\circ} := \sum_{M \vdash \{1,\ldots,n\}} (-1)^{|M| - 1}(|M| - 1)! \prod_{i = 1}^{|M|} \langle \vec{\mathcal{O}}_{M_i} \rangle^{\bullet},
\end{equation}
where if $I \subseteq \{1,\ldots,n\}$ has the elements $i_1 < \cdots < i_k$, we write $\vec{\mathcal{O}}_{I} = \mathcal{O}_{i_1}\cdots \mathcal{O}_{i_k}$.
\end{definition}
Equivalently, we have
\[
\langle \mathcal{O}_{1}\cdots \mathcal{O}_{n} \rangle^{\bullet} = \sum_{M \vdash \{1,\ldots,n\}} \prod_{i = 1}^{|M|} \langle \vec{\mathcal{O}}_{M_i} \rangle^{\circ}.
\]

We apply the same distinction to the Hurwitz generating series. First, define the generating series
\[
\HHall^{\circ}(\mu_1,\ldots,\mu_n) := \sum_{g \geq 0} \HH_{g,n}(\mu_1,\ldots,\mu_n) \in \mathbb{Q}[q_1,\ldots,q_d][[s]],
\]
which enumerates connected branched covers of~\Cref{doubledef} for all genera. The total power of $s$ in the contribution of a given branched cover is $s^{m}$, where $m$ is the number of simple branch points. We will also need
\begin{equation}
\label{contodis}
\HHall^{\bullet}(\mu_1,\ldots,\mu_n) := \sum_{M \vdash \{1,\ldots,n\}} \prod_{i = 1}^{|M|} \HHall^{\circ}(\vec{\mu}_{M_i}),
\end{equation}
which enumerates possibly disconnected branched covers; we drop the connectedness assumption in~\Cref{doubledef} and declare that the weight is multiplicative under disjoint union. The relation between the connected and disconnected generating series is again an inclusion-exclusion formula:
\begin{equation}
\label{distocon}
\HHall^{\circ}(\mu_1,\ldots,\mu_n) = \sum_{M \vdash \{1,\ldots,n\}} (-1)^{|M| - 1}(|M| - 1)!\,\prod_{i = 1}^{|M|} \HHall^{\bullet}(\vec{\mu}_{M_i}).
\end{equation}

\subsection{Double Hurwitz numbers as vacuum expectation}
\label{Svacc}

We begin with the following classical formula.
\begin{theorem}[\cite{oko-pan06}] \label{thm:DHinIW1} For $n,\mu_1,\ldots,\mu_n$ positive integers, we have
\begin{equation} \label{eq:DHinIW1}
\HHall^{\bullet}(\mu_1, \ldots, \mu_n) = \bigg\langle \exp \Big( \sum_{j = 1}^{d} \frac{q_j}{j} \alpha_j \Big) \, \exp(s{\mathcal F}_2) \, \prod_{i=1}^{n} \frac{\alpha_{-\mu_i}}{\mu_i} \bigg\rangle^{\bullet}.
\end{equation}
\end{theorem}

We rewrite it in a way that will be convenient for the proof of the polynomiality structure.

\begin{proposition} \label{thm:DHinIW2}
For any positive integers $n,\mu_1,\ldots,\mu_n$, we have the identity
\begin{equation} \label{eq:DHinIW2}
\HHall^{\bullet}(\mu_1,\ldots,\mu_n) = \big\langle \mathcal{C}(\mu_1) \, \mathcal{C}(\mu_2) \, \cdots \, \mathcal{C}(\mu_n) \big\rangle^{\bullet},
\end{equation}
where
\[
\mathcal{C}(\mu) := \frac{1}{\mu}\sum_{i \in \mathbb{Z}} \left[ \sum_{\lambda \vdash \mu - i} \frac{\vec{q}_{\lambda} (\mu s)^{\ell(\lambda)}}{|\Aut{\lambda}|} \prod_{k=1}^{\ell(\lambda)} \mathcal{S}\big(\mu \lambda_k s\big) \right] \mathcal{E}_{-i}(\mu s),
\]
and $\vec{q}_{\lambda} := q_{\lambda_1} q_{\lambda_2} \cdots q_{\lambda_{\ell(\lambda)}}$ for $\lambda = (\lambda_1, \lambda_2, \ldots, \lambda_{\ell(\lambda)})$. 
\end{proposition}

\begin{proof}
We begin with \eqref{eq:DHinIW1} and, observing that $\exp(-s\mathcal{F}_2)$ and $\exp\big( -\sum_{j = 1}^{d} \frac{\alpha_j q_j}{j} \big)$ fix the vacuum vector, rewrite it as
\[
\HHall^{\bullet}(\mu_1,\ldots,\mu_n) = \frac{1}{\prod_{i = 1}^n \mu_i} \bigg\langle \prod_{i=1}^{n} \exp \Big( \sum_{j=1}^{d} \frac{q_j}{j} \alpha_j \Big) \, \exp(s{\mathcal F}_2) \, \alpha_{-\mu_i} \, \exp(-s{\mathcal F}_2) \exp \Big( - \sum_{j=1}^{d} \frac{q_j}{j} \alpha_j \Big) \bigg\rangle^{\bullet}.
\]
To compute the conjugations, we iteratively apply Hadamard's lemma,
\begin{equation}
\label{hadamlem} e^A Be^{-A} = B + \sum_{k \geq 1} \frac{1}{k!} [A, [A, \ldots, [A, B] \cdots ]],
\end{equation}
where there are $k$ commutators in the $k$th summand, and use the commutation relation \eqref{Ecomut} for the $\mathcal{E}$-operators.

We first claim that
\begin{equation}
\label{claim12} e^{s\mathcal{F}_2} \alpha_{-\mu} e^{-s\mathcal{F}_2} = \mathcal{E}_{-\mu}(\mu s).
\end{equation}
This can be justified by observing that $\alpha_{-\mu} = \mathcal{E}_{-\mu}(0)$ for $\mu > 0$, $\mathcal{F}_{2} = [z^2]\,\mathcal{E}_{0}(z)$, and using Hadamard's lemma \eqref{hadamlem} and the commutation relation \eqref{Ecomut} for the $\mathcal{E}$-operators. Another way to derive \eqref{claim12} goes by remarking that
$v_{\lambda}$ is an eigenvector for $\mathcal{F}_2$. The eigenvalue $f_2(\lambda)$ is the sum of the contents of the Young diagram given by the partition $\lambda$, where the content of the box in column $j$ and row $i$ is $j-i$. In this case, for any $\lambda \in \mathscr{P}$,
\begin{align*}
e^{s\mathcal{F}_2} \alpha_{-\mu} e^{-s\mathcal{F}_2} v_{\lambda} &= e^{-sf_2(\lambda)} e^{s\mathcal{F}_2} \alpha_{-\mu} v_{\lambda} = e^{-sf_2(\lambda)} e^{s\mathcal{F}_2} \sum_{\lambda^{+\mu}} v_{\lambda^{+\mu}} \\
&= \sum_{\lambda^{+\mu}} e^{s(f_2(\lambda^{+\mu})-f_2(\lambda))} v_{\lambda^{+\mu}} = \sum_{k \in \mathbb{Z}+ \frac{1}{2}} e^{s\mu(k+\frac{\mu}{2})}:\psi_{k+\mu}\psi_k^*: v_{\lambda} \\
&= \mathcal{E}_{-\mu}(\mu s)v_{\lambda},
\end{align*}
where $\lambda^{+\mu}$ is $\lambda$ with an added $\mu$-ribbon. This in fact explains the relevance of the $\mathcal{E}$-operators.

The next conjugation gives
\begin{equation*}
\begin{split}
\exp\Big(\frac{\alpha_d q_d}{d } \Big) \EE_{-\mu}(\mu s) \exp\Big(- \frac{\alpha_d q_d}{d } \Big) & = \sum_{p_d \geq 0} \frac{\varsigma(\mu d s)^{p_d} q_d^{p_d}}{ p_d! \,d^{p_d}} \EE_{-\mu + dp_d}(\mu s) \\
& = \sum_{p_d \geq 0} \frac{(\mu q_d s)^{p_d} \mathcal{S}(\mu d s)^{p_d}}{p_d!} \EE_{-\mu + dp_d}(\mu s).
\end{split}
\end{equation*}
Subsequent conjugations therefore result in
\begin{equation*}
\begin{split}
& \quad \exp\Big( \sum_{j = 1}^{d} \frac{q_j}{j} \alpha_j \Big) \, \EE_{-\mu}(\mu s)\, \exp \Big( - \sum_{j=1}^{d} \frac{q_j}{j} \alpha_j \Big) \\
&= \sum_{p_1 \geq 0} \frac{(\mu q_1 s)^{p_1} \mathcal{S}(\mu s)^{p_1}}{p_1!} \sum_{p_2 \geq 0} \frac{(\mu q_2 s)^{p_2} \mathcal{S}(2\mu s)^{p_2}}{p_2!} \cdots \sum_{p_d \geq 0} \frac{(\mu q_d s)^{p_d} \mathcal{S}(d\mu s)^{p_d}}{p_d!} \  \EE_{-\mu + p_1 + 2p_2 + \cdots + dp_d} (\mu s)\\
&= \sum_{\lambda \in \mathscr{P}} \frac{ \vec{q}_{\lambda} (\mu s)^{\ell(\lambda)}}{|\Aut{\lambda}|} \prod_{k=1}^{\ell(\lambda)} \mathcal{S}(\mu \lambda_k s) \ \EE_{-\mu + |\lambda|} (\mu s),
\end{split}
\end{equation*}
where the tuple $(p_1,\ldots,p_d)$ has been realised by the partition $\lambda = (1^{p_1}\cdots d^{p_d})$ and recall the definition $|{\rm Aut}(\lambda)| = p_1!\cdots p_d!$. Letting $-i = |\lambda| - \mu$ yields the result.
\end{proof}

\section{Polynomiality results} \label{sec:poly}

The aim of this section is to prove the polynomiality~\Cref{thm:poly}. The proof is divided into four parts, each being carried out in a subsection of its own.
\begin{enumerate}
	\item We begin with the vacuum expectation~\eqref{eq:DHinIW2} for $\HHall^{\bullet}(\mu_1,\ldots, \mu_n)$ and consider the dependence on $\mu_1$. This yields an expression involving a sum over partitions of size $\mu_1 + a$ for a positive integer $a$. We then use a ``peeling lemma'', \Cref{lem:peel}, to reduce the expression to a sum over partitions of size $\mu_1 - r$ for $r \in \{1,2, \ldots, d\}$ (\Cref{lem:DHbehavmu1}). 
	\item The use of the peeling lemma in step (1) results in a rational expression with distinct factors of the form $\frac{\mu_1}{\mu_1 + k}$ for $k$ a non-negative integer. We study the residues of $\HHall^{\bullet}(\mu_1,\ldots,\mu_n)$ at $\mu_1 = -k$ for $k \geq 0$ in Lemmata~\ref{lem:residueofBop} and~\ref{lem:poleatzero}.
	\item We then show in~\Cref{thm:polyDH} that $\HHall^{\circ}(\mu_1,\ldots,\mu_n)$ is of the form
	\[
    \HHall^{\circ}(\mu_1,\ldots,\mu_n) = \sum_{r=1}^{d} \sum_{\lambda \vdash \mu_1-r} \frac{\vec{q}_{\lambda} (\mu_1 s)^{\ell(\lambda)}}{|\Aut{\lambda}|} \sum_{g \geq 0} P_{g}^{r;\mu_2, \ldots, \mu_n} (\mu_1),
	\]
	where, for $2g - 2 + n > 0$ and any fixed $r,\mu_2,\ldots,\mu_n$, the function $P_{g}^{r;\mu_2,\ldots,\mu_n}$ is a polynomial of $\mu_1$ whose degree is uniformly bounded in terms of $(g,n)$.
\item We eventually extract the genus $g$ part and use~\Cref{lem:symmetry} to conclude the proof.
\end{enumerate}

\subsection{Dependence on $\mu_1$} \label{subsec:mu1dependence}

Let us first study the dependence of $\HHall^{\bullet}(\mu_1, \ldots, \mu_n)$ on $\mu_1$, fixing $\mu_2, \mu_3, \ldots, \mu_n$ to be positive integers. As detailed in the outline above, the dependence on $\mu_1$ yields an expression involving a sum over partitions of size $\mu_1 + a$, whereas we want to reduce it to a sum over partitions of size $\mu_1 - r$. To this end, we will make use of the following lemma. 

\begin{lemma}[Peeling lemma] \label{lem:peel}
For any $\mu,a \in \mathbb{Z}$ with $\mu + a > 0$ we have
\begin{equation*} \label{eq:peel}
\sum_{\lambda \vdash \mu + a} \frac{\vec{q}_{\lambda} (\mu s)^{\ell(\lambda)}}{|\Aut{\lambda}|} \prod_{k=1}^{\ell(\lambda)} \mathcal{S}(\mu \lambda_k s) = \frac{\mu}{\mu + a} \sum_{r=1}^{d} rq_rs\,\mathcal{S}(\mu r s) \sum_{\lambda \vdash \mu + a - r} \frac{\vec{q}_{\lambda} (\mu s)^{\ell(\lambda)}}{|\Aut{\lambda}|} \prod_{k=1}^{\ell(\lambda)} \mathcal{S}(\mu \lambda_k s). 
\end{equation*}
\end{lemma}

\begin{proof}

Recall that $Q(z) = \sum_{j = 1}^d q_jz^j$ and observe that 
\[
z\partial_{z}e^{\mu sQ(z)} = \bigg(\sum_{j=1}^{d} j q_j \partial_{q_j}\bigg) e^{\mu sQ(z)}.
\]

Let us extract the coefficient of $z^{\mu + a}$ on both sides. On the left-hand side, we find
\begin{equation*}
\begin{split}
    [z^{\mu + a}]\,z\partial_{z} e^{\mu sQ(z)} &= [z^{\mu + a}]\,z\partial_{z} \left[\,\sum_{\lambda \in \mathscr{P}} \frac{(\mu s)^{\ell(\lambda)} q_{\lambda_1} q_{\lambda_2} \cdots q_{\lambda_{\ell(\lambda)}} \, z^{|\lambda|}}{|\Aut{\lambda}|}  \right]  \\
    & = \sum_{\lambda \vdash \mu + a} \frac{(\mu s)^{\ell(\lambda)} \vec{q}_{\lambda}\, (\mu + a)}{|\Aut{\lambda}|}.
\end{split}
\end{equation*}
while on the right-hand side, we obtain
\begin{equation*}
\begin{split}
    [z^{\mu + a}]\bigg( \sum_{j=1}^{d} jq_j \partial_{q_j}\bigg) e^{\mu sQ(z)}  &= [z^{\mu + a}]\,\mu s \bigg(\sum_{j = 1}^{d} jq_jz^{j} \bigg) \sum_{\lambda \in \mathscr{P}} \frac{(\mu s)^{\ell(\lambda)} \vec{q}_{\lambda} z^{|\lambda|} } {|\Aut{\lambda}|} \\
    &= \mu \sum_{r=1}^{d} rq_rs \sum_{\lambda \vdash \mu + a - r} \frac{(\mu s)^{\ell(\lambda)}\,\vec{q}_{\lambda}}{|\Aut{\lambda}|}. 
\end{split}
\end{equation*} 
This yields 
\[
(\mu + a) \sum_{\lambda \vdash \mu + a} \frac{\vec{q}_{\lambda} (\mu s)^{\ell(\lambda)}}{|\Aut{\lambda}|} = \mu \sum_{r=1}^{d} rq_rs \sum_{\lambda \vdash \mu + a - r} \frac{\vec{q}_{\lambda} (\mu s)^{\ell(\lambda)}}{|\Aut{\lambda}|}.
\]
We apply the rescaling $q_{j} \mapsto q_{j}\mathcal{S}(\mu j s)$ and divide by $\mu + a$ on both sides to obtain the desired result.
\end{proof}

The following sets are also required, to be utilised in~\Cref{lem:DHbehavmu1}. 

\begin{definition} \label{def:admissible}
 Define $\mathcal{P}_{d}$ to be the set of finite (possibly empty) sequences $P = (p_1,\ldots,p_{\ell(P)})$ such that $1 \leq p_i \leq d$ for all $i \in \{1,\ldots,\ell(P)\}$. Fix $r \in \{1,\ldots,d\}$ and let $\mathcal{P}_{d,r}$ be the set of all non-empty finite sequences $P \in \mathcal{P}_{d}$ such that $p_{\ell(P)} \geq r$. For a sequence $P \in \mathcal{P}_{d}$, we denote the sum of its terms by $e(P)$. We also denote $\mathcal{P}_{d}(e) \subset \mathcal{P}_{d}$ the subset consisting of sequences with $e(P) = e$.

\end{definition}

\begin{lemma} \label{lem:DHbehavmu1}
Fix positive integers $n,\mu_2,\ldots,\mu_n$. The double Hurwitz generating series can be written as
\begin{equation} \label{eq:DHbehavmu1}
	\HHall^{\bullet}(\mu_1,\ldots,\mu_n) = \sum_{r=1}^{d} \sum_{\lambda \vdash \mu_1 - r} \left[ \frac{\vec{q}_{\lambda} (\mu_1 s)^{\ell(\lambda)}}{|\Aut{\lambda}|} \prod_{k=1}^{\ell(\lambda)} \mathcal{S}(\mu_1 \lambda_k s) \right] \big\langle \mathcal{B}_r(\mu_1)\,\mathcal{C}(\mu_2) \cdots \mathcal{C}(\mu_n)\big\rangle^{\bullet},
\end{equation}
where
\[
\mathcal{B}_r(\mu) := \frac{1}{\mu} \sum_{P \in \mathcal{P}_{d,r}} \left[ \prod_{i=1}^{\ell(P)} \frac{\mu}{\mu - r + \sum_{j=i}^{\ell(P)} p_i} \, p_i q_{p_i}s\, \mathcal{S}(\mu p_i s)\right] \EE_{-r + e(P)}(\mu s).
\]
\end{lemma}

\begin{proof}
We begin with expression \eqref{eq:DHinIW2} for the double Hurwitz generating series given in~\Cref{thm:DHinIW2}:
\[
\HHall^{\bullet}(\mu_1,\ldots,\mu_n) = \big\langle \mathcal{C}(\mu_1)\,\mathcal{C}(\mu_2) \cdots \mathcal{C}(\mu_n) \big\rangle^{\bullet},
\]
where
\[
\mathcal{C}(\mu) = \frac{1}{\mu}\sum_{i \in \mathbb{Z}} \left[ \sum_{\lambda \vdash \mu - i} \frac{\vec{q}_{\lambda} (\mu s)^{\ell(\lambda)}}{|\Aut{\lambda}|} \prod_{k=1}^{\ell(\lambda)} \mathcal{S}\big(\mu \lambda_k s\big) \right] \mathcal{E}_{-i}(\mu s).
\]
We observe that, for the vacuum expectation to return a non-zero value, the energies of the $\mathcal{E}$-operators must sum to zero; that is, $i_1 + i_2 + \cdots + i_n = 0$. Hence, $i_1 = -i_2 - i_3 - \cdots - i_n$. Further, given that $\mathcal{E}_{-i_1}(\mu_1 s)$ is acting on the covacuum, $i_1$ has an upper bound $i_1 \leq 0$. Finally, for each $j \in \{1,\ldots, n\}$, the relation $-i_j = - \mu_j + |\lambda^{(j)}|$ gives a bound $-i_j \geq -\mu_j$, and thus $i_1 \geq -\mu_2 - \mu_3 - \cdots - \mu_n$. Given that $\mu_2, \mu_3, \ldots, \mu_n$ are fixed positive integers, this establishes a lower bound for $i_1$. Despite the fact that the sum over $a = -i_1$ is finite, we write
\[ 
	\HHall^{\bullet}(\mu_1, \ldots, \mu_n) = \frac{1}{\mu_1} \sum_{a\ge0}\left[ \sum_{\lambda \vdash \mu_1 + a} \frac{\vec{q}_{\lambda} (\mu_1 s)^{\ell(\lambda)}}{|\Aut{\lambda}|} \prod_{k=1}^{\ell(\lambda)} \mathcal{S}(\mu_1 \lambda_k s) \right] \Big\langle \mathcal{E}_{a}(\mu_1 s) \prod_{j=2}^{n} \mathcal{C}(\mu_j) \Big\rangle^{\bullet}.
\]
We now apply~\Cref{lem:peel} repeatedly to decrease $\mu_1+a$ to $\mu_1 - r$ for some $r \in \{1,\ldots,d\}$. The first application turns the term in the brackets to 
\[
	 \sum_{p_1 = 1}^{d} \frac{\mu_1}{\mu_1 + a} \, p_1 q_{p_1}s\, \mathcal{S}(\mu_1 p_1 s) \left[ \sum_{\lambda \vdash \mu_1 + a - p_1} \frac{\vec{q}_{\lambda} (\mu_1 s)^{\ell(\lambda)}}{|\Aut{\lambda}|} \prod_{k=1}^{\ell(\lambda)} \mathcal{S}(\mu_1 \lambda_k s) \right].
\]
Apply~\Cref{lem:peel} again, this time with the shift $a \mapsto a - p_1$ and leave summands corresponding to negative $a - p_1$ unchanged. Thus the second application will include terms of the form
\[
	\frac{1}{\mu_1} \, \frac{\mu_1}{(\mu_1 + a)} \, \frac{\mu_1}{(\mu_1 + a - p_1)} \, p_1 q_{p_1}s\,\mathcal{S}(\mu_1 p_1 s) \, p_2 q_{p_2}s\,\mathcal{S}(\mu_1 p_2 s) \left[\sum_{\lambda \vdash \mu_1 + a - p_1 - p_2} \frac{\vec{q}_{\lambda} (\mu_1 s)^{\ell(\lambda)}}{|\Aut{\lambda}|} \prod_{k=1}^{\ell(\lambda)} \mathcal{S}(\mu_1 \lambda_k s)\right].
\]
We repeat this process until $a - p_1 - p_2 - \cdots - p_{\ell(P)}$ is negative for all terms in the summation. The set of all possible sequences $(p_1, p_2, \ldots, p_{\ell(P)})$ that arise in this way for all possible $a$ is exactly the set $\mathcal{P}_{d,r}$ in~\Cref{def:admissible} (and summing over this set replaces the sum over $a$). The condition that each $p_i$ satisfies $1 \leq p_i \leq d$ ensures that the peeling process will terminate after a finite number of iterations. The condition $p_{\ell(P)} \geq r$ is implied by the stopping condition of the algorithm, which is given by the two inequalities
\begin{align*}
 a - p_1 - p_2 - \cdots - p_{\ell(P)-1} &\geq 0, \\
 a - p_1 - p_2 - \cdots - p_{\ell(P)-1} - p_{\ell(P)} &< 0.
\end{align*}
This concludes the proof of the lemma. 
\end{proof}

\subsection{Calculating the residue} \label{subsec:calculatingresidue}
Observe that in the expression for $\HHall^{\bullet}(\mu_1, \ldots,\mu_n)$ in~\Cref{lem:DHbehavmu1}, for fixed $\mu_2,\ldots,\mu_n$ the constraints on the set of possible energies dictate that the operator $\mathcal{B}_r(\mu_1)$ can be replaced by finite linear combination of $\EE$ operators whose coefficients are power series in $s$ without changing the value of the vacuum expectation. Further, for each fixed power of $s$, its coefficient is a rational function in $\mu_1$ (as well as in the parameters $q_1,\ldots,q_d$ which remain spectators here). Therefore, the notion of poles of $\mathcal{B}_r(\mu_1)$ is well-defined. We first note that the factors in the term 
\[
\prod_{i=1}^{\ell(P)} \frac{\mu_1}{\mu_1 - r + \sum_{l=i}^{\ell(P)} p_l}
\]
create at most simple poles when $\mu_1$ hits negative integers. To ultimately show that $\HH_{g,n}(\mu_1,\ldots, \mu_n)$ satisfies a polynomiality structure, we study the residue of $\HHall^{\bullet}(\mu_1, \ldots, \mu_n)$ at $\mu_1 = -b$ for positive integers $b$; the outcome is presented in~\Cref{lem:residueofBop} below.

A pole at $\mu_1 = 0$ can only occur in the summand corresponding to the sequence $P = (p_1 = r)$; in this case the factor $1/\mu_1$ in $\mathcal{B}_r(\mu_1, s)$ introduces a simple pole at $\mu_1 = 0$, and further, as this term must involve an $\EE$-operator with zero energy, the evaluation of $\EE_0(\mu_1 s)$ on the covacuum makes it a second order pole at $\mu_1 = 0$. However, it can be shown that, for $n\geq 2$, terms of this form get cancelled out when passing from disconnected to connected correlators; this is the content of ~\Cref{lem:poleatzero}.

\begin{lemma} \label{lem:residueofBop}
Fix $n \geq 2$ and positive integers $\mu_2, \ldots, \mu_n$. For all $b > 0$ and $r \in \{1, 2, \ldots, d\}$, we have
\begin{equation*}
\begin{split}
	& \quad \Res_{\mu_1 = -b} {\rm d}\mu_1\,\Big\langle \mathcal{B}_{r}(\mu_1) \prod_{j=2}^{n} \mathcal{C}(\mu_j) \Big\rangle^{\bullet} \\
	& = M_r(b) \bigg\langle \exp\Big(\sum_{j=1}^{d} \frac{\alpha_j q_j}{j} \Big) \exp(s\mathcal{F}_2) \alpha_{b} \exp(- s\mathcal{F}_2) \exp\Big( -\sum_{j=1}^{d} \frac{\alpha_j q_j}{j} \Big) \prod_{j=2}^{n} \mathcal{C}(\mu_j) \bigg\rangle^{\bullet},
\end{split}
\end{equation*}
where 
\[
M_r(b) := \sum_{P' \in \mathcal{P}_{d,r}} (-b)^{\ell(P')-1} \frac{ \prod_{i=1}^{\ell(P')} p'_i q_{p'_i}s\,\mathcal{S}(b p'_i s) }{\prod_{i=2}^{\ell(P')} (-b - r + \sum_{l=i}^{\ell(P')} p'_l)}.
\]
\end{lemma}

\begin{proof}
As observed at the start of this subsection, $\HHall^{\bullet}(\mu_1, \ldots, \mu_n)$ has at most simple poles at $\mu_1 = -b$ for positive integers $b$. Hence, calculating the residue at $-b$ by multiplying $\big\langle \mathcal{B}_{r}(\mu_1) \prod_{j=2}^n \mathcal{C}(\mu_j) \big\rangle^{\bullet}$ by $(\mu_1+b)$ then setting $\mu_1=-b$, we find
\begin{equation*}
\begin{split}
& \quad	 \Res_{\mu_1 = -b} {\rm d}\mu_1\,\Big\langle \mathcal{B}_{r}(\mu_1) \prod_{j=2}^{n} \mathcal{C}(\mu_j) \Big\rangle^{\bullet} \\
& = \lim_{\mu_1 = -b} \left( \frac{\mu_1 + b}{\mu_1} \sum_{P \in \mathcal{P}_{d,r}} \left[ \prod_{i=1}^{\ell(P)} \frac{\mu_1}{\mu_1 - r + \sum_{l=i}^{\ell(P)} p_l} \, p_i q_{p_i}s\,\mathcal{S}(\mu_1 p_i s)\right]\,\Big \langle \EE_{-r + e(P)}(\mu_1 s) \prod_{j=2}^{n} \mathcal{C}(\mu_j) \Big\rangle^{\bullet} \right).
\end{split}
\end{equation*}
This is non-zero only when $b = -r + p_k + p_{k+1} + \cdots + p_{\ell(P)}$ for some $k \in \{ 1, 2, \ldots, \ell(P)\}$. Cancelling this factor with $(\mu_1 + b)$, substituting in $\mu_1 = -b$ in the remaining expression, and noting that $\mathcal{S}$ is an even function yields
\begin{multline*}
\Res_{\mu_1 = -b} \dd\mu_1\,\Big\langle \mathcal{B}_{r}(\mu_1) \prod_{j=2}^{n} \mathcal{C}(\mu_j) \Big\rangle^{\bullet}
	= \sum_{\substack{P \in \mathcal{P}_{d,r} \\ b+r=\sum_{l=k}^{\ell(P)} p_l}} \!\!\!\!\!\! (-b)^{\ell(P)-1} \prod_{i=1}^{k-1} \frac{p_i q_{p_i}s\,\mathcal{S}(bp_i s)}{-b - r + \sum_{l=i}^{\ell(P)} p_l} \cdot p_k q_{p_k}s\,\mathcal{S}(b p_k s) \\
\prod_{i=k+1}^{\ell(P)} \frac{p_i q_{p_i}s\,\mathcal{S}(b p_i s)}{-b - r + \sum_{l=i}^{\ell(P)} p_l} \Big \langle \EE_{-r + e(P)}(-b s) \prod_{j=2}^{n} \mathcal{C}(\mu_j) \Big\rangle^{\bullet}.
	\end{multline*}
Let us rewrite the denominators of the fractions in the first product using the condition on $b$; that is, $b = - r + p_k + p_{k+1} + \cdots + p_{\ell(P)}$. We split $P = (p_1, p_2, \ldots, p_{\ell(P)})$ into two sequences $P' := (p_k, p_{k+1}, \ldots, p_{\ell(P)}) = (p'_1, p'_2, \ldots, p'_{\ell(P')}) $ and $P'' := (p_{k-1}, p_{k-2}, \ldots, p_1) = (p''_1, p''_2, \ldots, p''_{\ell(P'')})$ so that the sequence $P$ is the concatenation of $P''$ and $P'$. We notice that $P' \in \mathcal{P}_{d,r}$, while $P'' \in \mathcal{P}_{d}$, and $-r + e(P) = b + e(P'')$. Therefore,
\begin{equation*}
\begin{split}
& \quad 	\Res_{\mu_1 = -b} {\rm d}\mu_1\,\Big\langle \mathcal{B}_{r}(\mu_1) \prod_{j=2}^{n} \mathcal{C}(\mu_j) \Big\rangle^{\bullet} \\
& = \left[ \sum_{P' \in \mathcal{P}_{d,r}} (-b)^{\ell(P')-1} \frac{ \prod_{i=1}^{\ell(P')} p'_i q_{p'_i}s\,\mathcal{S}(b p'_i s) }{\prod_{i=2}^{\ell(P')} -b - r + \sum_{l=i}^{\ell(P')} p'_l} \right] \\
& \quad \quad \times \sum_{P'' \in \mathcal{P}_{d}} (-b)^{\ell(P'')} \prod_{i=1}^{\ell(P'')} \frac{p''_i q_{p''_i}s\,\mathcal{S}(b p''_i s)}{ \sum_{l=1}^{i} p''_l} \Big \langle \EE_{b + e(P'')}(-b s) \prod_{j=2}^{n} \mathcal{C}(\mu_j) \Big\rangle^{\bullet}.
\end{split}
\end{equation*}
On the other hand, we use the same technique as in~\Cref{thm:DHinIW2} for $b \geq 1$ to calculate
\begin{equation*}
\begin{split}
& \quad \exp\Big(\sum_{j = 1}^{d} \frac{\alpha_j q_j}{j} \Big) \exp(s\mathcal{F}_2) \alpha_{b} \exp(- s\mathcal{F}_2) \exp\Big( -\sum_{j=1}^{d} \frac{\alpha_j q_j}{j} \Big) \\
& = \sum_{\lambda \in \mathscr{P}} \left[ \frac{\vec{q}_{\lambda} (-b s)^{\ell(\lambda)} }{|\Aut{\lambda}|} \prod_{j=1}^{\ell(\lambda)} \mathcal{S}(b\lambda_j s)\right] \EE_{b+|\lambda|} (-bs) \\
&= \sum_{i \geq b} \left[\sum_{\lambda \vdash -b+i} \frac{\vec{q}_{\lambda} (-bs)^{\ell(\lambda)} }{|\Aut{\lambda}|} \prod_{j=1}^{\ell(\lambda)} \mathcal{S}(b\lambda_j s)\right] \EE_{i} (-bs).
\end{split}
\end{equation*}
We now apply~\Cref{lem:peel} iteratively to reduce $-b + i$ to 0. Writing separately for the moment the term $i = b$, the first application with $(\mu,a) = (-b,i)$ gives
\begin{multline*}
\sum_{i > b} \left[\sum_{\lambda \vdash -b+i} \frac{\vec{q}_{\lambda} (-b s)^{\ell(\lambda)} }{|\Aut{\lambda}|} \prod_{j=1}^{\ell(\lambda)} \mathcal{S}(b\lambda_j s)\right] \EE_{i} (-bs) \\
 = \sum_{i > b} \left[\sum_{p_1=1}^d \frac{-b}{-b+i} \, p_1 q_{p_1}s\,\mathcal{S}(b p_1 s) \sum_{\lambda \vdash -b+i-p_1} \frac{\vec{q}_{\lambda} (-bs)^{\ell(\lambda)} }{|\Aut{\lambda}|} \prod_{j=1}^{\ell(\lambda)} \mathcal{S}(b\lambda_j s)\right] \EE_{i} (-bs).
\end{multline*}
Repeating the process as many times as necessary yields a sum over the sequences $P = (p_1, p_2, \ldots, p_{\ell(P)})$ such that $p_1 + p_2 + \cdots + p_{\ell(P)} = -b + i$ for all possible $i > b$. We can take into account the missing term $i = b$ by allowing $P$ to be the empty sequence. As per~\Cref{def:admissible}, all possible such sequences comprise the set $\mathcal{P}_{d}(-b+i)$. Thus, applying the process above iteratively gives
\begin{equation*}
\begin{split}
& \quad 	\sum_{i \geq b} \sum_{\lambda \vdash -b+i} \frac{\vec{q}_{\lambda} (-bs)^{\ell(\lambda)} }{|\Aut{\lambda}|} \prod_{j=1}^{\ell(\lambda)} \mathcal{S}(b\lambda_j s) \EE_{i} (-bs) \\
	&= \sum_{i \geq b} \sum_{P \in \mathcal{P}_{d}(-b+i)} \frac{-b}{-b+i} \frac{-b}{-b+i-p_1} \cdots \frac{-b}{-b+i-p_1-\cdots-p_{\ell(P)-1}} \prod_{j=1}^{\ell(P)} p_j q_{p_j}s\,\mathcal{S}(b p_j s) \\
	& \qquad \qquad \left[ \sum_{ \lambda= \emptyset} \frac{\vec{q}_{\lambda} (-bs)^{\ell(\lambda)}}{|\Aut{\lambda}|} \prod_{j=1}^{\ell(\lambda)} \mathcal{S}(b \lambda_j s) \right] \ \EE_{i}(-bs) \\
	&= \sum_{P\in \mathcal{P}_{d}} (-b)^{\ell(P)} \prod_{j=1}^{\ell(P)} \frac{ p_j q_{p_j}s\,\mathcal{S}(b p_j s)}{\sum_{l=1}^{i} p_l } \EE_{b+e(P)}(-bs),
\end{split} 
\end{equation*}
where the last line uses the relabelling $(p_1, p_2, \ldots, p_{\ell(P)}) \mapsto (p_{\ell(P)}, p_{\ell(P)-1}, \ldots, p_2, p_1)$.
\end{proof}

\begin{remark} From the proof, we note that the result of~\Cref{lem:residueofBop} holds true also when replacing the product of the $\mathcal{C}$-operators with a general semi-infinite wedge operator $\mathcal{O}$ not depending on $\mu_1$. In particular, substituting in $\mathcal{O} = 1$ shows that~\Cref{lem:residueofBop} is also valid for $n=1$.
\end{remark}

It remains to consider the double pole at $\mu_1 = 0$, which is treated by the following lemma.

\begin{lemma} \label{lem:poleatzero}
Fix $n \geq 2$ and positive integers $\mu_2, \mu_3, \ldots, \mu_n$. For any $r \in \{1, 2, \ldots, d\}$ we have
\[
	 \Res_{\mu_1 = 0} \dd \mu_1\,\mu_1 \Big\langle \mathcal{B}_{r}(\mu_1) \prod_{j=2}^{n} \mathcal{C}(\mu_j) \Big\rangle^{\bullet} = \Res_{\mu_1=0} \dd\mu_1\, \mu_1 \Big\langle \mathcal{B}_r(\mu_1) \Big\rangle^{\bullet} \, \Big\langle \prod_{j=2}^{n} \mathcal{C}(\mu_j) \Big\rangle^{\bullet}.
\]
Hence, applying the inclusion-exclusion formula,
\[
	\Res_{\mu_1 = 0} \dd\mu_1\,\mu_1 \Big\langle \mathcal{B}_{r}(\mu_1) \prod_{j=2}^{n} \mathcal{C}(\mu_j) \Big\rangle^{\circ} = 0.
\]
\end{lemma}
\begin{proof}
As stated at the start of this section, a pole at zero can only occur in the summand of $\mathcal{B}_{r}(\mu_1)$ corresponding to $P = (p_1 = r)$, whose contribution is equal to
\[
r q_rs\,\mathcal{S}(\mu_1 r s)\,\frac{\EE_0(\mu_1 s)}{\mu_1}.
\]
Observe that the expression has no residue at zero, because $\mathcal{S}$ is an even function. We compute the coefficient of the double pole as
\begin{equation}
\label{slngfug}\begin{split}
\Res_{\mu_1 = 0} \dd\mu_1\,\mu_1 \Big\langle \mathcal{B}_{r}(\mu_1) \prod_{j=2}^{n} \mathcal{C}(\mu_j) \Big\rangle^{\bullet} &= \Res_{\mu_1 = 0} \dd\mu_1\,r q_rs\,\mathcal{S}(\mu_1 r s) \Big\langle \EE_0(\mu_1 s) \prod_{j=2}^{n} \mathcal{C}(\mu_j) \Big\rangle^{\bullet} \\
&= \Res_{\mu_1 = 0} \dd\mu_1\,r q_rs\,\mathcal{S}(\mu_1 r s) \Big\langle \EE_0(\mu_1 s) \Big\rangle^{\bullet} \Big\langle \prod_{j=2}^{n} \mathcal{C}(\mu_j) \Big\rangle^{\bullet} \\
&= \Res_{\mu_1=0}\dd\mu_1\,\mu_1 \Big\langle \mathcal{B}_r(\mu_1) \Big\rangle^{\bullet} \, \Big\langle \prod_{j=2}^{n} \mathcal{C}(\mu_j) \Big\rangle^{\bullet},
\end{split}
\end{equation}
where we used the fact that for any operator $\mathcal{O}$ we have $\langle \mathcal{E}_0(z) \, \mathcal{O}\rangle^{\bullet}= \varsigma^{-1}(z) \langle \mathcal{O}\rangle^{\bullet}$. This proves the first sentence of the statement. Now applying the residue to the inclusion-exclusion formula \eqref{inclexclO} and using \eqref{slngfug}, 
\begin{multline*}
\qquad \Res_{\mu_1 = 0} \dd \mu_1\,\mu_1\,\Big\langle \mathcal{B}_{r}(\mu_1) \prod_{j = 2}^n \mathcal{C}(\mu_j)\Big\rangle^{\circ} \\
= \Res_{\mu_1 = 0} \dd\mu_1\,\mu_1 \Bigg(\sum_{\substack{N \subseteq \{2,\ldots,n\} \\ M \vdash \{2,\ldots,n\} \setminus N}} (-1)^{|M|}|M|!\,\langle \mathcal{B}_{r}(\mu_1)\rangle^{\bullet}\langle \mathcal{C}(\vec{\mu}_{N})\rangle^{\bullet} \prod_{i = 1}^{|M|} \langle \mathcal{C}(\vec{\mu}_{M_i}) \rangle^{\bullet}\Bigg). \qquad
\end{multline*}
This can be rewritten as a sum over $M' \vdash \{2,\ldots,n\}$ of products of the form $\prod_{i = 1}^{M'} \langle \mathcal{C}(\vec{\mu}_{M'_i}) \rangle$ with $M' = M \cup \{N\}$. Such terms arise in two ways: either by $N = \emptyset$, in which case it arises with a factor $(-1)^{|M'|}|M'|!$; or from $N = M_i'$ for some $i \in \{1,\ldots,|M'|\}$ and in this case it comes with a factor $(-1)^{|M'| - 1}(|M'| - 1)!$. Therefore, all terms cancel in the sum as soon as $n \geq 2$.
\end{proof}

\subsection{Polynomiality for $\mu_1$}

Let us first deal with the extraction of the power of $s$ from the $\mathcal{S}$ functions. If $G \in \mathscr{R}$ is a symmetric function and $\lambda \in \mathscr{P}$ a partition, we will use the evaluation at finitely many variables $G(\lambda) := {\rm ev}_{\ell(\lambda)}G(\lambda_1,\ldots,\lambda_{\ell(\lambda)})$.
\begin{lemma} \label{lem:coeffofS}
For any $a \geq 0$ and $r \in \{1,\ldots,d\}$, there exists $G_{a} \in \mathscr{R}$ such that
\[
\sum_{\lambda \vdash \mu-r} \frac{\vec{q}_{\lambda} (\mu s)^{\ell(\lambda)}}{|\Aut{\lambda}|}\,[s^{2a}] \bigg(\prod_{k=1}^{\ell(\lambda)} \mathcal{S}(\mu \lambda_k s)\bigg) = \sum_{\lambda \vdash \mu-r} \frac{\vec{q}_{\lambda} (\mu s)^{\ell(\lambda)}\,\mu^{2a}}{|\Aut{\lambda}|}\,G_{a}(\lambda).
\]
Further, we have for any $c \geq 0$
\[
\bigg(\mu \mapsto \sum_{\lambda \vdash \mu-r} \frac{\vec{q}_{\lambda} (\mu s)^{\ell(\lambda)}\,\mu^{c}}{|\Aut{\lambda}|}\,G_{a}(\lambda) \bigg)\in W_{c + a}(\mu).
\]
\end{lemma}
\begin{proof}
First, the fact that extracting the coefficient of $s^{2a}$ from the product $\prod_{k=1}^{\ell(\lambda)} \mathcal{S}(\mu \lambda_k s)$ yields a symmetric function of $\lambda$ is clear by expanding each function $\mathcal{S}$ as a power series in $s$. 
For $m > 0$, let us define the operator
\[
\hat{\mathfrak{p}}_{m} := \sum_{j = 1}^{d} j^{m} q_j\partial_{q_j}.
\]
For $r \in \{1,\ldots,d\}$, the action of $\hat{\mathfrak{p}}_{m}$ on the elements $\phi_{0}^{r} \in W(\mu)$ yields
\[
\hat{\mathfrak{p}}_{m} \phi_0^{r}(\mu) = r \sum_{\lambda \vdash \mu - r} \frac{(\mu s)^{\ell(\lambda)}\,\vec{q}_{\lambda}}{|{\rm Aut}(\lambda)|}\,\mathfrak{p}_{m}(\lambda),
\]
and therefore,
\begin{equation}
\label{hatpmm}\hat{\mathfrak{p}}_{m_1}\cdots \hat{\mathfrak{p}}_{m_a} \phi_0^{r}(\mu) = r \sum_{\lambda \vdash \mu - r} \frac{(\mu s)^{\ell(\lambda)}\,\vec{q}_{\lambda}}{|{\rm Aut}(\lambda)|}\,\prod_{i = 1}^{a} \mathfrak{p}_{m_i}(\lambda).
\end{equation}
The derivations $q_j\partial_{q_j}$ preserve the vector space $W(\mu)$, move elements at most once up in the filtration (\Cref{lem:dmodq}), and commute with $\partial_{x}$. Furthermore, $\phi_0^{1},\ldots,\phi_0^d$ generate $W(\mu)$ as a $\mathbb{C}[\partial_{x}]$-module; hence, we deduce that both sides of \eqref{hatpmm} belong to $W_{a}(\mu)$.

In order to prove the last part of the lemma, we note, that the application of the operator $q_j \partial_{q_j}$ multiplies each monomial of the form
\[
\frac{(\mu s)^{\ell(\lambda)}\,\vec{q}_{\lambda}}{|{\rm Aut}(\lambda)|}
\]
by the number $p_j$ of length $j$ parts of $\lambda$. For each $j \geq 1$, the numbers of parts $p_j$ appears as the power of the factor $\mathcal S(\mu j s)$ in the expression
\[
\sum_{\lambda \vdash \mu - r} \frac{(\mu s)^{\ell(\lambda)}\,\vec{q}_{\lambda}}{|{\rm Aut}(\lambda)|} \prod_{k=1}^{\ell(\lambda)} \mathcal{S}(\mu \lambda_k s).
\]
For the proof, it is enough to note that $[s^{2a}]\, \mathcal{S}(\mu j s)^{p_j} \in \mathbb{Q}_{a}[p_j]$. So in order to compute 
\[
\sum_{\lambda \vdash \mu-r} \frac{\vec{q}_{\lambda} (\mu s)^{\ell(\lambda)}\,\mu^{c}}{|\Aut{\lambda}|}\,G_{a}(\lambda),
\] 
we have to apply a degree $a$ polynomial in the operators $\hat{\mathfrak{p}}_1,\hat{\mathfrak{p}}_2,\ldots$ to 
\[ 
\sum_{\lambda \vdash \mu-r} \frac{\vec{q}_{\lambda} (\mu s)^{\ell(\lambda)}\,\mu^{c}}{|\Aut{\lambda}|},
\]
meaning that the resulting sequence belongs to $W_{c + a}(\mu)$, due to the properties of operators $q_j\partial_{q_j}$ discussed in~\Cref{lem:dmodq}.
\end{proof}

We are now equipped to prove that $DH_{g,n}(\mu_1, \ldots,\mu_n)$ satisfies the polynomiality structure for $\mu_1$. Specifically, we prove the following theorem.
\begin{proposition} \label{thm:polyDH}
Fix $n \geq 1$ and positive integers $\mu_2,\ldots,\mu_n$. The connected double Hurwitz generating series can be written
\begin{equation}
\label{anequation} \HHall^{\circ}(\mu_1, \ldots, \mu_n) = \sum_{r=1}^{d} \left[\sum_{\lambda \vdash \mu_1 - r} \frac{\vec{q}_{\lambda} (\mu_1 s)^{\ell(\lambda)}}{|\Aut{\lambda}|} \right] \sum_{g \geq 0} P_{g}^{r;\mu_2, \ldots, \mu_n}(\mu_1),
\end{equation}
where for any $g,r$ such that $2g - 2 + n > 0$, the function $P_{g}^{r;\mu_2, \ldots, \mu_n}$ is an element of $\, \mathbb K_{3g - 3 + n}[\mu_1]$.
\end{proposition}

\begin{proof}
It is convenient for us to consider the cases $n=1$, $n=2$ and $n \geq 3$ separately. 

\emph{Case: $n = 1$.} In this case, we expect a double pole at $\mu=0$ in the genus 0 part and no other singularity. In fact, we are going to show that
\[
\HHall^{\circ}(\mu) = \sum_{r=1}^{d} \Bigg[ \sum_{\lambda \vdash \mu - r} \frac{\vec{q}_{\lambda} (\mu s)^{\ell(\lambda)}}{|\Aut{\lambda}|} \Bigg] \Bigg( \frac{r q_r}{\mu^2} + \sum_{g \geq 1} P_{g}^{r}(\mu) \Bigg),
\]
where $P_{g}^{r}$ is a polynomial in $\mathbb K[\mu]$ of degree at most $3g - 2$, whose coefficients are homogeneous polynomials in $\mathbb Q(q_1,\ldots,q_d)[s]$ of degree $2g$ for all $g \geq 1$. 

When $n=1$, the connected and disconnected double Hurwitz numbers coincide, in which case~\Cref{thm:DHinIW2} implies that
\[
\HHall^{\bullet}(\mu) = \HHall^{\circ}(\mu) = \frac{1}{\mu} \sum_{i \in \mathbb{Z}} \left[\sum_{\lambda \vdash \mu - i} \frac{\vec{q}_{\lambda} (\mu s)^{\ell(\lambda)}}{|\Aut{\lambda}|} \prod_{k=1}^{\ell(\lambda)} \mathcal{S}(\mu \lambda_k s)\right] \big\langle \mathcal{E}_{-i}(\mu s) \big\rangle^{\bullet}.
\]
Only the summand corresponding to $i=0$ contributes nontrivially, and in this case $\langle \EE_0(z) \rangle^{\bullet} = \varsigma^{-1}(z)$. We apply the peeling lemma (\Cref{lem:peel}) once with $a = 0$ to obtain
\begin{equation*}
\begin{split}
\HHall^{\circ}(\mu) &= \frac{1}{\mu} \sum_{r=1}^{d} \left[ \sum_{\lambda \vdash \mu - r} \frac{\vec{q}_{\lambda} (\mu s)^{\ell(\lambda)}}{|\Aut{\lambda}|} \prod_{k=1}^{\ell(\lambda)} \mathcal{S}(\mu \lambda_k s) \right] r q_{r}s\,\mathcal{S}( \mu r s) \frac{1}{\varsigma(\mu s)} \\
&= \sum_{r=1}^{d} \left[ \sum_{\lambda \vdash \mu - r} \frac{\vec{q}_{\lambda} (\mu s)^{\ell(\lambda)}}{|\Aut{\lambda}|} \prod_{k=1}^{\ell(\lambda)} \mathcal{S}(\mu \lambda_k s)\right] \frac{r q_{r}s\,\mathcal{S}( \mu r s)}{\mu^2 s\, \mathcal{S}(\mu s)}.
\end{split}  
\end{equation*}
Since the degree in $q$ of the term corresponding to $\lambda$ is $\ell(\lambda) + 1$, keeping only the monomials $s^{2g - 1 + (\ell(\lambda) + 1)}$ in this sum yields $\HH_{g,1}(\mu)$. This reads:
\begin{align*}
\HH_{g,1}^{\circ}(\mu) &= \sum_{r=1}^{d} s^{2g}\,\frac{rq_r}{\mu^2} \left[\sum_{\lambda \vdash \mu - r} \frac{\vec{q}_{\lambda} (\mu s)^{\ell(\lambda)}}{|\Aut{\lambda}|}\right]\,[s^{2g}] \bigg(\prod_{k=1}^{\ell(\lambda)} \mathcal{S}(\mu \lambda_k s) \cdot \frac{\mathcal{S}(\mu r s)}{\mathcal{S}(\mu s)}\bigg) \\
&= \sum_{\substack{a,b \geq 0 \\ a + b = g}} \sum_{r=1}^{d} s^{2g}\,\frac{rq_r}{\mu^2} \left[\sum_{\lambda \vdash \mu - r} \frac{\vec{q}_{\lambda} (\mu s)^{\ell(\lambda)}}{|\Aut{\lambda}|}\right]\,[s^{2a}] \bigg( \prod_{k=1}^{\ell(\lambda)} \mathcal{S}(\mu \lambda_k s)\bigg)\cdot [s^{2b}]\,\frac{\mathcal{S}( \mu r s)}{\mathcal{S}(\mu s)}.
\end{align*}
The case $g=0$ corresponds to choosing the constant from each $\mathcal{S}$ term and yields the expected expression
\[
P_{0}^{r}(\mu) = \frac{r q_r }{\mu^2}.
\]
When $g \geq 1$, we apply the first statement of~\Cref{lem:coeffofS} and find
\begin{multline*}
\sum_{\substack{a,b \geq 0 \\ a + b = g}} \sum_{r=1}^{d} s^{2g}\,\frac{rq_r}{\mu^2} \left[\sum_{\lambda \vdash \mu - r} \frac{\vec{q}_{\lambda} (\mu s)^{\ell(\lambda)}}{|\Aut{\lambda}|}\right]\,[s^{2a}] \bigg(\prod_{k=1}^{\ell(\lambda)} \mathcal{S}(\mu \lambda_k s) \bigg) \cdot [s^{2b}]\,\frac{\mathcal{S}( \mu r s)}{\mathcal{S}(\mu s)} \\
= \sum_{\substack{a,b \geq 0 \\ a + b = g }} \sum_{r=1}^{d} s^{2g}\, rq_r\,c_b(r) \left[ \sum_{\lambda \vdash \mu - r} \frac{\vec{q}_{\lambda} (\mu s)^{\ell(\lambda)}\,\mu^{2g - 2}}{|\Aut{\lambda}|} \,G_a(\lambda) \right],
\end{multline*}
where $c_b(r)$ is a degree $2b$ polynomial in the variable $r$ and $G_a \in \mathscr{R}$ comes from~\Cref{lem:coeffofS}. The application of the second statement of~\Cref{lem:coeffofS} with $c = 2g - 2$ concludes the proof of the case $n=1$. 

\emph{Case: $n=2$.} 
The expression given in~\Cref{lem:DHbehavmu1} in the case of $n=2$ is
\[
\HHall^{\bullet}(\mu_1, \mu_2) = \sum_{r=1}^{d} \left[\sum_{\lambda \vdash \mu_1 - r} \frac{\vec{q}_{\lambda} (\mu_1 s)^{\ell(\lambda)}}{|\Aut{\lambda}|} \prod_{k=1}^{\ell(\lambda)} \mathcal{S}(\mu_1 \lambda_k s) \right] \big \langle \mathcal{B}_r(\mu_1) \, \mathcal{C}(\mu_2) \big\rangle^{\bullet}.
\]
Let us analyse the possible structure of the poles. \Cref{lem:residueofBop} gives for the poles at negative integers
\begin{equation}\label{eq:resneq2}
\begin{split}
& \quad 	 \Res_{\mu_1 = -b} \dd\mu_1\,\big\langle \mathcal{B}_{r}(\mu_1)\,\mathcal{C}(\mu_2) \big\rangle^{\bullet} \\
&= M_r(b) \bigg\langle \exp\Big(\sum_{j=1}^{d} \frac{\alpha_j q_j}{j} \Big) \exp(s\mathcal{F}_2) \alpha_{b} \exp(- s\mathcal{F}_2) \exp\Big( -\sum_{j=1}^{d} \frac{\alpha_j q_j}{j} \Big) \mathcal{C}(\mu_2) \bigg\rangle^{\bullet} \\
&= M_r(b) \bigg\langle \exp\Big(\sum_{j=1}^{d} \frac{\alpha_j q_j}{j} \Big) \exp(s\mathcal{F}_2) \alpha_{b} \frac{\alpha_{-\mu_2}}{\mu_2} \bigg\rangle^{\bullet} \\
&= M_r(b) \delta_{b, \mu_2},
\end{split}
\end{equation} 
where the last line uses the commutation $[\alpha_{b},\alpha_{-\mu_2}] = \mu_2\delta_{b,\mu_2}$, the fact that $\alpha_{b}$ annihilates the vacuum since $b > 0$, and the fact that $\exp\big(\sum_{j = 1}^d \alpha_j q_j/j\big)$ and $e^{s\mathcal{F}_2}$ leave the vacuum invariant. This shows that for any $r \in \{1,\ldots,d\}$ the residue is only non-zero when $\mu_1 = -\mu_2$. Let us consider the expression for double Hurwitz numbers as given by~\Cref{thm:DHinIW2} and note that in the case $n=2$ the constraint on the energies is $i_1 = -i_2 = a$ with $ 0 \leq a \leq \mu_2$. Thus,
\begin{multline}
\HHall^{\bullet}(\mu_1,\mu_2) = \frac{1}{\mu_1\, \mu_2} \sum_{a = 0}^{\mu_2} \Bigg( \sum_{\lambda^{(1)} \vdash \mu_1 +a } \!\!\!\! \frac{\vec{q}_{\lambda^{(1)}} (\mu_1 s)^{\ell(\lambda^{(1)})}}{|\Aut{\lambda^{(1)}}|} \prod_{k=1}^{\ell(\lambda^{(1)})} \mathcal{S}(\mu_1 \lambda^{(1)}_k s) \Bigg) \\ \Bigg( \sum_{\lambda^{(2)} \vdash \mu_2 - a} \!\!\!\! \frac{\vec{q}_{\lambda^{(2)}} (\mu_2 s)^{\ell(\lambda^{(2)})}}{|\Aut{\lambda^{(2)}}|} \prod_{k=1}^{\ell(\lambda^{(2)})} \mathcal{S}(\mu_2 \lambda^{(2)}_k s) \Bigg)
\Big\langle \mathcal{E}_{a}(\mu_1 s) \mathcal{E}_{-a}(\mu_2 s) \Big\rangle^{\bullet}.
\end{multline}
The inclusion-exclusion formula in the case of $n=2$ is
\[
\langle \mathcal{E}_a(z_1) \mathcal{E}_{-a}(z_2) \rangle^{\circ} = \langle \mathcal{E}_a(z_1) \mathcal{E}_{-a}(z_2) \rangle^{\bullet} - \langle \mathcal{E}_a(z_1)\rangle^{\bullet} \, \langle \mathcal{E}_{-a}(z_1) \rangle^{\bullet}.
\]
Hence, because $\langle \mathcal{E}_a(z) \rangle^{\bullet}$ vanishes unless $a=0$, one can pass to the connected generating series by excluding the $a=0$ term from the summand. It produces the expression
\begin{equation}
\label{eq:neq2}
\begin{split}
& \quad \HHall^{\circ}(\mu_1,\mu_2) \\
& = \frac{1}{\mu_1\, \mu_2} \sum_{a = 1}^{\mu_2} \Bigg( \sum_{\lambda^{(1)} \vdash \mu_1 +a } \!\!\!\!\frac{\vec{q}_{\lambda^{(1)}} (\mu_1 s)^{\ell(\lambda^{(1)})}}{|\Aut{\lambda^{(1)}}|} \prod_{k=1}^{\ell(\lambda^{(1)})} \mathcal{S}(\mu_1 \lambda^{(1)}_k s) \Bigg) \, \Bigg( \sum_{\lambda^{(2)} \vdash \mu_2 - a} \!\!\!\!\frac{\vec{q}_{\lambda^{(2)}} (\mu_2 s)^{\ell(\lambda^{(2)})}}{|\Aut{\lambda^{(2)}}|} \prod_{k=1}^{\ell(\lambda^{(2)})} \mathcal{S}(\mu_2 \lambda^{(2)}_k s) \Bigg) \\
& \qquad \qquad \qquad \times \frac{a\mathcal S (a(\mu_1 + \mu_2)s)}{\mathcal S ((\mu_1 + \mu_2)s)}.
\end{split}
\end{equation}
We first analyse the excluded term. The inclusion-exclusion formula implies that the excluded term is a product of two copies of $n = 1$ generating series
\[ \HHall^{\circ}(\mu_1) \, \HHall^{\circ}(\mu_2). \]
We have already studied the structure of the poles of this expression while treating the $n = 1$ case. From this, we know that for any $r \in \{1,\ldots,d\}$, its only possible poles are located at $\mu_1 = 0$. But due to~\Cref{lem:poleatzero} we know, that the pole at $\mu_1 = 0$ of the connected generating series is removable for any $n \geq 2$. So, even after the removal of the disconnected part, the only possible pole of $\big \langle \mathcal{B}_r(\mu_1) \, \mathcal{C}(\mu_2) \big\rangle^{\circ}$ is located at $\mu_1 = -\mu_2$.

Recall that applying~\Cref{lem:peel} to an expression that is being summed over $\lambda \vdash \mu_1 + a$ introduces a factor of $\mu_1/(\mu_1 + a)$. We then note that a pole at $\mu_1 = -\mu_2$ can only arise when applying~\Cref{lem:peel} to an expression including a sum over $\lambda \vdash \mu_1 + \mu_2$, as this would yield a factor of $\mu_1/(\mu_1 + \mu_2)$ and no other application of the peeling lemma would (nor is that pole present without applying the peeling lemma). Observe that a sum of this form corresponds to the $a=\mu_2$ term in the sum over $a$. Hence, for the study of the pole at $\mu_1 = -\mu_2$ it suffices to consider the $a=\mu_2$ summand in the sum. Since the sum over $\lambda^{(2)} \vdash 0$ is then equal to $1$, we get
\begin{equation*}
\begin{split}
& \quad\frac 1{\mu_1\mu_2} \sum_{\lambda \vdash \mu_1 + \mu_2} \frac{\vec{q}_{\lambda} (\mu_1 s)^{\ell(\lambda)}}{|\Aut{\lambda}|} \prod_{k=1}^{\ell(\lambda)} \mathcal{S}(\mu_1 \lambda_k s) \Big\langle \EE_{\mu_2}(\mu_1 s) \EE_{-\mu_2}(\mu_2 s) \Big \rangle^{\circ} \\
& = \frac 1{\mu_1} \sum_{\lambda \vdash \mu_1 + \mu_2} \frac{\vec{q}_{\lambda} (\mu_1 s)^{\ell(\lambda)}}{|\Aut{\lambda}|} \prod_{k=1}^{\ell(\lambda)} \mathcal{S}(\mu_1 \lambda_k s) \cdot \frac{\mathcal{S}(\mu_2 (\mu_1 + \mu_2)s)}{\mathcal{S}((\mu_1 + \mu_2)s)}.
\end{split}
\end{equation*}
The expansion of the ratio of the $\mathcal{S}$-functions reads
\[
\frac{\mathcal{S}(\mu_2 (\mu_1 + \mu_2)s)}{\mathcal{S}((\mu_1 + \mu_2)s)} = 1 + \frac{(\mu_2^2-1)}{24} (\mu_1 + \mu_2)^2 s^2 + O((\mu_1 +\mu_2)^4s^4),
\]
and, extracting the coefficient of any positive power of $s$ from this term would include a factor of $(\mu_1 + \mu_2)^m$ for $m$ even and $m \geq 2$, which annihilates any simple pole at $\mu_1 = -\mu_2$. Therefore, we focus on
\[
f_{g}(\mu_1,\mu_2) := \frac 1{\mu_1} \sum_{\lambda \vdash \mu_1 + \mu_2} \frac{\vec{q}_{\lambda} (\mu_1 s)^{\ell(\lambda)}}{|\Aut{\lambda}|}\,[s^{2g}] \bigg(\prod_{k=1}^{\ell(\lambda)} \mathcal{S}(\mu_1 \lambda_k s)\bigg),
\]
for $g \geq 1$. We write
\[
\frac 1{\mu_1} \sum_{\lambda \vdash \mu_1 + \mu_2} \frac{\vec{q}_{\lambda} (\mu_1 s)^{\ell(\lambda)}}{|\Aut{\lambda}|} \prod_{k=1}^{\ell(\lambda)} \mathcal{S}(\mu_1 \lambda_k s) =\frac 1{\mu_1} \sum_{\substack{p_1,\ldots,p_d \geq 0 \\ \sum_{j = 1}^{d} jp_j = \mu_1 + \mu_2}} \prod_{j=1}^{d} (q_j \mu_1 s)^{p_j} \frac{\mathcal{S}(\mu_1 j s)^{p_j}}{p_j!},
\]
recalling that $\lambda = (1^{p_1}2^{p_2} \ldots d^{p_d})$. If we denote $\mathcal{S}(z) := \sum_{m \geq 0} S_{2m} z^{2m}$, we have
\begin{equation}
\frac{\mathcal{S}(z)^{p}}{p!} = \sum_{i \geq 0} z^{2i} \sum_{\substack{m_0,\ldots,m_{2i} \geq 0 \\ \sum_{a = 0}^{i} m_{2l} = p}} \frac{\prod_{a = 0}^{i} S_{2a}^{m_{2a}}}{m_0!m_2!\cdots m_{2i}!}
\end{equation}
for $ S_{2a}\in \mathbb Q$. Hence,
\begin{equation*}
\begin{split} 
f_{g}(\mu_1,\mu_2) = \frac 1{\mu_1}\sum_{\substack{h_1,\ldots,h_d \geq 0 \\ \sum_{j = 1}^d h_j = g}} \,\,\,\sum_{\substack{m^{(j)}\,:\,\{0,2,\ldots,2h_j\} \rightarrow \mathbb{N} \\ \sum_{a,j} m_{2a}^{(j)} = \mu_1 + \mu_2}}\,\,\, \prod_{j = 1}^{d} \left((\mu_1 j s)^{2h_j} \prod_{a = 0}^{h_j} (q_j \mu_1 s)^{m_{2a}^{(j)}}\,\frac{S_{2a}^{m_{2a}^{(j)}}}{m_{2a}^{(j)}!}\right).
\end{split}
\end{equation*}

Since $S_0 = 1$, we encode separately $m_{0}^{(j)}$ into a partition $\lambda' = (1^{m_0^{(1)}}\cdots d^{m_0^{(d)}})$, and we write $b = \mu_1 + \mu_2 - |\lambda'|$. By definition, we have $|{\rm Aut}(\lambda')| = \prod_{j = 1}^d m_0^{(j)}!$. Take into account the constraints in the sum and the fact that for $g > 0$, we have $b \in \{1,\ldots,g\}$. We can then rewrite 
\begin{equation}\label{eq:fgfinal}
\begin{split}
f_{g}(\mu_1,\mu_2) & = \frac 1{\mu_1} \sum_{b = 1}^{g} \left[\sum_{\lambda' \vdash \mu_1 + \mu_2 - b} \frac{\vec{q}_{\lambda'} (\mu_1 s)^{\ell(\lambda')}}{|{\rm Aut}(\lambda)|}\right] \\
& \qquad \times \sum_{\substack{h_1,\ldots,h_d \geq 0 \\ \sum_{j = 1}^d h_j = g}}\,\,\,\sum_{\substack{m^{(j)}\,\,:\,\,\{2,4,\ldots,2h_j\} \rightarrow \mathbb{N} \\
\nonumber
 \sum_{a,j} m_{2a}^{(j)} = b}}\,\,\,\prod_{j = 1}^{d} \left((\mu_1 j s)^{2h_j} \prod_{a = 1}^{h_j} (q_j \mu_1 s)^{m_{2a}^{(j)}}\,\frac{S_{2a}^{m_{2a}^{(j)}}}{m_{2a}^{(j)}!}\right),
\end{split}
\end{equation}
and the second line is an element of $\mathbb K_{2g + b}[\mu_1]$ 
Applying~\Cref{lem:peel} to $f_g(\mu_1,\mu_2)$ now will only introduce rational terms with denominators $\mu_1 + \mu_2 - b$ for $b > 0$; hence, this process cannot introduce a pole at $\mu_1 = -\mu_2$. 

Look closely on the expression we have for $f_g$. Its terms can be split into three groups:
\begin{itemize}
\item Terms with $\mu_2 - b \ge 0$. For these terms, we can apply~\Cref{lem:peel} in order to present them as a linear combination of elements of $W(\mu_1)$.
\item Terms with $b \in \{-d,-(d - 1),\ldots,-1\}$. These terms already yields elements of $W(\mu_1)$ and can be left intact.
\item Terms with $\mu_2 - b < -d$. Let us observe that these terms arise as expansion coefficients (in the variable $X(z)$) of the application of an element of $\mathbb K[\partial_x]$ to $z^{b-\mu_2}$. Note, however, that the degree of this differential operator to apply is $2g + b$. As $2g + b\geq 2b \geq \lceil \frac{b}{d} \rceil \geq \lceil \frac{b - \mu_2}{d} \rceil $, it can be expressed as a linear combination of elements $\hat \psi_k(z)$ as in Equation~\eqref{eq:defhatpsi}. So by Lemma~\ref{LemmaV}, the sequences of their expansion coefficients belong to $W(\mu_1)$.
\end{itemize}

From the analysis above we conclude that the terms of $f_g$, together with the remaining terms from the sum over $a$ in~\Cref{eq:neq2}, after the application of the peeling lemma,~\Cref{lem:peel}, give $DH_{g,2}$ for fixed $\mu_2$ in the form
\[
DH_{g,2}(\mu_1,\mu_2) =\sum_{r=1}^{d} \left[\sum_{\lambda \vdash \mu_1 - r} \frac{\vec{q}_{\lambda} (\mu_1 s)^{\ell(\lambda)}}{|\Aut{\lambda}|} \right] P_{g}^{r;\mu_2}(\mu_1),
\]
where a priori $P_{g}^{r,\mu_2}(\mu_1)$ are rational functions in $\mu_1$. We now are left to show that $P_{g}^{r,\mu_2}(\mu_1)$ are polynomials in $\mu_1$ for $g>0$. We know by homogeneity that $P_{g}^{r,\mu_2}$ is an eigenfunction of the operator $s\partial_s - \sum_{j = 1}^{d} q_j\partial_{q_j}$ with eigenvalue $2g$. The residues of $P_{g}^{r,\mu_2}(\mu_1)\,\dd\mu_1$ are also the eigenfunctions of the above mentioned operator with the same eigenvalue. It means that no poles outside $\mu_1 = -\mu_2$ can be introduced, because otherwise,~\Cref{eq:resneq2} would imply that these poles have to compensate each other, since any of their linear combination is an eigenvector with non-zero eigenvalue for $g > 0$. However, we have already discussed that the only possible pole of this rational function is at $\mu_1 = -\mu_2$. It means that for $g > 0$ it is in fact polynomial. Thus for any $\mu_2$ the genus $g>0$ part of the generating function is an element of $W(\mu_1)$. We can also see from~\Cref{eq:neq2}, exploiting the ideas from the proof of~\Cref{lem:coeffofS}, that for any fixed $\mu_2$ the polynomial $P^{r;\mu_2}_{g}(\mu_1)$ is of degree at most $3g-1$ in the variable $\mu_1$. This concludes the case $n=2$.

\emph{Case: $n\geq 3$.} Begin with the expression given in~\Cref{lem:DHbehavmu1},
\[
	\HHall^{\bullet}(\mu_1, \mu_2, \ldots, \mu_n) = \sum_{r=1}^{d} \left[\sum_{\lambda \vdash \mu_1 - r} \frac{\vec{q}_{\lambda} (\mu_1 s)^{\ell(\lambda)}}{|\Aut{\lambda}|} \prod_{k=1}^{\ell(\lambda)} \mathcal{S}(\mu_1 \lambda_k s) \right] \Big\langle \mathcal{B}_r(\mu_1) \prod_{j = 2}^n \mathcal{C}(\mu_j) \Big\rangle^{\bullet}.
\]
As in the $n=2$ case,~\Cref{lem:poleatzero} gives that any poles at zero are cancelled by passing to the connected generating series. We use~\Cref{lem:residueofBop} to calculate the residue at negative integers. That is,
\begin{equation*}
\begin{split}
& \quad	 \Res_{\mu_1 = -b} \dd\mu_1\,\Big\langle \mathcal{B}_{r}(\mu_1) \prod_{j=2}^{n} \mathcal{C}(\mu_j) \Big\rangle^{\bullet} \\
&= M_r(b) \Big\langle \exp\Big(\sum_{j=1}^{d} \frac{\alpha_j q_j}{j} \Big) \exp(s\mathcal{F}_2) \alpha_{b} \exp(- s\mathcal{F}_2) \exp\Big( -\sum_{j=1}^{d} \frac{\alpha_j q_j}{j} \Big) \prod_{j=2}^{n} \mathcal{C}(\mu_j) \Big\rangle^{\bullet} \\
&= M_r(b) \Big\langle \exp\Big(\sum_{j=1}^{d} \frac{\alpha_j q_j}{j} \Big) \exp(s\mathcal{F}_2) \, \alpha_{b} \frac{\alpha_{-\mu_2}}{\mu_2} \frac{\alpha_{-\mu_3}}{\mu_3} \cdots \frac{\alpha_{-\mu_n}}{\mu_n} \Big\rangle^{\bullet}.
\end{split}
\end{equation*}
Commuting $\alpha_b$ to the right, the relation $[\alpha_b, \alpha_{-\mu_k}] = b\delta_{b,\mu_k}$ implies that the residue vanishes except possibly if $b$ is equal to one of the fixed positive integers $\mu_k$ for some $k \in \{2, \ldots, n\}$. In this case, the residue reads
\[
\delta_{b,\mu_k} M_r(\mu_k) \Big\langle \exp\Big(\sum_{j=1}^{d} \frac{\alpha_j q_j}{j} \Big) \exp(s\mathcal{F}_2) \, \frac{\alpha_{-\mu_2}}{\mu_2} \cdots \widehat{\alpha}_{-\mu_k} \cdots \frac{\alpha_{-\mu_n}}{\mu_n} \Big\rangle^{\bullet} = \delta_{b,\mu_k} M_r(\mu_k) \Big\langle \prod_{\substack{j=2 \\ j \ne k}}^{n} \mathcal{C}(\mu_j) \Big\rangle^{\bullet}.
\]
By the same mechanism shown in the proof of~\Cref{lem:poleatzero}, this simple pole at $\mu_1 = -\mu_k$ will simplify in the inclusion-exclusion formula when passing to the connected generating series. Therefore, we can write 
\[
\HHall^{\circ}(\mu_1, \ldots, \mu_n) = \sum_{r=1}^{d} \left[\sum_{\lambda \vdash \mu_1 - r} \frac{\vec{q}_{\lambda} (\mu_1 s)^{\ell(\lambda)}}{|\Aut{\lambda}|}  \right] \sum_{g \geq 0} P_{g}^{r;\mu_2, \ldots, \mu_n}(\mu_1),
\] 
where $P_{g}^{r;\mu_2,\ldots,\mu_n}$ is a polynomial function of $\mu_1$, of degree at most $3g - 3 + n$. The hypothesis then follows from~\Cref{lem:coeffofS}.
\end{proof}

\subsection{End of the proof}

Keeping only the monomials $s^{2g-2+n + \deg(q)}$ in \eqref{anequation} implies that for any fixed $\mu_2,\ldots,\mu_n \in \mathbb Z_{>0}$ we have $\HH_{g,n}(\mu_1, \ldots, \mu_n) \in W_{f}(\mu_1)$ (\Cref{thm:polyDH}) for $f$ that can be chosen to depend only on $(g,n)$. It remains to invoke the symmetry of $\HH_{g,n}(\mu_1, \ldots, \mu_n)$ in $\mu_1, \ldots, \mu_n$ and apply~\Cref{lem:symmetry}: this shows that $\HH_{g,n}(\mu_1, \ldots, \mu_n) \in W(\mu_1) \otimes W(\mu_2) \otimes \cdots \otimes W(\mu_n)$.

\section{An ELSV-like formula for double Hurwitz numbers} \label{sec:ELSV}

\subsection{The ELSV formula for orbifold Hurwitz numbers}
\label{Chiodorev}

Recall that the $d$-orbifold Hurwitz numbers arise as a special case of double Hurwitz numbers via the subsitutions $q_1 = \cdots = q_{d - 1} = 0$ and $q_{d} = s = 1$. We denote $d$-orbifold Hurwitz numbers by
\[
H_{g,n}^{[d]}(\mu_1,\ldots,\mu_n) := \left. DH_{g,n}(\mu_1,\ldots,\mu_n) \right|_{q_1 = \cdots = q_{d - 1} = 0, q_{d} = s = 1}.
\]
They can be expressed via intersection theory in three different ways, as we now review.

The first formula is due to Johnson, Pandharipande and Tseng, and involves the moduli space $\overline{\mathcal M}_{g; a_1, \ldots, a_n}({\mathcal B} {\mathbb Z}_d)$ of maps from genus $g$ curves with $n$ marked points to the classifying space ${\mathcal B} {\mathbb Z}_d$, where the data $a_1, \ldots, a_n \in \mathbb{Z}_d \cong \{0, 1, 2, \ldots d-1\}$ encode the monodromy at the marked points~\cite{joh-pan-tse11}. There is a proper forgetful morphism $\epsilon: \overline{\mathcal M}_{g; a_1, \ldots, a_n}({\mathcal B} {\mathbb Z}_d) \to \overline{\mathcal M}_{g,n}$ that retains only the stabilisation of the domain curve. The Hodge bundle $\Lambda \to \overline{\mathcal M}_{g; a_1, \ldots, a_n}({\mathcal B} {\mathbb Z}_d)$ naturally carries a representation of $\mathbb{Z}_d$ and we denote by $\Lambda^U$ the subbundle corresponding to the summand on which the multiplicative unit of $\mathbb{Z}_d$ acts as multiplication by $e^{2\pi{\rm i}/d}$. For each $k \geq 0$, we let
\[
\lambda_k^U := c_k(\Lambda^{U}) \in H^{2k}\big(\overline{\mathcal M}_{g; a_1, \ldots, a_n}({\mathcal B} {\mathbb Z}_d); \mathbb{Q}\big).
\]
By specialising~\cite[Theorem 1]{joh-pan-tse11} to the case $\gamma = \emptyset$, one obtains the following ELSV formula for orbifold Hurwitz numbers, where we write $\overline{\mu}_i$ for the image of $\mu_i$ in $\mathbb{Z}_d$.
\begin{equation}
\label{orbELSV1}H_{g,n}^{[d]}(\mu_1, \ldots, \mu_n) = d^{2g-2+n+ \sum_{i = 1}^n \mu_i/d} \prod_{i=1}^n \frac{(\mu_i / d)^{\lfloor \mu_i / d \rfloor}}{\lfloor \mu_i / d \rfloor !} \int_{\overline{\mathcal M}_{g; -\overline{\mu}}({\mathcal B} {\mathbb Z}_d)} \frac{\sum_{k=0}^g (-1)^k \lambda_k^U}{\prod_{i=1}^n \big(1 - \frac{\mu_i}{d} \epsilon^* \psi_i\big)}
\end{equation}

A second formula for orbifold Hurwitz numbers involves the moduli space $\overline{\mathcal M}_{g; a_1, \ldots, a_n}^{r,s}$ of twisted $r$-spin curves. Points of this moduli space parametrise line bundles $L \to (C; p_1, \ldots, p_n)$ over genus $g$ curves with $n$ marked points together with an isomorphism $L^{\otimes r} \simeq \omega^{\otimes s}_{C} \left( \sum_{i=1}^n (s-a_i) p_i \right)$ where $\omega_{C}$ is the dualising sheaf. A comparison of degrees leads to the fact that such an isomorphism --- and hence, such a moduli space --- exists only if
\begin{equation} \label{eq:acondition}
s(2g-2+n) - \sum_{i=1}^n a_i = 0 \pmod{r}.
\end{equation}
There is a proper forgetful morphism $\tilde{\epsilon}: \overline{\mathcal M}_{g; a_1, \ldots, a_n}^{r,s} \to \overline{\mathcal M}_{g,n}$ that retains only the information of the base curve. Starting from the universal $r$th root $\pi: {\mathcal L} \to \overline{\mathcal M}_{g; a_1, \ldots, a_n}^{r,s}$, the {\em Chiodo class} is defined as the total Chern class
\[
{\rm C}_{g; a_1, \ldots, a_n}^{r,s} = c(-R^{\bullet}\pi_*\mathcal{L}) = \exp \bigg[ \sum_{m \geq 1} (-1)^m (m-1)! \, \mathrm{ch}_m(R^\bullet \pi_*{\mathcal L}) \bigg] \in H^*(\overline{\mathcal M}_{g; a_1, \ldots, a_n}^{r,s}; \mathbb{Q}),
\]
and was computed explicitly by Chiodo~\cite{chi08}. Specialising the results of~\cite{lew-pop-sha-zvo17} to the case $(r, s) = (d, d)$, we have
\begin{equation}
\label{orbELSV2}H_{g,n}^{[d]}(\mu_1, \ldots, \mu_n) = d^{2g-2+n+ \sum_{i = 1}^{n} \mu_i/d} \prod_{i=1}^n \frac{(\mu_i / d)^{\lfloor \mu_i / d \rfloor}}{\lfloor \mu_i / d \rfloor !} \int_{\overline{\mathcal M}_{g; -\overline{\mu}}^{d,d}} \frac{{\rm C}_{g;-\overline{\mu}}^{d,d}}{\prod_{i=1}^n \big(1 - \frac{\mu_i}{d} \tilde{\epsilon}^*\psi_i\big)}.
\end{equation}

A third formula can be derived from either \eqref{orbELSV1} or \eqref{orbELSV2} by pushforward to $\overline{\mathcal M}_{g,n}$, resulting in
\begin{equation} \label{eq:orbELSV}
H_{g,n}^{[d]}(\mu_1, \ldots, \mu_n) = d^{2g-2+n+ \sum_{i = 1}^n \mu_i/d} \prod_{i=1}^n \frac{(\mu_i / d)^{\lfloor \mu_i / d \rfloor}}{\lfloor \mu_i / d \rfloor !} \int_{\overline{\mathcal M}_{g,n}} \frac{\Omega_{g;-\overline{\mu}}^{[d]}}{\prod_{i=1}^n \big(1 - \frac{\mu_i}{d} \psi_i\big)},
\end{equation}
where we have defined the class
\[
\Omega_{g;a_1,\ldots,a_n}^{[d]} := \tilde{\epsilon}_* {\rm C}_{g; a_1,\ldots,a_n}^{d,d} = \epsilon_*\bigg(\sum_{k=0}^g (-1)^k \lambda_k^U\bigg) \in H^*(\overline{\mathcal M}_{g,n}; \mathbb{Q}).
\]
The second equality here is a consequence of the work of the fourth-named author, Popolitov, Shadrin and Zvonkine~\cite{lew-pop-sha-zvo17}. From the relation of~\Cref{eq:acondition}, we see that ${\rm C}_{g;a_1,\ldots,a_n}^{d,d}$ exists only if $\sum_{i = 1}^n a_i = 0 \pmod{d}$. When this condition fails to hold, we set $\Omega_{g;a_1,\ldots,a_n}^{[d]} = 0$ by convention. The class $\Omega_{g;a_1,\ldots,a_n}^{[d]}$ is tautological and was explicitly expressed in terms of $\psi$-classes and $\kappa$-classes as a sum over stable graphs by Janda, Pandharipande, Pixton and Zvonkine in their work on double ramification cycles~\cite[Corollary 4]{jan-pan-pix-zvo17}.

\subsection{Topological recursion for orbifold Hurwitz numbers}
\label{sec:orbTR}
The spectral curve $\mathscr{S}^{[d]}$ relevant for orbifold Hurwitz numbers has domain $\mathscr{C}^{[d]} = \mathbb{C}^*$ and
\begin{equation} \label{eq:orbspectralcurve}
x_{[d]}(z) = \ln z - z^d,\qquad y_{[d]}(z) = z^d,\qquad\omega_{0,2}^{[d]}(z_1, z_2) = \frac{{\mathrm d} z_1\otimes {\mathrm d}z_2}{(z_1-z_2)^2},
\end{equation}
and we denote by $\omega_{g,n}^{[d]}(z_1, \ldots, z_n)$ the correlation differentials that the topological recursion associates to it. We will also use the free energies
\begin{equation}
\label{Fgnddd} F_{g,n}^{[d]}(z_1, \ldots, z_n) = \int_0^{z_n} \cdots \int_0^{z_1} \omega_{g,n}^{[d]}.
\end{equation}
and the exponentiated variable
\[
X_{[d]}(z) := \exp\big(x_{[d]}(z)\big) = z \exp(-z^d).
\]

\begin{theorem}[\cite{do-lei-nor16, bou-her-liu-mul14}] \label{thm:orbifold}
For $2g - 2 + n > 0$, we have the all-order series expansion of the free energies when $z_i \rightarrow 0$,
\[
F_{g,n}^{[d]}(z_1, \ldots, z_n) \sim \sum_{\mu_1, \ldots, \mu_n \geq 1} H_{g,n}^{[d]}(\mu_1, \ldots, \mu_n) \prod_{i=1}^n \big(X_{[d]}(z_i)\big)^{\mu_i}.
\]
\end{theorem}

The previous theorem allows us to express $F_{g,n}^{[d]}(z_1, \ldots, z_n)$ in terms of the basis of rational functions already introduced in~\Cref{defnugfngu} that we here specialise to $\mathscr{S}^{[d]}$. Namely, for $a \in \{1,\ldots,d\}$ we set $\hat{\phi}_{-1}^{a,[d]}(z) = z^a$ and for $m \geq -1$,
\begin{equation}
\label{phidorb} \hat{\phi}_{m + 1}^{a,[d]}(z) = \partial_{x_{[d]}} \hat{\phi}_{m}^{a,[d]}(z) = X_{[d]}\partial_{X_{[d]}} \hat{\phi}^{a,[d]}_{m}(z).
\end{equation}
Further, we recall the series expansion from~\Cref{lem:expmoins1} when $z \rightarrow 0$,
\begin{equation}
\label{thedexp} a^{-1}\hat{\phi}_{m}^{a,[d]}(z) \sim \sum_{k \geq 0} \frac{(kd + a)^{k + m}}{k!} \big(X_{[d]}(z)\big)^{kd + a}.
\end{equation}

Substituting~\Cref{eq:orbELSV} for the orbifold Hurwitz numbers into~\Cref{thm:orbifold} and invoking \eqref{thedexp}, we obtain \begin{equation} \label{eq:orbifoldelsv}
F_{g,n}^{[d]}(z_1, \ldots, z_n) = d^{2g-2+n} \sum_{\substack{1 \leq a_1, \ldots, a_n \leq d \\ m_1, \ldots, m_n \geq 0}} d^{\sum_{i = 1}^n (a_i/d - m_i)} \left( \int_{\overline{\mathcal M}_{g,n}} \Omega_{g;d-a}^{[d]} \prod_{i=1}^n \psi_i^{m_i} \right) \prod_{i=1}^n a_i^{-1}\hat{\phi}_{m_i}^{a_i,[d]}(z_i),
\end{equation}
where we adopt the shorthand $d-a = (d-a_1, \ldots, d-a_n)$. Observe that in fact, the summation only ranges over $a_1, \ldots, a_n$ whose sum is a multiple of $d$ and $m_1, \ldots, m_n$ whose sum is at most $3g-3+n$. The former of these two conditions is due to~\Cref{eq:acondition}, while the latter is due to cohomological degree considerations. Equation~\eqref{eq:orbifoldelsv} will be the starting point for the proof of Theorem~\ref{thm:ELSV} in the next subsection.

\subsection{Deformation to the double Hurwitz numbers} \label{sec:flow}

We now assume $d \geq 2$. In order to treat the double Hurwitz spectral curves as deformations of the spectral curve for orbifold Hurwitz numbers, let $\hat{Q}(z) = q_1z + \cdots + q_{d-1}z^{d-1}$ and consider the 1-parameter family ${\mathscr S}^t$ of spectral curves with domain $\mathscr{C}^t = \mathbb{C}^*$ and
\[
x_t(z) = \ln z - s(q_dz^d + t\hat{Q}(z)),\qquad y_t(z) = q_dz^d + t \hat{Q}(z), \qquad \omega_{0,2}^t(z_1, z_2) = \frac{{\mathrm d} z_1\otimes {\mathrm d}z_2}{(z_1-z_2)^2}.
\]
We assume that $q_1, \ldots, q_d, s \in \mathbb{C}$, with $sq_d \neq 0$. Denote by $\omega_{g,n}^t(z_1, \ldots, z_n)$ the correlation differentials arising from the topological recursion applied to ${\mathscr S}^t$ and define the free energies as
\[
F_{g,n}^t(z_1, \ldots, z_n) = \int_0^{z_n} \cdots \int_0^{z_1} \omega_{g,n}^t.
\]
At $t=0$, setting $q_d = 1$ recovers the orbifold Hurwitz spectral curve ${\mathscr S}^{[d]}$, while at $t = 1$, we recover the double Hurwitz spectral curve $\mathscr{S}$ described in Equation~\eqref{spdouble}. Consistently with previous notations, we drop the superscript $t$ in the case $t = 1$.

\Cref{thm:TR} establishes the fact that $\omega_{g,n} = \omega_{g,n}^{t = 1}$ is a generating series for double Hurwitz numbers. Eynard provides a general formula that expresses the correlation differentials associated to a spectral curve in terms of intersection theory on the moduli space of ``coloured'' stable curves, in which components of the curve are coloured by the branch points of the spectral curve~\cite{eyn14}. Although one could in theory apply this result to the spectral curve $\mathscr{S}$ directly, the result would not be entirely satisfactory, since a direct relation to the intersection theory on $\overline{\mathcal M}_{g,n}$ is preferable. Our strategy will instead start from~\Cref{eq:orbifoldelsv} and use the deformation parameter $t$ as a way to interpolate between the case of ${\mathscr S}^{[d]}$ previously understood in~\cite{lew-pop-sha-zvo17} and~${\mathscr S}^t$.

If we use the map
\[
\begin{array}{rcl} \mathscr{C}^{t = 0} & \longrightarrow & \mathscr{C}^{[d]} \\
z & \longmapsto & \tilde{z} = (q_ds)^{1/d}z
\end{array}
\]
to identify their respective domains, the spectral curves $\mathscr{S}^{t = 0}$ and ${\mathscr S}^{[d]}$ are related by the rescalings
\[
x_{0} = x_{[d]} - \frac{\ln(q_ds)}{d}, \qquad y_{0} = \frac{y_{[d]}}{s} \qquad \omega_{0,2}^{0} = \omega_{0,2}^{[d]},
\]
while the corresponding specialisations of the rational functions of~\Cref{defnugfngu} are related by
\[
\hat{\phi}_{m}^{a,0} = (q_ds)^{-a/d}\, \hat{\phi}_{m}^{a,[d]}.
\]
The homogeneity property of topological recursion under rescaling yields the relations
\begin{equation}
\label{compar0d} \omega_{g,n}^0 = s^{2g-2+n}\,\omega_{g,n}^{[d]} \qquad \text{and} \qquad
F_{g,n}^0  = s^{2g-2+n} \, F_{g,n}^{[d]}.
\end{equation}

If $s$ is chosen small enough relative to $q_1, \ldots, q_d$, the branch points of the deformed curve remain simple, so the family ${\mathscr S}^t$ is smooth for $t$ in a neighbourhood $\mathcal{T}$ of $[0,1]$. Therefore, $F_{g,n}^t(z_1,\ldots,z_n)$ is an analytic function of $t \in {\mathcal T}$ and we can compute its value at $t = 1$ by the Taylor series
\begin{equation} \label{eq:phiseries}
F_{g,n}(z_1',\ldots,z_n') = F_{g,n}^{t = 1}(z_1', \ldots, z_n') = \sum_{\ell \geq 0} \frac{1}{\ell!} \big(\partial_{t}^{\ell} F_{g,n}^t(z_1, \ldots, z_n)\big)\big|_{t = 0},
\end{equation}
where the $t$-derivatives are computed while keeping $x_t(z_i) = x_{1}(z_i')$ fixed.

A result of Eynard and Orantin~\cite[Theorem 5.1]{eyn-ora07} allows one to compute these derivatives with respect to $t$. One should first find an analytic function $f(w)$ satisfying
\[
\big(\partial_{t} y_t(z)\big)\mathrm{d}x_t(z) -\big(\partial_{t}x_t(z)\big)\mathrm{d}y_t(z) = -\mathop{\mathrm{Res}}_{w=\infty} \omega_{0,2}^t(z, w) \, f(w),
\]
where the $t$-derivatives are here computed at fixed $z$. It is easy to see that one can take
\[
f(w) = \sum_{j=1}^{d-1} \frac{q_j}{j} w^j.
\]
The result of Eynard and Orantin then implies that\footnote{Beware that \cite{eyn-ora07} has an opposite sign convention in the residue formula defining $\omega_{g,n}$, which for us introduces the sign in \eqref{vareq}.}
\begin{equation}
\label{vareq}
\partial_{t}\omega_{g,n}^t(z_1, \ldots, z_n) = -\mathop{\mathrm{Res}}_{z = \infty} \omega_{g,n+1}^t(z_1, \ldots, z_n,z)\,f(z),
\end{equation}
where the $t$-derivative is taken for fixed $x_t(z_i)$. This feature leads to a complication when one wants to compute the higher derivatives: they will also hit the variable of $f(z)$ because we have to work at fixed $x_t(z)$.

By successive applications of the chain rule, we obtain for $\ell \geq 0$
\begin{equation}
\label{onelll}
\begin{split} \frac{\partial_{t}^{\ell}\omega_{g,n}^{t}(z_1,\ldots,z_n)}{\ell!} & = \sum_{k = 1}^{\ell}  \sum_{\substack{l_1,\ldots,l_k \geq 0 \\ k + \sum_{\kappa} l_{\kappa} = \ell}} \frac{(-1)^{k}}{(l_{k} + 1)(l_k + l_{k - 1} + 2)\cdots (l_{k} + \cdots + l_{1} + k)} \\
& \qquad \times \mathop{\mathrm{Res}}_{w_1 = \infty} \cdots \mathop{\mathrm{Res}}_{w_{k} = \infty} \bigg( \prod_{\kappa = 1}^{k} \frac{\partial_{t}^{l_{\kappa}}f(w_{\kappa})}{l_{\kappa}!}\bigg)\,\omega_{g,n + k}^{t}(z_1,\ldots,z_n,w_{1},\ldots,w_{k}),  
\end{split}
\end{equation}
where the $t$-derivatives are taken for fixed $x_t(z_i)$ and  $x_t(w_{\kappa})$. Since the second line is invariant under permutations of the $l_{\kappa}$, we can symmetrise this formula using, for $a_1,\ldots,a_{k} > 0$,
\begin{equation*}
\begin{split}
\frac{1}{k!} \sum_{\pi \in \mathfrak{S}_{k}} \frac{1}{a_{\pi(1)}(a_{\pi(1)} + a_{\pi(2)})\cdots (a_{\pi(1)} + \cdots + a_{\pi(k)})} & = \sum_{\pi \in \mathfrak{S}_{k}} \int_{\substack{0 < y_1 < \cdots < y_k \\ \sum_{\kappa} a_{\pi(\kappa)} y_{\kappa} \leq 1}} \prod_{\kappa = 1}^{k} \dd y_{\kappa} \\
& = \int_{\substack{y_{1},\ldots,y_k > 0 \\ \sum_{\kappa} a_{\kappa}y_{\kappa} \leq 1}} \prod_{\kappa = 1}^{k} \dd y_{\kappa} = \frac{1}{k!\prod_{\kappa= 1}^{k} a_{\kappa}}.
\end{split}
\end{equation*}
Using this relation with $a_\kappa = l_\kappa + 1$, we obtain
\[
\frac{\partial_{t}^{\ell}\omega_{g,n}^{t}(z_1,\ldots,z_n)}{\ell!} = \sum_{k= 1}^{\ell} \frac{(-1)^{k}}{k!} \sum_{\substack{l_1,\ldots,l_k \geq 0 \\ k + \sum_{\kappa} l_{\kappa} = \ell}} \mathop{\mathrm{Res}}_{w_1 = \infty} \cdots \mathop{{\rm Res}}_{w_k = \infty} \bigg(\prod_{\kappa = 1}^{k} \frac{\partial_{t}^{l_{\kappa}} f(w_{\kappa})}{(l_{\kappa} + 1)!}\bigg)\,\omega_{g,n + k}^{t}(z_1,\ldots,z_n,w_1,\ldots,w_k).
\]

We eventually want to set $t = 0$ so that the right-hand side involves $\omega_{g,n + k}^0$ and is related to $\omega_{g,n + k}^{[d]}$ by \eqref{compar0d}. Introducing 
\[
f^{(l)}(z) := \partial_{t}^{l}f(z)\big|_{t = 0},
\]
where the derivative is taken at $x_t(z)$ fixed before setting $t = 0$, we deduce after integration by parts that
\begin{equation}
\label{eq:derivatives} \begin{split}
& \quad \frac{1}{\ell!} \big(\partial_{t}^{\ell} F_{g,n}^t(z_1,\ldots,z_n)\big)|_{t = 0} \\
&= \sum_{k= 1}^{\ell} \sum_{\substack{l_1,\ldots,l_{k} \geq 0 \\ k + \sum_{\kappa} l_{\kappa} = \ell}} \frac{1}{k!} \mathop{\mathrm{Res}}_{w_1 = \infty} \cdots \mathop{\mathrm{Res}}_{w_k = \infty} F_{g,n+k}^0(z_1,\ldots,z_n,w_1, \ldots, w_k) \bigotimes_{\kappa=1}^k \mathrm{d}\bigg(\frac{f^{(l_{\kappa})}(w_{\kappa})}{(l_{\kappa} + 1)!}\bigg)  \\
&= \sum_{k = 1}^{\ell} \sum_{\substack{l_1,\ldots,l_{k} \geq 0 \\ k + \sum_{\kappa} l_{\kappa} = \ell}} \frac{s^{2g-2+n+k}}{k!} \mathop{\mathrm{Res}}_{w_1 = \infty} \cdots \mathop{\mathrm{Res}}_{w_k = \infty} F_{g,n+k}^{[d]}(\tilde{z}_1,\ldots,\tilde{z}_n, \tilde{w}_1, \ldots, \tilde{w}_k) \bigotimes_{\kappa=1}^k \mathrm{d}\bigg(\frac{f^{(l_{\kappa})}(w_{\kappa})}{(l_{\kappa} + 1)!}\bigg).
\end{split}
\end{equation} 
 where we recall that $\tilde{z} = (q_ds)^{1/d}z$ and we impose the same relation between $w$ and $\tilde{w}$. Since $F_{g,n+k}^{[d]}$ can be decomposed on the basis $a^{-1}\hat{\phi}_{m}^{a,[d]}(\tilde{z})$ via~\Cref{eq:orbifoldelsv}, we need to compute
\begin{equation}
\label{Tamlel} T_{a,m}^{(l)} := \mathop{\mathrm{Res}}_{z=\infty} a^{-1}\hat{\phi}_{m}^{a,[d]}(\tilde{z})\,\mathrm{d}\big(f^{(l)}(z)\big).
\end{equation}
Since $\hat{\phi}_{m}^{a,[d]}(\tilde{z}) = O(1)$ when $z \rightarrow \infty$, we only need the expansion of $f^{(l)}(z)$  up to $O(1)$ when $z \rightarrow \infty$. This is achieved by the next lemmata.

\begin{lemma}
\label{Qlemma} We have
\[
f^{(l)}(z) = \sum_{j = 1}^{d - 1} \frac{\mathcal{Q}_{j}^{(l)}}{d-j}\,z^{d - j} + O(1),\qquad z \rightarrow \infty
\]
with
\[
\frac{\mathcal{Q}_{j}^{(l)}}{(l + 1)!} = \frac{(-1)^{l}}{q_{d}^{l}}\,[1 - j/d]_{l} \sum_{\substack{\rho \in \mathscr{P}_{d - 1} \\ \ell(\rho) = l + 1 \\ |\rho| = j}} \frac{\vec{q}_{\check{\rho}}}{|{\rm Aut}(\rho)|}.
\]
Here, $\mathscr{P}_{d - 1}$ is the set of partitions whose parts are at most $d - 1$ and we adopt the notations introduced in the statement of \cref{thm:ELSV}. In particular, $\mathcal{Q}_{j}^{(l)} = 0$ for $l > j$.
\end{lemma}

\begin{proof}
It is easy to see that there exists a unique $Z_t(z) \in z + \mathbb{K}[[z^{-1}, t]]$ determined by $x_0(z) = x_t(Z_t(z))$, that is
\[
\ln Z_t(z) - s\bigg(q_dZ_t(z)^{d} + \sum_{j = 1}^{d - 1} tq_{j} Z_t(z)^{j}\bigg) = \ln z - sq_dz^{d}.
\]
We would like to compute the series
\begin{equation}
\label{thefunZ} \sum_{l \geq 0} \frac{f^{(l)}(z)\,t^{l}}{l!} = \sum_{j = 1}^{d - 1} \frac{q_j Z_t(z)^{j}}{j}
\end{equation}
only up to $O(1)$ when $z \rightarrow \infty$.  For this, it is enough to know $Z_t(z)$ up to $O(z^{1 - d})$, and since $\ln(Z_t(z)/z) = O(z^{-1})$, this expansion is in fact determined by the solution of
\begin{equation}
\label{thefunZT} Z_t(z)^{d} + \sum_{j = 1}^{d - 1} \frac{tq_j}{q_d}\,Z_t(z)^{j} = z^{d} + O(1).
\end{equation}

Let us call $\mathcal{Q}_{j}^{(l)}$ the coefficient of $\frac{z^{d - j}}{d - j}\cdot \frac{t^{l}}{l!}$ in \eqref{thefunZ}, for $j \in \{1,\ldots,d - 1\}$ and $l \geq 0$. We compute by Lagrange inversion
\begin{equation*}
\begin{split}
\sum_{l \geq 0} \frac{\mathcal{Q}_{j}^{(l)}}{l!}\,t^{l} & = -\mathop{\mathrm{Res}}_{z = \infty} \frac{(d - j)\dd z}{z^{d - j + 1}} \bigg(\sum_{a = 1}^{d - 1} \frac{q_a\,Z_t(z)^{a}}{a}\bigg) \\
& = - \sum_{a = 1}^{d - 1} q_a \mathop{\mathrm{Res}}_{z = \infty} z^{-(d - j)}Z_t(z)^{a - 1}\dd Z_t(z) \\
& = - \sum_{a = 1}^{d - 1} q_{d - a} \mathop{\mathrm{Res}}_{\xi = \infty} \xi^{-(d - j)}\bigg(1 + \sum_{b = 1}^{d - 1} \frac{tq_{d - b}}{q_d}\,\xi^{-b}\bigg)^{-(d - j)/d} \xi^{d - a - 1}\dd \xi \\
& =\sum_{a = 1}^{d - 1} q_{d - a}\,[\xi^{a - j}] \bigg(1 + \sum_{b = 1}^{d - 1} \frac{tq_{d - b}}{q_d} \xi^{-b}\bigg)^{-(d - j)/d} \\
& =  \sum_{a = 1}^{d - 1} q_{d - a} \sum_{\substack{p_1,\ldots,p_{d - 1} \geq 0 \\ \sum_{i} ip_i = j - a}} \frac{(-t/q_d)^{\sum_{i} p_i}}{p_1!\cdots p_{d - 1}!} \big[1 - j/d\big]_{\sum_{i} p_i} \prod_{i = 1}^{d - 1} q_{d - i}^{p_i}.
\end{split}
\end{equation*}

We then replace the sum over $(a,p_1,\ldots,p_{d - 1})$ by a sum over the partition $\rho = (1^{\overline{p}_1} \cdots (d - 1)^{\overline{p}_{d - 1}}) \in \mathscr{P}_{d - 1}$ with $\overline{p}_{i} = p_i + \delta_{i,a}$. Note that $d - \rho_i$ are the parts of the complement partition $\check{\rho}$. 
\begin{equation*}
\begin{split}
\sum_{l \geq 0} \frac{\mathcal{Q}_{j}^{(l)}}{l!}\,t^{l} & = \sum_{\substack{\rho \in \mathscr{P}_{d - 1} \\ |\rho| = j}} (-t/q_d)^{\ell(\rho) - 1}\,\cdot \big[1 - |\rho|/d\big]_{\ell(\rho) - 1}\,\frac{\vec{q}_{\check{\rho}}}{|{\rm Aut}(\rho)|} \bigg(\sum_{i = 1}^{d - 1} \overline{p}_i\bigg) \\
& = \sum_{\substack{\rho \in \mathscr{P}_{d - 1} \\ |\rho| = j}} (-t/q_d)^{\ell(\rho) - 1}\,\big[1 - |\rho|/d\big]_{\ell(\rho) - 1}\frac{\ell(\rho)\,\vec{q}_{\check{\rho}}}{|{\rm Aut}(\rho)|} .
\end{split} 
\end{equation*}
Extracting the coefficient of $t^{l}/l!$ restricts the sum to $\ell(\rho) = l + 1$ and yields the announced result. 
\end{proof}

\begin{lemma}
\label{Tlemma} For $a \in \{1,\ldots,d\}$ and $m,l \geq 0$, we have
\[
T_{a,m}^{(l)} = \delta_{m,0} \cdot \begin{cases}
\dfrac{\mathcal{Q}^{(l)}_{a}}{d (q_ds)^{(d-a)/d}}, & \text{if } a \in \{1,\ldots,d - 1\}, \\
\,\,\,0, & \text{if } a = d.
\end{cases}
\]
\end{lemma}

\begin{proof}
To compute \eqref{Tamlel}, one first checks by induction that
\[
\hat{\phi}_m^{a,[d]}(\tilde{z}) = \frac{\tilde{z}^ap_{a,m}(\tilde{z}^{d})}{(1 - d\tilde{z}^{d})^{2m + 1}}
\]
for some polynomial $p_{a,m}$ that has degree $m$ if $a \in \{1,\ldots,d - 1\}$ and degree $m - 1$ if $a = d$. It follows that
\begin{equation}
\label{eq:Rbk}
T_{a,m}^{(l)} = 0\qquad {\rm for}\,\,m \geq 1.
\end{equation}
Equation~\eqref{phidorb} allows us to directly calculate $T_{a,0}^{(l)}$ and yields
\begin{equation}
\label{eq:Rb0}
T_{a,0}^{(l)} = - \sum_{j = 1}^{d - 1} \mathcal{Q}_{j}^{(l)} \,[z^{-(d - j)}]\,\frac{(q_d s)^{a/d} z^{a}}{1 - dq_ds z^{d}},
\end{equation}
Evaluating this gives the claim.
\end{proof}

We now substitute~\Cref{eq:derivatives} into the Taylor series of~\Cref{eq:phiseries} and use the expression for $F_{g,n+k}^{[d]}$ of ~\Cref{eq:orbifoldelsv}. The result is a finite sum as $\mathcal{Q}_{j}^{(l)} = 0$ for $l > j$:
\begin{equation*}
\begin{split}
F_{g,n}(z_1, \ldots, z_n) & = (ds)^{2g-2+n} \sum_{\substack{1 \leq a_1, \ldots, a_n \leq d \\ m_1, \ldots, m_n \geq 0}} \sum_{k \geq 0} \frac{1}{k!} \sum_{\substack{1 \leq b_1, \ldots, b_k \leq d-1 \\ l_1,\ldots,l_k \geq 0}}   d^{\sum_{i = 1}^n (a_i/d - m_i)} (ds)^{k} \\
& \quad \times \left( \int_{\overline{\mathcal M}_{g,n+k}} \Omega_{g;d-a,d-b}^{[d]} \prod_{i=1}^n \psi_i^{m_i} \right) \prod_{\kappa=1}^k \frac{T_{b_{\kappa},0}^{(l_{\kappa})}}{(l_{\kappa} + 1)!}  \prod_{i=1}^n  a_i^{-1}\hat{\phi}_{m_i}^{a_i,[d]}(\tilde{z}_i),
\end{split}
\end{equation*}
where we recall the identification
\begin{equation} \label{idprimed}
X(z_i) = (q_ds)^{-1/d}X_{[d]}(\tilde{z}_i).
\end{equation}
Inserting Lemma~\ref{Tlemma}, performing the substitution $b_i \mapsto d - b_i$ and rearranging yields the following equation.
\begin{equation} \label{eq:phiMgn}
\begin{split}
F_{g,n}(z_1, \ldots, z_n) & = (ds)^{2g-2+n} \sum_{\substack{1 \leq a_1, \ldots, a_n \leq d \\ m_1, \ldots, m_n \geq 0}} d^{\sum_{i = 1}^n (a_i/d -m_i)} \prod_{i=1}^n a_i^{-1}\hat{\phi}_{m_i}^{a_i,[d]}(\tilde{z}_i) \\
& \times \sum_{k \geq 0} \frac{1}{k!} \sum_{\substack{1 \leq b_1,\ldots,b_k \leq d - 1 \\ l_1,\ldots,l_k \geq 0}} \frac{(ds)^{\sum_{\kappa = 1}^{k} b_{\kappa}/d}}{q_{d}^{\sum_{\kappa = 1}^{k} (d - b_{\kappa})/d}}  \left( \int_{\overline{\mathcal M}_{g,n+k}} \Omega_{g;d-a,d - b}^{[d]} \prod_{i=1}^n \psi_i^{m_i} \right)\prod_{\kappa=1}^k \frac{\mathcal{Q}_{b_{\kappa}}^{(l_{\kappa})}}{(l_{\kappa} + 1)!}.
\end{split} 
\end{equation}

The double Hurwitz numbers are encoded in the series expansion of \eqref{eq:phiMgn} in the variable $X(z_i)$ at $z_i \rightarrow 0$, see~\Cref{thm:TR}. Taking into account the series expansion \eqref{thedexp} for $\hat{\phi}^{a,[d]}_m(\tilde{z}_i)$ in the variable $X_{[d]}(\tilde{z}_i)$ and the identification \eqref{idprimed}, we find 
\begin{equation}
\label{thefinalf}\begin{split}
&DH_{g,n}(\mu_1,\ldots,\mu_n) = (ds)^{2g - 2 + n} \prod_{i = 1}^n \frac{(\mu_i/d)^{\lfloor \mu_i /d \rfloor}}{\lfloor \mu_i /d \rfloor !} \\
& \times \!\!\!\! \sum_{\substack{k \geq 0 \\ l_1,\ldots,l_{k} \geq 0 \\ 1 \leq b_1,\ldots,b_{k} \leq d - 1}} \!\!\!\!\! \frac{1}{k!} (ds)^{\sum \mu_i/d + \sum b_{\kappa}/d} q_d^{\sum \mu_i/d - \sum (d - b_{\kappa})/d}
 \prod_{\kappa = 1}^{k} \frac{\mathcal{Q}_{b_{
\kappa}}^{(l_{\kappa})}}{(l_{\kappa} + 1)!} \bigg(\int_{\overline{\mathcal{M}}_{g,n+k}} \frac{\Omega^{[d]}_{g;-\overline{\mu},d - b}}{\prod_{i = 1}^n \big(1 - \frac{\mu_i}{d}\psi_i\big)} \bigg). 
\end{split}  
\end{equation}
Next, we substitute the decomposition from Lemma~\ref{Qlemma}, using the involution $\rho \mapsto \check{\rho}$ in the summation
\begin{equation}
\frac{\mathcal{Q}_{b}^{(l)}}{(l + 1)!} = \sum_{\substack{\rho \in \mathscr{P}_{d - 1} \\ |\check{\rho}| = b \\ \ell(\rho) = l + 1}} \frac{(-1)^{\ell(\rho) - 1}}{q_{d}^{-(\ell(\rho) - 1)}}\,\big[1 - |\check{\rho}|/d]_{\ell(\rho) - 1}\,\frac{\vec{q}_{\rho}}{|{\rm Aut}(\rho)|},
\end{equation}  
and collect the coefficient of a given monomial $q_1^{p_1}\cdots q_d^{p_d} = \vec{q}_{\lambda'} q_{d}^{p_d}$ that can occur for $p_1,\ldots,p_{d - 1} \geq 0$ and $p_{d} \in \mathbb{Z}$. The (possibly empty) partition $\lambda' = (1^{p_1}\cdots (d - 1)^{p_{d - 1}}) \in \mathscr{P}_{d - 1}$ is built from the concatenation of partitions $\rho^{(\kappa)}$ coming from the factors labelled by $\kappa \in \{1,\ldots,k\}$. We say that $\lambda = (1^{p_1}\cdots d^{p_d})$ where $p_d$ is allowed to be negative is an ``extended partition'' and we denote the set of such extended partitions by $\overline{\mathscr{P}}_d$. We also extend the usual notions in the obvious way
\[
\ell(\lambda) = \sum_{a = 1}^{d} p_a,\qquad |\lambda| = \sum_{a = 1}^{d} ap_a,\qquad \vec{q}_{\lambda} = \prod_{a = 1}^d q_a^{p_a}.
\]
In terms of the indices appearing in \eqref{thefinalf}, the extended partition $\lambda$ has the following characteristics
\begin{equation}
\label{charho}\begin{split}
p_{d} & = \sum_{i = 1}^n \frac{\mu_i}{d} - \sum_{\kappa = 1}^{k} \frac{d(l_{\kappa} + 1) - b_{\kappa}}{d}  = \frac{|\mu| - |\lambda'|}{d}, \\
\ell(\lambda) & = \sum_{\kappa = 1}^k (l_{\kappa} + 1) + p_{d} =  \sum_{i = 1}^n \frac{\mu_i}{d} + \sum_{\kappa = 1}^{k} \frac{b_{\kappa}}{d}, \\
|\lambda| & = |\mu|.
\end{split}
\end{equation}
The fact that $p_d$ is an integer comes from the vanishing of the Chiodo class in \eqref{thefinalf} when $\sum_{i = 1}^n \mu_i + \sum_{\kappa = 1}^{k} b_{\kappa}$ is not divisible by $d$. We deduce from \eqref{charho} the identification of the factor $(ds)^{2g - 2 + n + \ell(\lambda)}$, and we can write the result in a form closer to Theorem~\ref{thm:ELSV}, that is
\begin{equation*}
\label{thefinalfbis} 
\begin{split} 
DH_{g,n}(\mu_1,\ldots,\mu_n) & = \sum_{\substack{\lambda \in \overline{\mathscr{P}}_{d} \\ |\lambda| = |\mu|}} (ds)^{2g - 2 + n + \ell(\lambda)}\,\vec{q}_{\lambda}  \prod_{i = 1}^n \frac{(\mu_i/d)^{\lfloor \mu_i/d \rfloor}}{\lfloor \mu_i/d \rfloor !} \\ 
& \quad  \times \Bigg( \sum_{k = 0}^{\ell(\lambda')} \frac{(-1)^{\ell(\lambda') - k}}{k!} \sum_{\substack{\boldsymbol{\rho} \in (\mathscr{P}_{d - 1})^{k} \\ \sqcup_{\kappa} \rho^{(\kappa)} = \lambda' \\ |\check{\rho}^{(\kappa)}| \leq d - 1}}  \prod_{\kappa = 1}^{k} \frac{\big[\frac{d - |\check{\rho}^{(\kappa)}|}{d}\big]_{\ell(\rho^{(\kappa)}) - 1}}{|{\rm Aut}(\rho^{(\kappa)})|}  \int_{\overline{\mathcal{M}}_{g,n + k}} \frac{\Omega^{[d]}_{g;-\overline{\mu},d - |\check{\boldsymbol{\rho}}|}}{\prod_{i = 1}^n \big(1 - \frac{\mu_i}{d}\psi_i\big)} \Bigg).
\end{split}
\end{equation*} 
Here, we sum over tuples of partitions $\boldsymbol{\rho} = \big(\rho^{(1)},\ldots,\rho^{(k)}\big)$ whose concatenation $\sqcup_{\kappa} \rho^{(\kappa)}$ is $\lambda'$, and whose respective complement sizes $|\check{\rho}^{(\kappa)}|$ belong to $\{1,\ldots,d - 1\}$. We have denoted $d - |\check{\boldsymbol{\rho}}| = \big(d - |\check{\rho}^{(1)}|,\ldots,d - |\check{\rho}^{(k)}|\big)$. If $\lambda' = \emptyset$, there remains only the term $k = 0$ in the parentheses and it is equal to $1$. When $\lambda' \neq \emptyset$, the sum rather starts at $k \geq 1$. As the ``complement'' operation is compatible with the concatenation, we can equally write
\begin{equation*}
\begin{split} 
DH_{g,n}(\mu_1,\ldots,\mu_n) & = \sum_{\substack{\lambda \in \overline{\mathscr{P}}_{d} \\ |\lambda| = |\mu|}} (ds)^{2g - 2 + n + \ell(\lambda)}\,\vec{q}_{\lambda} \prod_{i = 1}^n \frac{(\mu_i/d)^{\lfloor \mu_i/d \rfloor}}{\lfloor \mu_i/d \rfloor !} \\ 
& \quad \times \Bigg(\sum_{k= 0}^{\ell(\lambda')} \frac{(-1)^{\ell(\lambda') - k}}{k!} \sum_{\substack{\boldsymbol{\rho} \in (\tilde{\mathscr{P}}_{d - 1})^{k} \\ \sqcup_{\kappa} \rho^{(\kappa)} = \check{\lambda'}}} \prod_{\kappa = 1}^{k} \frac{\big[\frac{d - |\rho^{(\kappa)}|}{d}\big]_{\ell(\rho^{(\kappa)}) - 1}}{|{\rm Aut}(\rho^{(\kappa)})|} \int_{\overline{\mathcal{M}}_{g,n + k}} \frac{\Omega^{[d]}_{g;-\overline{\mu},d - |\boldsymbol{\rho}|}}{\prod_{i = 1}^n \big(1 - \frac{\mu_i}{d}\psi_i\big)} \Bigg),
\end{split}
\end{equation*} 
where $\tilde{\mathscr{P}}_{d - 1}$ is the set of partitions of total size $\leq d - 1$.

One should think of the first term of the sum as $k = \ell(\lambda')$. Up to ordering which is killed by $|{\rm Aut}(\lambda')|/k!$ we must have $\rho^{(\kappa)} = \check{\lambda'}_{\kappa}$ and thus it is equal to
\[
\frac{1}{|{\rm Aut}(\lambda')|} \int_{\overline{\mathcal{M}}_{g,n + \ell(\lambda')}} \frac{\Omega^{[d]}_{g;-\overline{\mu},\lambda'}}{\prod_{i = 1}^n \big(1 - \frac{\mu_i}{d}\psi_i\big)}.
\]
The following lemma says that the ``boundedness assumption'' from \cite{joh-pan-tse11} is precisely the condition under which this is the only surviving term.

\begin{lemma}
If $\lambda' \in \mathscr{P}_{d - 1}$ is such that $\max_{i \neq j} (\lambda_i + \lambda_j) \leq d$, all terms in the parentheses are zero except for the $k = \ell(\lambda')$ term.
\end{lemma}

\begin{proof}
Any summand with $k \leq \ell(\lambda') - 1$ has at least one $\kappa \in \{1,\ldots,k\}$ such that $\ell(\rho^{(\kappa)}) \geq 2$. The condition is equivalent to $\min_{i \neq j} (\check{\lambda'_i} + \check{\lambda'_j}) \geq d$. Therefore $|\rho^{(\kappa)}| \geq d$ but these are not allowed in the sum.
\end{proof}

From their definition, the double Hurwitz numbers for fixed $\mu_1,\ldots,\mu_n$ are polynomials in $q_1,\ldots,q_d$ and the power of $s$ associated with a monomial of degree $\ell$ in the $q$-variables must be $2g - 2 + n + \ell$. In \eqref{thefinalf}, we see that the power of $s$ is as desired but there could be terms with negative powers of $q_d$, so the corresponding coefficient must vanish. This gives us the following linear relation between Chiodo integrals. For any $n,\ell \geq 1$ and any partitions $\mu$ of length $n$ and $\lambda'$ of length $\ell$, we have
\[
|\lambda'| \geq d + \sum_{i = 1}^n \mu_i \quad \Longrightarrow \quad \sum_{k = 1}^{\ell(\lambda')} \frac{(-1)^{\ell(\lambda') - k}}{k!} \sum_{\substack{\boldsymbol{\rho} \in (\tilde{\mathscr{P}}_{d - 1})^{k} \\ \sqcup_{\kappa} \rho^{(\kappa)} = \check{\lambda'}}} \prod_{\kappa = 1}^{k} \frac{\big[\frac{d - |\rho^{(\kappa)}|}{d}\big]_{\ell(\rho^{(\kappa)}) - 1}}{|{\rm Aut}(\rho^{(\kappa)})|} \int_{\overline{\mathcal{M}}_{g,n + k}} \frac{\Omega^{[d]}_{g;-\overline{\mu},d - |\boldsymbol{\rho}|}}{\prod_{i = 1}^{n} \big(1 - \frac{\mu_i}{d}\psi_i\big)} = 0.
\]
Setting $\eta = \check{\lambda'}$ gives the vanishing result of Theorem~\ref{thm:vanishC}. This also completes the proof of the ELSV-like formula of Theorem~\ref{thm:ELSV}.

\appendix

\section{Generating series for \texorpdfstring{$DH_{0,2}$}{DH02}} \label{app02}

The following result is proved in~\cite{do-kar18} from the cut-and-join equation. We propose an alternative proof solely based on the evaluation of the vacuum expectation in the semi-infinite wedge, appearing in~\Cref{eq:neq2}.

\begin{theorem}
We have the all-order series expansion near $z_i \rightarrow 0$,
\[
\sum_{\mu_1,\mu_2 \geq 1} \HH_{0,2}(\mu_1,\mu_2) X_1^{\mu_1} X_2^{\mu_2} \sim \ln \bigg(\frac{z_1 - z_2}{z_1z_2}\bigg) - \ln \bigg(\frac{X_1 - X_2}{X_1X_2}\bigg),
\]
where $X_i = X(z_i) = e^{x(z_i)} = z_ie^{-sQ(z_i)}$.
\end{theorem}

\begin{proof}
We start with~\Cref{eq:neq2}, which implies for genus 0:
\[
DH_{0,2}(\mu_1,\mu_2) = \frac{1}{\mu_1\, \mu_2} \sum_{a = 1}^{\mu_2} a\Bigg( \sum_{\lambda^{(1)} \vdash \mu_1 +a } \frac{ (\mu_1 s)^{\ell(\lambda^{(1)})}\vec{q}_{\lambda^{(1)}}}{|\Aut{\lambda^{(1)}}|}\Bigg) \, \Bigg( \sum_{\lambda^{(2)} \vdash \mu_2 - a} \frac{(\mu_2 s)^{\ell(\lambda^{(2)})}\vec{q}_{\lambda^{(2)}}}{|\Aut{\lambda^{(2)}}|} \Bigg).
\]
Invoking the expressions for the expansion coefficients found in the proof of~\Cref{lem:expmoins1}, we obtain the alternative formula
\[
DH_{0,2}(\mu_1,\mu_2) = - [X_1^{\mu_1}X_2^{\mu_2}]\left( \sum_{a \geq 1} \frac{z_1^{-a}z_2^a}{a} \right).
\]
Note that, for any fixed $\mu_2$, only finitely many terms have to be considered. For $|z_2| < |z_1|$ and $|X_2| < |X_1|$, we have the expansion
\begin{equation}\label{eq:02exp}
\ln {\left( \frac {z_1 - z_2}{z_1z_2} \right)} - \ln {\left( \frac {X_1 - X_2}{X_1X_2} \right)} \sim \\
-\ln{z_2} - \sum_{a \geq 1} \frac{z_1^{-a}z_2^a}{a} + \ln{X_2} + \sum_{a \geq 1} \frac {X_1^{-a}X_2^a}{a}.
\end{equation}
Recalling that $z^{-a} = X^{-a} + O(X^{-a+1})$ when $z \rightarrow 0$, it is immediate to see that for $\mu_1,\mu_2 > 0$
\[
[X_1^{\mu_1} X_2^{\mu_2}]\left( \ln {\left( \frac {z_1 - z_2}{z_1z_2} \right)} - \ln {\left( \frac {X_1 - X_2}{X_1X_2} \right)} \right) = -[X_1^{\mu_1} X_2^{\mu_2}]\left(\sum_{a \geq 1} \frac{z_1^{-a}z_2^a}{a}\right).
\]
We now need to check that if either of the integers $\mu_i$ is non-positive, we have
\[
[X_1^{\mu_1} X_2^{\mu_2}]\left\{ \ln {\left( \frac {z_1 - z_2}{z_1z_2} \right)} - \ln {\left( \frac {X_1 - X_2}{X_1X_2} \right)} \right\} = 0.
\]
As the expression is symmetric in the $\mu_i$, it is enough to show it for non-positive $\mu_2$ and any $\mu_1 \in \mathbb{Z}$. If $\mu_2 < 0$, for any $\mu_1 \in \mathbb{Z}$, we indeed have that
\[
[X_1^{\mu_1} X_2^{\mu_2}]\left(
-\ln{z_2} - \sum_{a \geq 1} \frac{z_1^{-a}z_2^a}{a} + \ln{X_2} + \sum_{a \geq 1} \frac {X_1^{-a}X_2^a}{a}\right) = 0.
\]
The same can be checked for $\mu_2 = 0$ and any $\mu_1 \neq 0$, while if both $\mu_1 = \mu_2 = 0$, we have
\begin{multline*}
[X_1^{0} X_2^{0}]\left(-\ln{z_2} - \sum_{a \geq 1} \frac{z_1^{-a}z_2^a}{a} + \ln{X_2} + \sum_{a \geq 1} \frac {X_1^{-a}X_2^a}{a} \right) \\ = [X_1^{0} X_2^{0}]\big( -\ln{z_2} + \ln{X_2}\big) = - [X_1^{0} X_2^{0}] \ln\big(1 + O(X_2)\big) = 0. \tag*{\qedhere}
\end{multline*}
\end{proof}

\section{Numerics for \texorpdfstring{$DH_{1,1}$}{DH11}}

Our main results can be checked by computation and we provide some of the evidence here. In particular, we independently computed double Hurwitz numbers of the form $DH_{1,1}(\mu)$ for $d = 3$ in four different ways:
\begin{itemize}
\item from the semi-infinite wedge formalism (\cref{thm:DHinIW1});
\item from the topological recursion (\cref{thm:TR}), leading to
\begin{equation} \label{1n1n1}
\begin{split}
\omega_{1,1}(z) & = \frac{z\,\dd z}{24(1 - q_1z - 2q_2z^2 - 3q_3z^3)^4}\Big( -81q_3^3z^7 - 216q_2q_3^2z^6 -3q_3(81-q_1q_3+32q_2^2)z^5 \\ 
& \quad + (-144q_1q_2q_3-16q_2^3+540q_3^2)z^4 + (-3q_1^2q_3-32q_1q_2^2+360q_2q_3)z^3 + \\
& \quad 8q_2(-q_1^2 + 10q_2)z^2 + (-q_1^3+24q_1q_2+72q_3)z + 4(q_1^2+3q_2) \Big) \\ 
& = -\frac{q_{1}}{24}\,\dd\hat{\phi}_0^1(z) - \frac{q_{2}}{24}\,\dd \hat{\phi}_{0}^2(z) - \frac{q_{3}}{24}\,\dd\hat{\phi}_{0}^3 + \frac{q_{1}}{24}\,\dd\hat{\phi}_{1}^{1}(z) + \frac{q_{2}}{12}\,\dd\hat{\phi}_{1}^{2}(z) + \frac{q_{3}}{8}\,\dd\hat{\phi}_{1}^3(z);
\end{split} 
\end{equation}
\item from the cut-and-join equation, which was solved for $DH_{1,1}$ for any $d$ in \cite[Formula (12)]{do-kar18}; and
\item from intersection theory of Chiodo classes (\cref{thm:ELSV}) using the SageMath package \textsc{admcycles}~\cite{admcycles}.\footnote{Alessandro Giacchetto and the third-named author have also conducted extensive numerical checks on the reliability of the SageMath package \textsc{admcycles} for integrals of Chiodo classes. This involved the computation of $r$-spin $q$-orbifold Hurwitz numbers independently via topological recursion and the package itself, by exploiting the $qr$-ELSV formula \cite{kra-lew-pop-sha19} now proved in \cite{bor-kra-lew-pop-sha20,dun-kra-pop-sha19-2}. The \textsc{admcycles} package was adapted by Johannes Schmitt to calculate intersection numbers involving Chiodo classes.}
\end{itemize}

The table below shows the coefficients of $\vec{q}_\lambda$ appearing in $DH_{1,1}(\mu)$ for $d = 3$ and $\mu \leq 7$, along with the corresponding intersection-theoretic expressions produced by \cref{thm:ELSV}. The cases in which a linear combination of more than one integral of Chiodo classes appears are exactly the ones that are covered by our Theorem~\ref{thm:ELSV}, but not by the formula of Johnson, Pandharipande and Tseng~\cite{joh-pan-tse11}. The three entries in the table marked by ($\ast$) could not be computed using \textsc{admcycles} in a reasonable amount of time, but are covered by the aforementioned formula of Johnson, Pandharipande and Tseng, thus providing an independent verification.

\begin{center}
\begin{tabular}{ccccc} \toprule
$DH_{1,1}(2)$ & \multirow{2}{*}{} $\![q_1^2]$ {\rule{0pt}{4ex}}{\rule[-1.8ex]{0pt}{0pt}}& $\frac{1}{12}$ & $\frac{27}{2} \int_{\overline{\mathcal{M}}_{1,3}} \frac{\Omega_{1;1,1,1}^{[3]}}{1 - 2 \psi_1/3}$ \\
& $[q_2]$ {\rule{0pt}{4ex}}{\rule[-1.8ex]{0pt}{0pt}}& $\frac{1}{4}$ & $9 \int_{\overline{\mathcal{M}}_{1,2}} \frac{\Omega^{[3]}_{1;1,2}}{1 - 2\psi_1/3}$ \\ \midrule
$DH_{1,1}(3)$ & \multirow{3}{*}{} $\![q_1^3]$ {\rule{0pt}{4ex}}{\rule[-1.8ex]{0pt}{0pt}}& $\frac{3}{8}$ & $\frac{27}{2} \int_{\overline{\mathcal{M}}_{1,4}} \frac{\Omega^{[3]}_{1;0,1,1,1}}{1 - \psi_1}$ \\
& $[q_1q_2]$ {\rule{0pt}{4ex}}{\rule[-1.8ex]{0pt}{0pt}}& $\frac{3}{2}$ & $27 \int_{\overline{\mathcal{M}}_{1,3}} \frac{\Omega_{1;0,1,2}^{[3]}}{1 - \psi_1}$ \\
& $[q_3]$ {\rule{0pt}{4ex}}{\rule[-1.8ex]{0pt}{0pt}}& $1$ & $3 \int_{\overline{\mathcal{M}}_{1,1}} \frac{\Omega_{1;0}^{[3]}}{1 - \psi_1}$ \\ \midrule
$DH_{1,1}(4)$ & \multirow{4}{*}{} $\![q_1^4]$ {\rule{0pt}{4ex}}{\rule[-1.8ex]{0pt}{0pt}}& $\frac{4}{3}$ & $\frac{27}{2} \int_{\overline{\mathcal{M}}_{1,5}} \frac{\Omega^{[3]}_{1;2,1,1,1,1}}{1 - 4\psi_1/3}$ \\
& $[q_1^2q_2]$ {\rule{0pt}{4ex}}{\rule[-1.8ex]{0pt}{0pt}}& $\frac{20}{3}$ & $54 \int_{\overline{\mathcal{M}}_{1,4}} \frac{\Omega^{[3]}_{1;2,1,1,2}}{1 - 4\psi_1/3}$ \\
& $[q_2^2]$ {\rule{0pt}{4ex}}{\rule[-1.8ex]{0pt}{0pt}}& $\frac{7}{3}$ & $18 \int_{\overline{\mathcal{M}}_{1,3}} \frac{\Omega^{[3]}_{1;2,2,2}}{1 - 4\psi_1/3} - 6 \int_{\overline{\mathcal{M}}_{1,2}} \frac{\Omega^{[3]}_{1;2,1}}{1 - 4\psi_1/3}$ \\
& $[q_1q_3]$  {\rule{0pt}{4ex}}{\rule[-1.8ex]{0pt}{0pt}}& $6$ & $36 \int_{\overline{\mathcal{M}}_{1,2}} \frac{\Omega^{[3]}_{1;2,1}}{1 - 4\psi_1/3}$ \\ \midrule
$DH_{1,1}(5)$ & \multirow{5}{*}{} $\! [q_1^5]$ {\rule{0pt}{4ex}}{\rule[-1.8ex]{0pt}{0pt}}& $\frac{625}{144}$ & $\frac{81}{8} \int_{\overline{\mathcal{M}}_{1,6}} \frac{\Omega^{[3]}_{1;1,1,1,1,1,1}}{1 - 5\psi_1/3}$ \\
& $[q_1^3q_2]$ & $\frac{625}{24}$ & $\frac{135}{2} \int_{\overline{\mathcal{M}}_{1,5}} \frac{\Omega^{[3]}_{1;1,1,1,1,2}}{1 - 5\psi_1/3}$ \\
& $[q_1q_2^2]$ {\rule{0pt}{4ex}}{\rule[-1.8ex]{0pt}{0pt}}& $\frac{125}{6}$ & $\frac{135}{2} \int_{\overline{\mathcal{M}}_{1,4}} \frac{\Omega^{[3]}_{1;1,1,2,2}}{1 - 5\psi_1/3} - \frac{45}{2} \int_{\overline{\mathcal{M}}_{1,3}} \frac{\Omega^{[3]}_{1;1,1,1}}{1 - 5\psi_1/3}$ \\
& $[q_1^2q_3]$ {\rule{0pt}{4ex}}{\rule[-1.8ex]{0pt}{0pt}}& $\frac{625}{24}$ & $\frac{135}{2} \int_{\overline{\mathcal{M}}_{1,3}} \frac{\Omega_{1;1,1,1}^{[3]}}{1 - 5\psi_1/3}$ \\
& $[q_2q_3]$ {\rule{0pt}{4ex}}{\rule[-1.8ex]{0pt}{0pt}}& $\frac{25}{2}$ & $45 \int_{\overline{\mathcal{M}}_{1,2}} \frac{\Omega^{[3]}_{1;1,2}}{1 - 5\psi_1/3}$ \\ \bottomrule
\end{tabular}
\end{center}

\begin{center}
\begin{tabular}{ccccc} \toprule
$DH_{1,1}(6)$ & \multirow{7}{*}{} $\!\!\! [q_1^6]$ {\rule{0pt}{4ex}}{\rule[-1.8ex]{0pt}{0pt}}& $\frac{27}{2}$ & $\frac{243}{40} \int_{\overline{\mathcal{M}}_{1,7}} \frac{\Omega^{[3]}_{1;0,1,1,1,1,1,1}}{1 - 2\psi_1}$ ~ ($\ast$) \\
& $[q_1^4q_2]$ {\rule{0pt}{4ex}}{\rule[-1.8ex]{0pt}{0pt}}& $\frac{189}{2}$ & $\frac{243}{4} \int_{\overline{\mathcal{M}}_{1,6}} \frac{\Omega^{[3]}_{1;0,1,1,1,1,2}}{1 - 2\psi_1}$ \\
& $[q_1^2q_2^2]$ {\rule{0pt}{4ex}}{\rule[-1.8ex]{0pt}{0pt}}& $\frac{243}{2}$ & $\frac{243}{2} \int_{\overline{\mathcal{M}}_{1,5}} \frac{\Omega^{[3]}_{1;0,1,1,2,2}}{1 - 2\psi_1} - \frac{81}{2} \int_{\overline{\mathcal{M}}_{1,4}} \frac{\Omega^{[3]}_{1;0,1,1,1}}{1 - 2\psi_1}$ \\
& $[q_2^3]$ {\rule{0pt}{4ex}}{\rule[-1.8ex]{0pt}{0pt}}& $\frac{33}{2}$ & $27 \int_{\overline{\mathcal{M}}_{1,4}} \frac{\Omega^{[3]}_{1;0,2,2,2}}{1 - 2\psi_1} - 27 \int_{\overline{\mathcal{M}}_{1,3}} \frac{\Omega^{[3]}_{1;0,1,2}}{1 - 2\psi_1}$ \\
& $[q_1^3q_3]$ {\rule{0pt}{4ex}}{\rule[-1.8ex]{0pt}{0pt}}& $99$ & $81 \int_{\overline{\mathcal{M}}_{1,4}} \frac{\Omega^{[3]}_{1;0,1,1,1}}{1 - 2\psi_1}$ \\
& $[q_1q_2q_3]$ {\rule{0pt}{4ex}}{\rule[-1.8ex]{0pt}{0pt}}& $117$ & $162 \int_{\overline{\mathcal{M}}_{1,3}} \frac{\Omega_{1;0,1,2}^{[3]}}{1 - 2\psi_1}$ \\
& $[q_3^2]$ {\rule{0pt}{4ex}}{\rule[-1.8ex]{0pt}{0pt}}& $\frac{51}{4}$ & $54 \int_{\overline{\mathcal{M}}_{1,1}} \frac{\Omega^{[3]}_{1;0}}{1 - 2\psi_1}$ \\ \midrule
$DH_{1,1}(7)$ & \multirow{8}{*}{} $[q_1^7]$ {\rule{0pt}{4ex}}{\rule[-1.8ex]{0pt}{0pt}}& $\frac{117649}{2880}$ & $\frac{567}{160} \int_{\overline{\mathcal{M}}_{1,8}} \frac{\Omega^{[3]}_{1;2,1,1,1,1,1,1,1}}{1 - 7\psi_1/3}$ ~ ($\ast$) \\ 
& $[q_1^5q_2]$ {\rule{0pt}{4ex}}{\rule[-1.8ex]{0pt}{0pt}}& $\frac{117649}{360}$ & $\frac{3969}{80} \int_{\overline{\mathcal{M}}_{1,7}} \frac{\Omega^{[3]}_{1;2,1,1,1,1,1,2}}{1 - 7\psi_1/3}$ ~ ($\ast$) \\
& $[q_1^3q_2^2]$ {\rule{0pt}{4ex}}{\rule[-1.8ex]{0pt}{0pt}}& $\frac{84035}{144}$ & $\frac{1323}{8} \int_{\overline{\mathcal{M}}_{1,6}} \frac{\Omega^{[3]}_{1;2,1,1,1,2,2}}{1 - 7\psi_1/3} - \frac{441}{8} \int_{\overline{\mathcal{M}}_{1,5}} \frac{\Omega^{[3]}_{1;2,1,1,1,1}}{1 - 7\psi_1/3}$ \\
& $[q_1q_2^3]$ {\rule{0pt}{4ex}}{\rule[-1.8ex]{0pt}{0pt}}& $\frac{2401}{12}$ & $\frac{441}{4} \int_{\overline{\mathcal{M}}_{1,5}} \frac{\Omega^{[3]}_{1;2,1,2,2,2}}{1 - 7\psi_1/3} - \frac{441}{4} \int_{\overline{\mathcal{M}}_{1,4}} \frac{\Omega^{[3]}_{1;2,1,1,2,2}}{1 - 7\psi_1/3}$ \\
& $[q_1^4q_3]$ {\rule{0pt}{4ex}}{\rule[-1.8ex]{0pt}{0pt}}& $\frac{16807}{48}$ & $\frac{1323}{16} \int_{\overline{\mathcal{M}}_{1,5}} \frac{\Omega^{[3]}_{1;2,1,1,1,1}}{1 - 7\psi_1/3}$ \\
& $[q_1^2q_2q_3]$ {\rule{0pt}{4ex}}{\rule[-1.8ex]{0pt}{0pt}}& $\frac{16807}{24}$ & $\frac{1323}{4} \int_{\overline{\mathcal{M}}_{1,4}} \frac{\Omega^{[3]}_{1;2,1,1,2}}{1 - 7\psi_1/3}$ \\
& $[q_2^2q_3]$ {\rule{0pt}{4ex}}{\rule[-1.8ex]{0pt}{0pt}}& $\frac{343}{3}$ & $\frac{441}{4} \int_{\overline{\mathcal{M}}_{1,3}} \frac{\Omega^{[3]}_{1;2,2,2}}{1 - 7\psi_1/3} - \frac{147}{4} \int_{\overline{\mathcal{M}}_{1,2}} \frac{\Omega^{[3]}_{1;2,1}}{1 - 7\psi_1/3}$ \\
& $[q_1q_3^2]$ {\rule{0pt}{4ex}}{\rule[-1.8ex]{0pt}{0pt}}& $\frac{1029}{8}$ & $\frac{441}{2} \int_{\overline{\mathcal{M}}_{1,2}} \frac{\Omega^{[3]}_{1;2,1}}{1 - 7\psi_1/3}$ \\ \bottomrule
\end{tabular}
\end{center}

\section{Examples of relations for Chiodo integrals} \label{appC}

We illustrate the practical use of the vanishing of Theorem~\ref{thm:vanishC} with partitions $\eta$ of lengths $\ell = 2$ or $3$.

 \subsection{$\ell = 2$} 
 
If $\ell=2$, we only have the summands $k=1$ and $k=2$. For $k=1$, we must have $\rho^{(1)} = \eta$ (let us assume that $|\eta| < d$, otherwise the term $k=1$ vanishes) of length $\ell(\eta) = 2$, so the Pochhammer symbol gives a factor of $(d - |\eta|)/d$. The $k=2$ term has summands $\rho^{(1)} = \eta_1$ and $\rho^{(2)} = \eta_2$, and an additional term for $\rho^{(1)} = \eta_2$ and $\rho^{(2)} = \eta_1$ if $\eta_1 \neq \eta_2$. This possibility produces the same result and simplifies against the automorphisms of $\rho^{(1)}$ in the denominator of the $k=1$ term. We get an identity when $|\eta|$ is unstable, i.e. $|\mu| + |\eta| = dK$ with $K \leq \ell(\eta) - 1 = 1$. This forces $K = 1$, and we then obtain
\begin{equation}
\frac{d - \eta_1 - \eta_2}{d } \int_{\overline{\mathcal{M}}_{g,n+1}} \frac{\Omega^{[d]}_{g;-\overline{\mu},d - \eta_1 - \eta_2}}{\prod_{i = 1}^n \big(1 - \frac{\mu_i}{d}\psi_i\big)}
=
\int_{\overline{\mathcal{M}}_{g,n+2}} \frac{\Omega^{[d]}_{g;-\overline{\mu},d - \eta_1, d - \eta_2}}{\prod_{i = 1}^n \big(1 - \frac{\mu_i}{d}\psi_i\big)}.
\end{equation}
For instance, for $(g,n) = (1,1)$, $\mu = (3)$, $\eta= (2,1)$ and $d=6$, we have
\begin{equation}
\frac 1 2 \int_{\overline{\mathcal{M}}_{1,2}} \frac{\Omega^{[6]}_{1;3,3}}{\big(1 - \frac{\psi_1}{2} \big)}
=
\int_{\overline{\mathcal{M}}_{1,3}} \frac{\Omega^{[6]}_{1;3,4,5}}{\big(1 - \frac{\psi_1}{2} \big)}.
\end{equation}
The integrals can be computed with the SageMath package \textsc{admcycles} to be $1/36$ and $1/72$, respectively, numerically confirming the identity in this case.

\subsection{$\ell = 3$} 

 If $\ell = 3$, we only have the summands $k = 1,2,3$. The summand $k=1$ only occurs when $|\eta| < d$: in that case $\rho^{(1)} = \eta$ and it has length $3$, so the Pochhammer symbol yields a factor $\frac{d - |\eta|}{d}\cdot \frac{2d - |\eta|}{d}$. The $k=2$ term has summands for $\rho^{(1)}$ and $\rho^{(2)}$ of lengths one and two or vice versa, counting all distinct possibilities. The $k=3$ term has all three $\rho^{(\kappa)}$ of length one and hence, no Pochhammer symbols, permuted in all possible distinct ways. We get an identity when $\eta$ is unstable, that is $|\mu| + |\eta| = Kd$ with $K \leq \ell(\eta) - 1$, hence $K = 1$ or $2$. In these cases, the identity reads
\begin{multline} \label{the3rel}
\frac{1}{|\Aut{\eta}|} \!\int_{\overline{\mathcal{M}}_{g,n+3}} \frac{\Omega^{[d]}_{g;-\overline{\mu},d - \eta_1, d - \eta_2, d - \eta_3}}{\prod_{i = 1}^n \big(1 - \frac{\mu_i}{d}\psi_i\big)} +
\frac{\mathbf{1}_{|\eta| < d}}{|\Aut{\eta}|}\,\frac{d - |\eta|}{d }\,\frac{2d - |\eta|}{d}\! \int_{\overline{\mathcal{M}}_{g,n+1}} \frac{\Omega^{[d]}_{g;-\overline{\mu},d - |\eta|}}{\prod_{i = 1}^n \big(1 - \frac{\mu_i}{d}\psi_i\big)} \\
=
\sum_{j = 1}^3 \frac{\mathbf{1}_{|\eta| - \eta_j < d}}{|{\rm Aut}(\eta \setminus \{\eta_j\})|}\,\frac{d - |\eta| + \eta_j}{d}\,
\int_{\overline{\mathcal{M}}_{g,n+2}} \!\!\! \frac{\Omega^{[d]}_{g;-\overline{\mu},d - \eta_j, d - |\eta| + \eta_j}}{\prod_{i = 1}^n \big(1 - \frac{\mu_i}{d}\psi_i\big)}, 
\end{multline}
where $|{\rm Aut}(\eta\setminus\{\eta_j\})|$ is $2$ if $\eta \setminus \{\eta_j\}$ contains two parts of same size, and $1$ otherwise.

For instance, let us consider $(g,n) = (1,1)$, $\mu = (2)$, $\eta= (1,2,3)$ and $d=4$. Notice that we have $|\mu| + |\eta| = 2d$, and the $k = 1$ term as well as the $k = 2$, $j = 2,3$ terms vanish. So we obtain the relation
\begin{equation}
\int_{\overline{\mathcal{M}}_{1,4}} \frac{\Omega^{[4]}_{1;2,1, 2,3}}{ \big(1 - \frac{\psi_1}{2}\big)}
=
\frac{1}{4} \int_{\overline{\mathcal{M}}_{1,3}} \!\! \frac{\Omega^{[4]}_{1;2,1,1}}{ \big(1 - \frac{\psi_1}{2} \big)}.
\end{equation}
The integrals can be computed with the SageMath package \textsc{admcycles} to be $1/1536$ and $1/384$, respectively, numerically confirming the identity in this case.

Let us now look at a case without vanishing summands. For instance, let us consider $(g,n) = (1,1)$, $\mu = (2)$, $\eta= (1,2,2)$ and $d=7$, which gives $|\mu| + |\eta| = d$. So we obtain the relation
\begin{equation}
\frac{1}{2} \!\int_{\overline{\mathcal{M}}_{1,4}} \frac{\Omega^{[7]}_{1;5,5,5,6}}{ \big(1 - \frac{2}{7}\psi_1 \big)}
+
\frac{1}{2}\!\cdot\!\frac 2 7 \!\cdot\! \frac 9 7 \! \int_{\overline{\mathcal{M}}_{1,2}} \frac{\Omega^{[7]}_{1;5,2}}{ \big(1 - \frac{2}{7}\psi_1 \big)}
=
\frac 3 7 \! \cdot\! \frac 1 2 \! \int_{\overline{\mathcal{M}}_{1,3}} \frac{\Omega^{[7]}_{1;5,6,3}}{ \big(1 - \frac{2}{7}\psi_1 \big)}
+ 
\frac 4 7\! \int_{\overline{\mathcal{M}}_{1,3}} \!\! \frac{\Omega^{[7]}_{1;5,5,4}}{ \big(1 - \frac{2}{7}\psi_1 \big)}.
\end{equation}
The integrals can be computed with the SageMath package \textsc{admcycles} to be $1/2401$, $1/196$, $1/686$ and $1/686$, respectively, numerically confirming the identity in this case.

Finally, let us illustrate~\Cref{equate}. When $K = 1$, the right-hand side of \eqref{the3rel} is unstable, so we can use the $\ell = 2$ identity with the partition $(\eta_j,|\eta| - \eta_j)$ to find
\begin{multline} \label{eq:59}
\int_{\overline{\mathcal{M}}_{g,n + 3}} \frac{\Omega^{[d]}_{g;-\overline{\mu},d - \eta_1,d - \eta_2,d - \eta_3}}{\prod_{i = 1}^n \big(1 - \frac{\mu_i}{d}\psi_i\big)} \\
= \frac{d - |\eta|}{d}\bigg(\frac{|\eta| - 2d}{d} + \sum_{j = 1}^3 \frac{|{\rm Aut}(\eta)|(1 + \delta_{\eta_j,|\eta| - \eta_j})}{|{\rm Aut}(\eta\setminus\{\eta_j\})|} \frac{d - |\eta| + \eta_j}{d}\bigg) \int_{\overline{\mathcal{M}}_{g,n + 1}} \frac{\Omega^{[d]}_{g;-\overline{\mu},d - (\eta_1 + \eta_2 + \eta_3)}}{\prod_{i = 1}^n \big(1 - \frac{\mu_i}{d}\psi_i\big)}.
\end{multline}
When all automorphism factors on the right-hand side are $1$, the coefficient of the Chiodo integral is $\big(\tfrac{d - |\eta|}{d}\big)^2$, but in general it is more complicated.

\bibliographystyle{amsplain}
\bibliography{BibliHur2}

\end{document}